\documentclass[letterpaper,11pt,reqno]{amsart}

\usepackage[margin=1in]{geometry}
\usepackage[utf8]{inputenc}
\usepackage[foot]{amsaddr}
\usepackage{amssymb,amsthm}
\usepackage{enumerate}
\usepackage{upref}
\usepackage[colorinlistoftodos,textsize=small]{todonotes}
\usepackage{hyperref} 
\hypersetup{colorlinks = true, linkcolor = blue, citecolor = blue}

\newtheorem{theorem}{Theorem}[section]
\newtheorem{corollary}[theorem]{Corollary}
\newtheorem{lemma}[theorem]{Lemma}

\theoremstyle{definition}
\newtheorem{defn}{Definition}[section]
\newtheorem{assumption}{Assumption}

\theoremstyle{remark}
\newtheorem*{remark}{Remark}

\newcommand{\eps}{\varepsilon}
\newcommand{\veps}{\varepsilon}
\newcommand{\dist}{\mbox{dist}}

\newcommand{\la}{\lambda}
\newcommand{\lan}{\langle}
\newcommand{\ran}{\rangle}

\newcommand{\PP}{\mathbf{P}}
\newcommand{\EE}{\mathbf{E}}

\newcommand{\R}{\mathbb{R}}
\newcommand{\RR}{\mathbb{R}}
\newcommand{\Z}{\mathbb{Z}}
\newcommand{\N}{\mathbb{N}}
\newcommand{\NN}{\mathbb{N}}

\newcommand{\clc}{\mathcal{C}}
\newcommand{\cle}{\mathcal{E}}
\newcommand{\clf}{\mathcal{F}}
\newcommand{\clk}{\mathcal{K}}
\newcommand{\cln}{\mathcal{N}}
\newcommand{\clp}{\mathcal{P}}
\newcommand{\clr}{\mathcal{R}}
\newcommand{\clv}{\mathcal{V}}
\newcommand{\clq}{\mathcal{Q}}

\newcommand{\ap}{\textsc{ap}}

\title[Convergence of QSD in Unbounded Domains]{Asymptotics of Quasi-Stationary Distributions of Small Noise Stochastic Dynamical Systems in Unbounded Domains}
\date{November 2019}

\author{Amarjit Budhiraja}
\author{Nicolas Fraiman}
\author{Adam Waterbury}
\address{University of North Carolina at Chapel Hill}



\begin{document}

\begin{abstract}
We consider a collection of Markov chains that model the evolution of multitype biological populations. The state space of the chains is the positive orthant, and the boundary of the orthant is the absorbing state for the Markov chain and represents the extinction states of different population types. We are interested in the long-term behavior of the Markov chain away from extinction, under a small noise scaling. Under this scaling, the trajectory of the Markov process over any compact interval converges in distribution to the solution of an ordinary differential equation (ODE) evolving in the positive orthant. We study the asymptotic behavior of the quasi-stationary distributions (QSD) in this scaling regime. Our main result shows that, under conditions, the limit points of the QSD are supported on the union of interior attractors of the flow determined by the ODE. We also give lower bounds on expected extinction times which scale exponentially with the system size. Results of this type when the  deterministic dynamical system obtained under the scaling limit is given by a discrete time evolution equation and the dynamics are essentially in a compact space (namely, the one step map is a bounded function)  have been studied by Faure and Schreiber (2014). Our results extend these to a setting of an unbounded state space and continuous time dynamics. The proofs rely on uniform large deviation results for small noise stochastic dynamical systems and methods from the theory of continuous time dynamical systems.

In general QSD for Markov chains with absorbing states and unbounded state spaces may not exist. We study one basic family of Binomial-Poisson models in the positive orthant where one can use Lyapunov function methods to establish existence of QSD and also to argue the tightness of the QSD of the scaled sequence of Markov chains. The results from the first part are then used to characterize the support of limit points of this sequence of QSD.

\end{abstract}

\maketitle

\section{Introduction}
In this work we study discrete time Markov chains with values in the $d$-dimensional positive orthant that are absorbed upon hitting the boundary of the orthant. 
Such processes are well suited to model biological and ecological systems \cite{gylsil,hog} where each coordinate represents the population size of individuals of a given type/species.
One of the  fundamental issues in mathematical biology is to characterize the conditions for a population of interacting species to coexist, that is, to survive for a long time with no extinctions. Many real-world systems are certain to go extinct eventually, yet appear to be stationary over any reasonable time scale. Generally, the finite nature of the resources available prevents the system from growing without limit. Thus, provided we wait long enough, a sufficiently strong downward fluctuation in population size is bound to occur.  We are interested in studying the long-term behavior of such systems away from extinction, under a suitable scaling of the system.

The processes we consider have a natural scaling parameter ($N$) representing the  system size. From standard
results, as $N\to \infty$, the linearly interpolated trajectory of the  state process $X^N$, over any compact time interval $[0,T]$, converges in  distribution in $C([0,T]:\RR_+^d)$ (the space of continuous functions from $[0,T]$ to $\RR_+^d$, equipped with 
the  uniform topology) to the solution of an ordinary differential equation (ODE) of the form $\dot{\varphi}(t) = G(\varphi(t))$, $\varphi(0) = x$ (see \eqref{eq:dynsys}).
Our goal is to analyze the limiting behavior of the steady states of $X^N$, conditioned on non-extinction, as $N\to \infty$, in terms of the properties of the flow determined by the above ODE.
The steady state of a Markov chain conditioned on non-extinction is made precise through the notion of a quasi-stationary distribution (QSD) (see Definition \ref{def:qsd}). We refer the reader to \cite{melvil} for a comprehensive background and survey of results in the theory of quasi-stationary distributions. QSD are important objects in biological models and discussions of applications in biology can be found in \cite{pol,pol2,bucpol,gos,gos2}.  

Our first main result (Theorem \ref{thm:main}) studies asymptotics of QSD of $X^N$(denoted as $\mu_N$), as $N\to \infty$, provided they exist and the sequence $\{\mu_N\}$ is tight. Specifically, in Theorem \ref{thm:main} we show that, under Assumptions \ref{assu:LLN}, \ref{assu:ap-classesfinite}, \ref{assu:mgf} and \ref{assu:irrbdr}, any limit point $\mu$ of the sequence of QSD $\{\mu_N\}$ is invariant under the flow determined by the ODE \eqref{eq:dynsys} and is supported on the union of interior attractors of the flow. We also provide lower bounds on the probability of non-extinction over a fixed time horizon that scale exponentially in system size. These bounds readily give similar lower bounds on expected time to extinction.

In general Markov chains with absorbing states and an unbounded state space may fail to have a QSD. Conditions for existence of QSD have been studied in \cite{ferkesmarpic,van,vanpol}; however these results are not easily applicable to the models considered in this work. We instead make use of the recent work of Champagnat and Villemonais \cite{chavil} that gives general and broadly applicable Lyapunov function-based Foster type criterion for existence of QSD (see Theorem \ref{thm:foster}). In our second main result we consider a basic family of Markov chains that we refer to as Binomial-Poisson models where the results of \cite{chavil} can be applied to give existence of QSD. Using the stability properties of these Markov chains we obtain bounds on exponential moments of certain hitting times that allow us to construct suitable Lyapunov functions (and related objects) for which the conditions in Theorem \ref{thm:foster} are satisfied, thus establishing the existence of a QSD $\mu_N$ for each  $N$. In fact, this QSD can be characterized as the limit, as $n\to \infty$, of the law of $X^N_n$, conditioned on non-extinction, starting from an arbitrary initial condition in the interior. Using this characterization, and similar moment estimates as used in the construction of the Lyapunov functions, we then argue that the sequence of QSD is tight. Finally, from these results and other properties of the model, we establish our second main result (Theorem \ref{thm:poisbin}), which says that the Binomial-Poisson model introduced in Section \ref{sec:mainresults}  satisfies all the conditions in Theorem \ref{thm:main} and therefore provides an important class of Markov chains where the conclusions of Theorem 
\ref{thm:main} hold.

\subsection{Approach and Proof Idea}
We now comment on the proof of Theorem \ref{thm:main}.  Our results are motivated by the work of 
Faure and Schrieber \cite{fausch} (see also the unpublished manuscript of Marmet \cite{marmet}) 
which considers analogous problems for a class of Markov chains where the  deterministic dynamical system obtained under the scaling limit is given by a discrete time evolution equation and the dynamics are essentially in a compact space (namely, the one step map is a bounded function). As in \cite{fausch}, one of the important ingredients in the proof is an analysis of the large deviation behavior of the sequence of small noise Markov chains in Section \ref{sec:themodel}. However due to the continuous time setting here one needs to study large deviation principles on suitable path spaces. One of the issues that arises in the large deviation analysis is that transition probabilities of the Markov chain behave in a degenerate manner near the boundaries. Due to this, the associated local rate functions have poor regularity properties, which in turn makes establishing a global large deviation principle (LDP) on the path space technically challenging. Another issue arises from the unboundedness of the state space. In particular, the moment generating functions of the noise sequences can become arbitrarily large as the system state becomes large. In order to handle these issues, we instead consider LDP for a collection of modified chains in $\RR^d$. These modified chains behave identically to the original chain until  exiting from a given compact set $K$ in the interior of the orthant, and, upon exiting, the modified chains change their behavior to a more regular dynamics in an appropriate sense. The large deviation estimates that are needed for our analysis can be obtained by piecing together such LDP associated with all such compact sets $K$. A similar approach, in a setting where the state space is compact, has been proposed in \cite{marmet}. Another important point in the analysis is that one needs large deviation estimates that are uniform in initial condition in compact sets, in the sense of Freidlin and Wentzell \cite[Chapter 3.3, pages 91-92]{frewen}. For this we use results on uniform Laplace principles for small noise stochastic difference equations that have been developed in \cite[Section 6.7]{dupell}.
The recent work \cite{salbuddup} shows that a uniform Laplace principle implies a uniform Large deviation principle in the sense of Freidlin and Wentzell. These results together allow  us to establish uniform probability estimates that are needed in our large deviation analysis (see Section \ref{sec:ldest}).

The proof of Theorem \ref{thm:main}, analogous to  \cite{fausch}, also requires a detailed analysis of the 
dynamical system properties of the flow associated with the ODE \eqref{eq:dynsys}. In particular a careful understanding of the properties of continuous time analogs of absorption preserving pseudo-orbits (in the terminology of \cite{fausch}) and those of the associated recurrence classes are key to the proof (see Section \ref{sec:ap}). Although some of the arguments are similar to \cite{fausch} there are new challenges that arise due to the unboundedness of the state space and the continuous time dynamics. To handle these features we exploit the stability properties of the underlying ODE and develop several a priori estimates for pseudo-orbits that are uniform in time and/or space.  The dynamical systems results in Section \ref{sec:ap} and the large deviation estimates in Section \ref{sec:ldest} take us most of the way to the proof of Theorem \ref{thm:main}. In particular in Section \ref{sec:behnb}, using these results, we establish the lower bound on probabilities of non-extinction given in Theorem 
\ref{thm:main} and also that the limit points $\mu$ of the QSD are invariant under the flow, they do not charge the boundary, and in fact that they are supported on the union of absorption preserving recurrence classes in the interior. The final step is to show that the support in fact lies in the union of the interior attractors. For this, following \cite{fausch}, we reformulate the notion of recurrence in terms of the quasipotential associated with the rate functions in the underlying large deviation principles.  Section \ref{sec:vrec} introduces the quasipotential and this alternative notion of recurrence and proves the equivalence between these two definitions of recurrence classes. The second definition is more well suited for the analysis and allows the use of large deviation estimates of Section \ref{sec:ldest} in studying the behavior of the stochastic dynamical system in terms of the properties of the recurrence classes. Combining the results of Section \ref{sec:vrec} with the results of Section
\ref{sec:ldest} and properties of absorption preserving pseudo-orbits studied in Section \ref{sec:ap}, the proof of the main result is completed in Section \ref{sec:pfofmain}.

\subsection{Organization}
The paper is organized as follows. 
In Section \ref{sec:res} we introduce the model of interest, state the assumptions and present the main results of the paper. 
In Section \ref{sec:ap} we introduce some notions from the theory of dynamical systems, and study properties of recurrence points and associated (pseudo) orbits for the dynamical system associated with the law of large numbers limit of the underlying sequence of scaled Markov chains.
In Section \ref{sec:ldest} we establish some key large deviation estimates.
In Section \ref{sec:behnb} we give some important asymptotic properties of  QSD (provided they exist) for the Markov chains considered in this work.
In Section \ref{sec:vrec} we introduce the quasipotential $V$ that governs the large deviation behavior of the model and study the properties of $V$-chain recurrence.
In Section \ref{sec:pfofmain} we complete the proof of our first  main theorem, namely Theorem \ref{thm:main}. Finally, Section \ref{sec:qsd} proves the second main result of this work, Theorem \ref{thm:poisbin}, which gives an important family of models for which Theorem \ref{thm:main} can be applied.

\subsection{Notation}
\label{sec:notat}
Let $\Delta \doteq \R_+^d$, 
 $\Delta^o \doteq \{x\in \Delta: x>0\}$, where inequalities for vectors are interpreted componentwise, and $\partial \Delta \doteq \Delta \setminus \Delta^o$. 
 Let for $N\in \NN$,
 $\Delta_N \doteq \Delta \cap \frac{1}{N}\Z^d$, $\partial \Delta_N \doteq \partial \Delta \cap \frac{1}{N}\Z^d$, and 
 $\Delta_N \doteq \Delta^o \cap \frac{1}{N}\Z^d$. 
 For $x,y \in \RR^d$, $\lan x, y\ran \doteq \sum_{i=1}^d x_i y_i$.
For $x \in \RR^d$ and $A\subset \RR^d$, $\dist(x,A)\doteq \inf_{y\in A}\|x-y\|$. 
We denote by $\cln^{\eps}(A)$ the $\eps$-neighborhood of a  set $A$ in $\Delta$, namely
$\cln^{\eps}(A) \doteq \{x \in \Delta: \dist(x,A)<\eps\}$.
For $r>0$ and $x \in \R^d$, $B_r(x)$ will denote the open ball of radius $r$ centered at $x$. Denote by $\clp(S)$ the space of probability measures on a Polish space $S$, equipped with the topology of weak convergence. For a $\mu \in \clp(S)$ and $\mu$-integrable $f: S\to \RR$, we write $\int f d\mu$ as $\mu(f)$. The support of $\mu \in \clp(S)$ will be denoted as $\mbox{supp}(\mu)$.
For a signed measure $\eta$ on $S$, $\|\eta\|_{TV}$ denotes its total variation norm, namely
$$\|\eta\|_{TV} = \sup_{f}\left |\int f d\eta\right|,$$
where the supremum is taken over all measurable maps $f: S \to \R$ such that $\sup_{x \in S}|f(x)|\le 1$.
For a bounded $F: S \to \RR$, we denote $\sup_{x\in S} |F(x)|$ by $\|F\|_{\infty}$.
We denote by $\clk$ the collection of all convex compact subsets  with a nonempty interior that are contained in $\Delta^o$.
For $T<\infty$, we denote by $C([0,T]:S)$  the space of continuous functions from $[0,T]$ to $S$, equipped with 
the  uniform topology.
For $\phi \in C([0,T]:\R^d)$, let $\|\phi\|_{*,T}\doteq \sup_{0\le t \le T}\|\phi(t)\|$.
Given a metric space $S_1$ and a Polish space $S_2$, a stochastic kernel $x\mapsto \theta(dy | x)$ on $S_2$ given $S_1$ is a measurable map from $S_1$ to $\clp(S_2)$.

\section{Statement of results}\label{sec:res}

\subsection{The model}
\label{sec:themodel}

Consider the sequence $\{X_k^N\}_{k\in \N_0}$ of $\Delta_N$-valued random variables defined as
\begin{equation}\label{eq:bmc}
\begin{aligned}
	X_{k+1}^N &= X_k^N + \frac{1}{N} \eta^N_{k+1}(X_k^N), \; k \in \N_0,\\
	X_0^N &= x^N
\end{aligned}
\end{equation}
where for each $x \in \Delta_N$, $\eta^N_k(x)$ is a $\Z^d$-valued random variable with distribution $\theta^N(\cdot | x)$ 
such that $\mbox{supp}(\theta^N(\cdot | x)) \subset \prod_{i=1}^d [-Nx_i, \infty)$.

 We will denote by $\PP^N_{\nu}$ the probability measure under which the Markov chain $\{X_k^N\}$ has the initial distribution $\nu$, namely $\PP^N_{\nu}(X_0^N \in A) = \nu(A)$. If $\nu = \delta_x$, we write $\PP^N_{\nu}$ as simply $\PP^N_x$.
\begin{defn}
	\label{def:qsd}
	A probability measure $\mu_N$ on $\Delta_N^o$ is said to be a \emph{quasi-stationary distribution} (QSD) for the Markov chain $\{X_k^N\}$ if for every $n \in \N$
	$$\PP_{\mu_N}[X_k^N =j \mid X_k^N \in \Delta_N^o] = \mu_N(j), \; \mbox{ for all } j \in \Delta_N^o \mbox{ and } k \in \N.$$
\end{defn}

\subsection{Definitions and Assumptions}
\label{sec:defnassu}
Consider the continuous time process $\hat X^N$ obtained from a linear interpolation of $X^N$, given as
\begin{align}\label{eqn:continterpolation}
	\hat{X}^N(t) = X^N_n + [X^N_{n+1} - X^N_{n}](Nt - n), \; t \in [n/N, (n+1)/N], \; n \in \N_0,
\end{align}
The following assumption on the law of large numbers behavior of  $\hat X^N$ will play a central role in our study of 
asymptotic properties of   QSD of $X^N$.
\begin{assumption}
	\label{assu:LLN}
	There is a Lipschitz function $G: \Delta \to \RR^d$ such that for any sequence $x_N \to x$, with $x_N \in \Delta_N$ for every $N \in \NN$,
	\begin{equation}
		\PP_{x_N}\left(\sup_{0\le t \le T}\|\hat X^N(t) - \varphi_t(x)\| > \eps\right) \to 0, \mbox{ as } N \to \infty, \mbox{ for every } T\in [0,\infty) \mbox{ and } \eps >0 \label{eq:llnps}
		\end{equation}
	where $\{\varphi_t(x)\}_{t\ge 0}$ is the solution of the ODE
\begin{equation}\label{eq:dynsys}
	\dot{\varphi}(t) = G(\varphi(t)), \; \varphi(0) = x.
\end{equation}

\end{assumption}

We now  introduce the notion of absorption preserving pseudo-orbits for the flow associated with the ODE \eqref{eq:dynsys}. Discrete time analogs of these were introduced in \cite{fausch}.

\begin{defn}
Given $\delta, T > 0$, consider a family of points $\xi = (\xi_0 = x,\dots, \xi_n = y) \in \Delta^{n+1}$ and a collection of times $T \leq T_1, \dots,  T_{n-1}$ such that
\begin{itemize}
\item $\|\xi_0 - \xi_1\| < \delta$
\item whenever $\xi_i \in \partial \Delta$, $\xi_{i+1} \in \partial \Delta$
\item $\|\xi_{i+1} - \varphi_{T_i}(\xi_i)\| < \delta$ for $1 \leq i \leq n-1$.
\end{itemize}
The piecewise continuous path
\[
\left(x, \{\varphi_t(\xi_1) : t \in [0,T_1]\}, \{\varphi_t(\xi_2) : t \in[0,T_2]\}, \dots, \{\varphi_t(\xi_{n-1}) : t \in [0,T_{n-1}]\}, y\right).
\]
is said to be a $(\delta,T)$ absorption preserving pseudo-orbit (\ap--pseudo-orbit) from $x$ to $y$. Occasionally, we will also refer to the sequence $\{\xi_i\}_{i=0}^n$ as a $(\delta, T)$ \ap--pseudo-orbit from $x$ to $y$.
\end{defn}

\begin{defn}
For two points $x,y \in \Delta$, say that $x <_\ap y$ if for all $\delta, T > 0$ 
there is a $(\delta,T)$ \ap--pseudo-orbit from $x$ to $y$. If $x <_\ap y$ 
and $y <_\ap x$, we write $x \sim_\ap y$. If $x \sim_\ap x$,
 then $x$ is said to be an \ap--chain recurrent point. Let $\clr_\ap$ denote the set of of
  \ap--chain recurrent points, and note that $\sim_\ap$ is an equivalence relation 
  on $\clr_\ap$. For $x \in \clr_\ap$, the equivalence class $[x]_\ap$ of all $y \in \clr_{\ap}$ such that $y \sim_{\ap} x$ is said to
   be \ap--basic class. Such a class  is called maximal if, whenever for some $y\in \clr_{\ap}$, $x <_\ap y$, we have $y \in [x]_{\ap}$.
  A maximal \ap--basic class is called an \ap--quasiattractor. We let $\clr_{\ap}^* \doteq \clr_{\ap}\cap \Delta^o$.
\end{defn}

\label{sec:hypo}
The following will be  our main assumptions on the dynamical system $\{\varphi_t(x)\}$.
Parts (c) and (d) say that the velocity fields decay as the boundaries are approached but not at too fast a rate.
Part (e) is our main stability assumption on the dynamics. Parts (a), (b) are requirements on recurrence classes for the flow that are satisfied quite broadly.
\begin{assumption}\label{assu:ap-classesfinite}
\begin{enumerate}[(a)]
		\item There are a a finite  number of \ap--basic classes contained in $\Delta^o$, which are denoted by $\{K_i\}_{i=1}^{v}$. Each 
		$K_i$ is a  closed set. Additionally, for some $l < v$, $\{K_i\}_{i=1}^l$ are \ap-quasiattractors and $\{K_i\}_{i=l+1}^v$ 
	are non \ap-quasiattractors.
	\item For each $i = 1, \ldots, v$ there is a $x_i\in K_i$ such that, for every $T>0$, $\{\varphi_t(x_i): t\ge T\}$ is dense in $K_i$. 
	 \item There exists $\eps>0$ and $m>0$ such that for every
 $i= 1, \ldots, d$, $G_i(x)> mx_i$  whenever  $x \in \Delta^o$ and $x_i \le \eps$.
 \item For every $i=1, \ldots, d$, as $\delta \to 0$, $\sup\limits_{x \in \Delta: x_i \le \delta}G_i(x) \to 0$.
 \item For some $\kappa \in (0,\infty)$ and $M \in (1,\infty)$, $\lan x, G(x)\ran \le -\kappa \|x\|^2$ for all $x \in \Delta$ with $\|x\|\ge M$.
 \end{enumerate}
\end{assumption}

We will need certain assumptions on the moment generating functions of $\theta^N(\cdot | x)$. 
\begin{assumption}\label{assu:mgf}
	The following hold:
	\begin{enumerate}[(a)]
		\item For every $N\in \NN$, $\zeta \in \RR^d$,   and $x \in  \Delta_N^o$
		$$H^N(x,\zeta) \doteq \log \int_{\RR^d} \exp\{\lan \zeta, y\ran\} \theta^N(dy | x) <\infty.$$
		\item There exists a stochastic kernel $\theta(dy | x)$ on $\RR^d$ given $\Delta^o$ such that 
		\begin{enumerate}[(i)]
			\item For every $x\in \Delta^o$, the convex hull of $\mbox{supp}(\theta(\cdot | x)) = \RR^d$.
			\item The map $x\mapsto \theta(\cdot | x)$ is a continuous map from $\Delta^o$ to $\clp(\RR^d)$.
			\item For every $\zeta \in \RR^d$ and $K \in \clk$, $\sup_{x\in K} H(x,\zeta) <\infty$,
			where
			$$H(x,\zeta) \doteq \log \int_{\RR^d} \exp\{\lan \zeta, y\ran\} \theta(dy | x).$$
			Furthermore,
			as $N\to \infty$,
			$$\sup_{x\in K\cap \Delta_N} |H^N(x,\zeta) - H(x,\zeta)| \to 0.$$
			
			\end{enumerate}
	\end{enumerate}
	
\end{assumption}

We introduce one final assumption to provide a lower bound on the probability that $X^N$ is absorbed when its initial state is sufficiently close to $\partial \Delta$.

\begin{assumption}
	\label{assu:irrbdr}
	\begin{enumerate}[(a)]
		\item For each $N\in \NN$ and $x,y \in \Delta_N^o$, there is a $k \in \NN$ such that $\PP_y^N(X^N_k=x)>0$.
		\item  For every $\gamma \in (0,\infty)$ and $T\in \N$, there is an open neighborhood $U_{\gamma}$ of $\partial \Delta$ in $\Delta$ such that
	$$\liminf_{N\to \infty} \inf_{x \in U_{\gamma}\cap \Delta_N} \frac{1}{N} \log \PP_x(\hat X^{N}(T) \in \partial \Delta) \ge -\gamma.$$
	\end{enumerate}
	
\end{assumption}

We now present our main results.

\subsection{Main results}
 \label{sec:mainresults}



It is easy to see that under Assumption \ref{assu:ap-classesfinite}, for all $x\in \Delta$ and $t\ge 0$, $\varphi_t(x) \in \Delta$. In particular $\varphi_t$ is a measurable map from $\Delta$ to itself for every $t\ge 0$. We recall the definition of an invariant measure for the flow $\{\varphi_t\}$.
\begin{defn}
A probability measure $\mu$ on $\Delta$ is \emph{$\{\varphi_t\}$-invariant} if $\mu(\varphi_t^{-1}(A)) = \mu(A)$ for every measurable $A \subseteq \Delta$ and $t > 0$.
\end{defn}

\begin{theorem}\label{thm:main}
	Suppose that for every $N\in \N$, there exists a quasi-stationary distribution $\mu_N$ for $\{X^N_n\}_{n\in \NN_0}$ and that  the sequence $\{\mu_N\}$ is relatively compact
	as a sequence of probability measures on $\Delta^o$.  Suppose that Assumptions \ref{assu:LLN}, \ref{assu:ap-classesfinite}, \ref{assu:mgf} and \ref{assu:irrbdr} are satisfied.
	Then any weak limit point $\mu$ of this sequence  is $\{\varphi_t\}$-invariant and is supported on $\cup_{i=1}^l K_i$.
	Moreover, letting 
	\begin{equation}\label{eq:lan}
		\lambda_N \doteq [\PP_{\mu_N}(X_1^N \in \Delta^o)]^N,\end{equation} 
	there is a $c>0$ and $N_0 \in \N$ such that $\lambda_N \ge 1 - e^{-cN}$ for all $N \ge N_0$.
\end{theorem}

We now introduce a basic family of Markov chains which we refer to as the {\em Binomial-Poisson models} for which Theorem \ref{thm:main} can be applied.

Consider a population with $d$ types of particles evolving in discrete time in which at each time step, any given particle dies with probability $1/N$, and given that the population size  at previous time step was $Nx = (Nx_i)_{i=1}^d$, the number of particles of type $i$ that are produced at the next time step follows a
Poisson distribution with mean $F_i(x)$  distribution for some
 $F:\Delta \to \RR_+^d$. 
 Denoting the total number of particles of type $i$ at time $k$ as $NX^{N,i}_k$, the evolution of $X^N_k = (X^{N,1}_k, \ldots, X^{N,d}_k)$ is then given by \eqref{eq:bmc} where, for each $N$, $\theta^N(dy|x) \equiv \theta^{N,*}(dy|x)$ is the distribution of $U-V$ where $U = (U_i)_{i=1}^d$, $V = (V_i)_{i=1}^d$, $\{U_i, V_j, i,j=1, \ldots, d\}$ are mutually independent,  $U_i \sim \mbox{Poi}(F_i(x))$ (namely, a Poisson random variable with mean $F_i(x)$)
and $V_i \sim \mbox{Bin}(Nx_i, \frac{1}{N})$ (namely a Binomial random variable with $Nx_i$ trials and probability of success $1/N$).

Define, 
\begin{equation}
	\tau^N_{\partial} \doteq \inf\{k \in \N_0: X_k^N \in \partial \Delta_N\}.
	\label{eq:tnpa}
	\end{equation}
For a bounded and measurable $f : \Delta_N \rightarrow \R$,
\begin{equation}\label{eq:pnnf}
P_n^N f (x) \doteq \EE_x[ f({X}^N_n);\; \tau_{\partial}^N > n].
\end{equation}
\begin{theorem}\label{thm:poisbin}
	Suppose that, for each $N$, $X^N$ is given by \eqref{eq:bmc} with $\theta^N \equiv \theta^{N,*}$.
	Further suppose that $F$ is a bounded Lipschitz map and Assumption \ref{assu:ap-classesfinite}(a)-(d) are satisfied with $G(x)=F(x)-x$.
	Then, there is a $\mu_N \in \clp(\Delta^o_N)$ such that for every $N \in \N$, and $x_N \in \Delta^o_N$, 
$$ \frac{\delta_{x_N} P_n^N}{\delta_{x_N} P_n^N(1_{\Delta_N^o})}$$ converges to $\mu_N$ in the total variation distance as $n \rightarrow \infty$.
The measure $\mu_N$ is a QSD for $\{X^N\}$.
The sequence $\{\mu_N\}_{N\in \N}$ is relatively compact as a sequence of probability measures on $\Delta$, and any weak limit point $\mu$ of this sequence  is $\{\varphi_t\}$-invariant and is supported by $\cup_{i=1}^l K_i$.
Finally, letting $\lambda_N \doteq [\PP_{\mu_N}(X_1^N \in \Delta^o)]^N$, 
	there is a $c>0$ and $N_0 \in \N$ such that $\lambda_N \ge 1 - e^{-cN}$ for all $N \ge N_0$.
\end{theorem}

Theorem \ref{thm:main} is  proved in Section \ref{sec:pfofmain} while Theorem \ref{thm:poisbin} is established in Section \ref{thm:absorption}.


\section{Absorption preserving pseudo-orbits}\label{sec:ap}
In this section we present some basic facts for absorption preserving  pseudo-orbits that will be used to prove Theorem \ref{thm:main}.
Throughout the section we will take Assumptions \ref{assu:LLN} and \ref{assu:ap-classesfinite} to hold.

The proofs of many of these results are similar to those found in \cite{fausch} for discrete time flows but we provide the details for completeness.
Recall that the solution of the  ODE \eqref{eq:dynsys} with initial value $\varphi(0)=x$ is denoted as $\{\varphi_t(x)\}_{t\ge 0}$. The following lemma is a consequence of the stability condition in Assumption \ref{assu:ap-classesfinite}(e).
\begin{lemma}\label{lem:stabilitylemma}
	For every $T>0$ and compact $A \subset \Delta$, there is a $\delta_0>0$ and a  compact $A_1 \subset \Delta$ such that for 
	any $(\delta_0, T)$ \ap--pseudo-orbit $\{\xi_i\}_{i=0}^{n+1}$ with $\xi_0 \in A$, we have $\xi_i \in A_1$ for all $i = 0, \ldots, n+1$.	
\end{lemma}
\begin{proof}
	For fixed, $x \in \Delta$, $\|\varphi_t(x)\|^2$ solves the 
	 ODE
	$$\frac{d}{dt} \|\varphi_t(x)\|^2 =  2 \lan G(\varphi_t(x)),  \varphi_t(x)\ran.$$
	From Assumption \ref{assu:ap-classesfinite} (e), when $\|x\| \ge M$
	\begin{align*}
		2 \lan G(x), x\ran 
		\le - 2\kappa\|x\|^2.
\end{align*}
This implies the following two facts:
\begin{enumerate}[(a)]
	\item If for any $R\ge M$, $x \in B_{R}\doteq \{z: \|z\| \le R\}$ then $\varphi_t(x) \in B_{R}$ for every $t\ge 0$.
	\item Given $T>0$, define $\delta_0= \delta_0(T) \doteq \frac{\kappa T}{2}\wedge 1$. Then  for any $\delta \le \delta_0$, and any $R\ge M$, 
	whenever $x \in B_{R+\delta}$, we have that $\varphi_t(x) \in B_R$ for all $t \ge T$.
\end{enumerate}
Now fix $T>0$ and a compact $A \subset \R_+^d$. Without loss of generality assume that there is a $R\ge M$ such that
$A \subset B_R$.  
Let $\delta_0 = \delta_0(T)$ be as defined above and consider a $(\delta_0,T)$ \ap--pseudo-orbit $\{\xi_i\}_{i=0}^{n+1}$ with $\xi_0 \in A$.  Then the above two facts imply that $\xi_i \in B_{R+1}$ for all $i= 0, 1, \ldots, n+1$.
The result follows on taking $A_1 = B_{R+1}$.
\end{proof}
As a consequence of Lemma \ref{lem:stabilitylemma} we get the following result on the boundedness of
\ap--basic classes.
\begin{lemma}\label{lem:apbasicbounded}
The \ap--basic classes are bounded.
\end{lemma}

\begin{proof}
Fix $x \in \mathcal{R}_\ap$ and $y \in [x]_\ap$.  Let $T> 0$ and $A= \{x\}$.
From Lemma \ref{lem:stabilitylemma}, there is a 
 $\delta_0>0$  and  a  compact $A_1$ in $\Delta$ such that for each $\delta \leq \delta_0$, any  $(\delta,T)$ \ap--pseudo-orbit starting at $x$  is contained in $A_1$. Since $y \in [x]_\ap$, there must exist a $(\delta, T)$ \ap--pseudo-orbit from $x$ to $y$
 which says that $y$ must lie in $A_1$.  The result follows. 
\end{proof}
 For $x \in \Delta^o$, we denote the forward orbit of $\varphi$ by
$
\gamma^+(x) \doteq \{ \varphi_t(x) | t \geq 0\}.
$
From Assumption \ref{assu:ap-classesfinite}(b)
and arguments as in Lemma \ref{lem:stabilitylemma} the following result is immediate.
\begin{lemma}
	\label{lem:bdryrepel}
	The following hold:
	\begin{enumerate}[(a)]
	\item There exists $\alpha_0 \in (0,1)$ such that if for some $\alpha \in (0,\alpha_0]$ and  $x \in \Delta^o$, $\dist(x, \partial \Delta) \ge \alpha$, then for all $t\ge 0$,
	$\dist(\varphi_t(x), \partial \Delta) > \alpha$. 
	\item There exists $M_0 \in (0,\infty)$ such that if for some $M \ge M_0$ and $x \in \Delta^o$, $\|x\|\le M$, then for all $t\ge 0$, $\|\varphi_t(x)\|< M$. 
	\item For every $A \in \clk$, there exist $T>0$, $A_1, A_2 \in \clk$ such that $A_1 \supset A$, $A_2 \subset A_1$, $\dist(A_2, \partial A_1) >0$, and for all $x \in A_1$ and $t\ge T$, $\varphi_t(x) \in A_2$.
	\item For every  $A_0 \in \clk$, there is an  $A_1\in \clk$ such that for every $x \in A_0$, the forward orbit
	$\gamma^+(x) \subset A_1$.
	\end{enumerate}
\end{lemma}

 Proof of the following lemma follows from the observation (a) in the proof of Lemma \ref{lem:stabilitylemma}.
\begin{lemma}\label{lem:compactbound}
For each compact  $K \subset \Delta$, 
$
\sup\limits_{x \in K}\sup\limits_{t\geq0}\|\varphi_t(x)\| < \infty.
$
\end{lemma}

We say a $(\delta,T)$ \ap--pseudo-orbit described by a collection of points $\xi = (\xi_0,\dots, \xi_n) \in \Delta^{n+1}$ and a collection of times $T \leq T_1, \dots,  T_{n-1}$ intersects a set $A \subset \Delta$, if for some $j \in  \{1, \ldots, n-1\}$,
and $t \in [0, T_j]$, $\varphi_t(\xi_j) \in A$. We say such an orbit lies in $A$ if its intersection with $A^c$ is empty.
The following lemma shows that for  small 
$\delta$ and large $T$, $(\delta,T)$ \ap--pseudo-orbits starting from the interior stay away from the boundary.
\begin{lemma}\label{lem:appobd}
Suppose $A\in \clk$. Then there exist $\eps_0>0$, $T>0$, $\delta>0$ such that any 
 $(\delta, T)$ \ap--pseudo-orbit, $\{\xi_k\}_{k=0}^n$ with $\xi_0\in A$ does not intersect
$E_{\eps_0}\doteq \{x\in \Delta: x_i \le \eps_0 \mbox{ for some } i= 1, \ldots, d\}$.
In particular, there is an $A_1\in \clk$ such that any such  \ap--pseudo-orbit starting in $A$ lies in $A_1$.
\end{lemma}
\begin{proof}
    Let $\eps_1 \doteq \dist(A, \partial \Delta)$ and let $\eps$ and $m$ be as in Assumption \ref{assu:ap-classesfinite}(c). Let $\eps_0 \doteq (\eps\wedge \eps_1)/4$. Note that for any $x \in \Delta$
    and $i=1, \ldots, d$,
    \begin{equation}
        \label{eq:lowbdode}
        \frac{d}{dt}([\varphi_t(x)]_i)^2 = 2[\varphi_t(x)]_i G_i(\varphi_t(x)) > 2m([\varphi_t(x)]_i)^2 \mbox{ whenever } [\varphi_t(x)]_i\le \eps.
    \end{equation}
    Since $m>0$, we can choose a $T>0$ such that for any $x\in \Delta$ and $i=1, \ldots , d$ with $x_i\ge \eps_0$, we have $[\varphi_t(x)]_i > 3\eps_0$ for all $t\ge T$. Fix $\delta \in (0, \eps_0)$.
    Consider a $(\delta, T)$ \ap--pseudo-orbit, $\{\xi_k\}_{k=0}^n$ with $\xi_0\in A$ and associated time instants
    $T \leq T_1, \dots,  T_{n-1}$. Clearly $\xi_0 \not\in E_{2\eps_0}$ and by \eqref{eq:lowbdode},
    $\varphi_t(\xi_0) \not\in E_{2\eps_0}$ for all $t \in [0, T_1]$. Also, by our choice of $T$,
    $\varphi_{T_1}(\xi_0) \not \in E_{3\eps_0}$ and consequently $\xi_1 \not \in E_{2\eps_0}$. A recursive argument now shows that the pseudo-orbit has no intersection with $E_{\eps_0}$. The result follows.
\end{proof}

%
%
%
%
%
%
%

We now recall a definition from the theory of dynamical systems.

\begin{defn}
The $\omega$-limit set of $B \subset \Delta$ is
\[
\omega(B) \stackrel{\cdot}{=} \left\{ x \in \Delta : \text{there is a sequence }t_n \uparrow \infty\text{ and a sequence }x_n \in B \text{ such that } \varphi_{t_n}(x_n) \rightarrow x\right\},
\]
so for $x \in \Delta$,
\[
\omega(x) \doteq \left\{ y \in \Delta :\text{ there is a sequence }t_n \uparrow \infty\text{ such that }\varphi_{t_n}(x) \rightarrow y\right\}.
\]

\end{defn}
The following result follows from classical arguments and on observing that under Assumption \ref{assu:ap-classesfinite}(b), if $x \in \Delta^o$, then $\omega(x)\subset \Delta^o$. For a proof of the lemma in the discrete time setting see \cite{fausch}. The proof for the continuous time setting considered here is  similar and we omit details.

\begin{lemma}\label{lem:prop4.12}
For any  $x \in  \Delta$, $\omega(x) \subset \mathcal{R}_\ap$. 
\end{lemma}

The following lemma gives a useful property of an \ap-quasiattractor.
\begin{lemma}\label{lem:maximalap}
If $[x]_\ap$ is maximal, then $x <_\ap z$ if and only if $z \in [x]_\ap$.  
\end{lemma}

\begin{proof}
Suppose that $x <_\ap z$. In order to show that $z \in [x]_\ap$, it is enough to show that $z <_\ap x$. Note that $\omega\left(z\right)$ is nonempty. Let $z' \in \omega\left(z\right)$.
From Lemma \ref{lem:prop4.12}   $z' \in \mathcal{R}_\ap$. 
We now show that $z <_\ap z'$. 
Since $z' \in \omega\left(z\right)$, there is a sequence $T_i \uparrow \infty$  such that $\varphi_{T_i}\left(z\right) \rightarrow z'$. Fix  $\delta, T > 0$. Then we can find  $T' > T$ such that $\| \varphi_{T'}\left(z\right) - z' \| < \delta$. This shows that $(z, z, z', z')$, is a $(\delta,T)$ \ap--pseudo-orbit from $z$ to $z'$. Since $\delta, T>0$ are arbitrary, we have $z <_\ap z'$.
Combining this with $x <_\ap z$ we now see that $x <_\ap z'$. Since $z' \in \mathcal{R}_\ap$ and $[x]_\ap$ is maximal, we must have $z' \in [x]_\ap$.
 and therefore $z <_\ap x$. This completes the proof of the  lemma.
 
\end{proof}

The following lemma provides an important invariance property of \ap-classes under the flow $\{\varphi_t\}$.
\begin{lemma}\label{lem:apbasicinvariant}
Any \ap--basic class $[x]_\ap$ is positively $\varphi_t$-invariant for all $t \geq 0$: $\varphi_t([x]_\ap) \subset [x]_\ap$.
Additionally, if $[x]_{\ap} \subset \Delta^o$, then $[x]_\ap$ is $\varphi_t$-invariant for all $t \geq 0$: $\varphi_t([x]_\ap) = [x]_\ap$.
\end{lemma}

\begin{proof}
Let $y \in [x]_\ap$. To begin, fix $t, \delta, T > 0$, and let $T' > T + t$. We can find some $\delta_0 \doteq \delta_0(y) < \delta$ such that if $\| y - x_0\| < \delta_0$, then $\| \varphi_t(y) - \varphi_t(x_0)\| < \delta$. Since $y \in \mathcal{R}_\ap$, there is a $(\delta_0, T')$ \ap--pseudo-orbit from $y$ to $y$, which we denote by $\xi = (y, \xi_1, \dots, \xi_{n-1}, y)$, with corresponding time instants $(T_1, \dots,T_{n-1})$. Then $\tilde{\xi} \doteq ( \varphi_t(y), \varphi_t(\xi_1), \xi_2, \dots, \xi_n = y)$ is a $(\delta, T)$ \ap--pseudo-orbit from $\varphi_t(y)$ to $y$ with corresponding time instants $(T_1 - t, T_2, \dots,T_{n-1})$, since
$$
\| \varphi_t(x) - \varphi_t(\xi_1)\| < \delta,
\mbox{ and }
\| \varphi_{T_1-t}(\varphi_t(\xi_1)) - \xi_2\|  = \| \varphi_{T_1}(\xi_1) - \xi_2\| < \delta.
$$
Thus $\varphi_t(y) <_\ap y$. 


Next, define $\tilde{\xi} \doteq (y,\xi_1,\dots,\xi_{n-1}, \varphi_t(y))$, and note that $\tilde{\xi}$ is a $(\delta,T)$ \ap--pseudo-orbit from $y$ to $\varphi_t(y)$ with time instants $(T_1,\dots,T_{n-2}, T_{n-1} + t)$, since
$$
\| \varphi_{T_{n-1}}(\xi_{n-1}) - y\| < \delta_0,
$$
which ensures that
$$
\| \varphi_{T_{n-1}+t}(\xi_{n-1}) - \varphi_t(y)\|   = \| \varphi_t(\varphi_{T_{n-1}}(\xi_{n-1})) - \varphi_t(y)\| < \delta.
$$
We have shown that $\varphi_t(y) \sim_\ap y$, and so $\varphi_t(y) \in [y]_\ap = [x]_\ap$. Since $y \in [x]_\ap$ is arbitrary, $\varphi_t([x]_\ap) \subset [x]_\ap$. This proves the first part of the lemma.

For the second part, suppose now that $[x]_\ap \subset \Delta^o$. 
In order to see that $[x]_\ap \subset \varphi_t([x]_\ap)$ for each $t \geq 0$, let $y \in [x]_\ap$ and fix $t > 0$. We need to show that there is some $z \in [x]_\ap$ such that $\varphi_t(z) = y$. Fix a sequence $(\delta_k, T_k)$ such that $\delta_k \downarrow 0$ and $T^k \uparrow \infty$.
Since $y \in \mathcal{R}_\ap$,  we can find a sequence of $(\delta_k,T^k)$ \ap--pseudo-orbits with corresponding time instants $\{T_i^k\}_{i=0}^{n(k)-1}$ from $y$ to $y$, which we  denote by $\xi^k = (\xi^k_0,\dots,\xi^k_{n(k)})$.  We  assume without loss of generality that $T^k > t$ for all $k$ and let $\tilde{T}^k \doteq T^k_{n(k) -1} - t$. From Lemma \ref{lem:stabilitylemma} there is a  compact $\tilde K$ in $\Delta$ such that for all sufficiently large $k$, $\xi^k_i \in \tilde{K}$ for all $i \in \{0,\dots,n(k)\}$.
From  Lemma \ref{lem:compactbound} we then have that, for all such $k$, $\varphi_{\tilde{T}^k}(\xi^k_{n(k)-1})$
lies in some compact set $\tilde{K}'$.
Thus (passing to a subsequence) we may assume that $\varphi_{\tilde{T}^k}(\xi^k_{n(k)-1}) \rightarrow z \in \tilde{K}'$. Since
$$
\varphi_{t}( \varphi_{\tilde{T}^k}(\xi^k_{n(k)-1}))  = \varphi_{T^k_{n(k)-1}}(\xi^k_{n(k)-1})  \rightarrow y,
$$
the continuity of $\varphi_t$ ensures that $\varphi_t(z) = y.$ 
Now we show that $z \in [x]_{\ap}$.
Fix $\delta, T > 0$, and let $k$ be large enough so that $\delta_k < \delta$, $\| \varphi_{\tilde{T}^k}(\xi^k_{n(k)-1}) - z\| < \delta_k$, $T^k > T$, and $\tilde{T}^k > T$. Then $(\xi^k_0, \dots, \xi^k_{n(k)-1}, z)$ is a $(\delta, T)$ \ap--pseudo-orbit from $y$ to $z$ with corresponding time instants $(T^k_1,\dots,T^k_{n(k)-2}, \tilde{T}^k)$, so $y <_\ap z$. Now, fix $\tilde{t} > \max\{t, T\}$, and note that 
$$
\varphi_{\tilde{t}}(z) = \varphi_{\tilde{t} -t }( \varphi_t(z)) = \varphi_{\tilde{t} - t}(y).
$$
Since $y \in [x]_\ap$, it follows from the positive $\varphi_t$-invariance of $[x]_\ap$ that $\varphi_{\tilde{t}-t}(y) \in [x]_\ap$, so there is a $(\delta,T)$ \ap--pseudo-orbit from $\varphi_{\tilde{t}-t}(y)$ to $y$, which we  denote by $(\xi_0,\dots,\xi_n)$. Denote the corresponding time instants by $T_1,T_2,\dots,T_{n-1}$. Then $\tilde{\xi} \doteq (z, z, \xi_1, \dots, \xi_n)$ is a $(\delta,T)$ \ap--pseudo-orbit from $z$ to $y$ with time instants $(\tilde{t},T_1,\dots,T_{n-1})$, so $z <_\ap y$ and $z \in [x]_\ap$.

\end{proof}

We now recall the definition of an attractor for the flow $\{\varphi_t\}$.
\begin{defn}
A compact set $A$ is an attractor for the flow $\{\varphi_t\}$ if $\varphi_t(A) =A$ for each $t \geq 0$ and there is some neighborhood $U$ of $A$ such that
\[
\lim\limits_{t\rightarrow\infty}\sup\limits_{x\in U}\dist(\varphi_t(x),A) = 0.
\]
The neighborhood $U$ is referred to as a fundamental neighborhood for the attractor $A$.
\end{defn}

The proof of Corollary \ref{cor:quasiattract} follows the proof of \cite[Proposition 4.2]{kif}.

\begin{corollary}\label{cor:quasiattract}
If $[x]_\ap \subset \Delta^o$ is an  \ap--quasiattractor, then $[x]_\ap$ is an attractor.
\end{corollary}

\begin{proof} 
Recall that $\clr_{\ap}^*$ denotes the collection of all \ap-chain recurrent points in $\Delta^o$. Note that, from Assumption \ref{assu:ap-classesfinite}(a) and Lemma \ref{lem:apbasicbounded}, for each $z \in \clr_{\ap}^*$, $[z]_{\ap}$ is a compact set.
Choose $\delta>0$ such that $\cln^{\delta}([x]_{\ap})$ is an isolating neighborhood of $[x]_{\ap}$
with closure contained in $\Delta^o$. Then, from Lemma \ref{lem:compactbound} and Assumption \ref{assu:ap-classesfinite}(b), there is a compact $K_0 \subset \Delta^o$ such that for all $z\in \cln^{\delta}([x]_{\ap})$, $\varphi_t(z) \in K_0$ for all $t\ge 0$.
Let $\eps^* \doteq \inf\limits_{y \in \mathcal{R}_\ap\setminus [x]_\ap}\dist([y]_\ap,[x]_\ap)$
Let for $\eps \le \eps^*$, $K^{\eps} \doteq K_0 \setminus \left( \bigcup\limits_{y\in\mathcal{R}^*_\ap \setminus [x]_\ap}\cln^{\eps}([y]_\ap)\right)$. 
We claim that there is a $\eps \le \eps^*$ and a $\delta_0 \le \delta$ such that 
\begin{equation}\label{eq:claim1}
\mbox{ for all } z \in \cln^{\delta_0}([x]_{\ap}), \; \varphi_t(z) \in K^{\eps} \mbox{ for all } t \ge 0.
\end{equation}
We argue via contradiction. Suppose the claim is false, then, since there are finitely many \ap-basic classes in
$\clr_{\ap}^*$, there exist $\delta_n\downarrow 0$, $\eps_n \downarrow 0$, $z_n \in \cln^{\delta_n}([x]_{\ap})$, $t_n\ge 0$, $y \in \mathcal{R}^*_\ap \setminus [x]_\ap$, such that
$\varphi_{t_n}(z_n) \in \cln^{\eps_n}([y]_{\ap})$. Passing to a subsequence we may assume that $z_n \to z$ and $\varphi_{t_n}(z_n) \to u$. Then $z \in [x]_\ap$ and $u \in [y]_\ap$. 
We consider two cases: (I) along a further subsequence $t_n$ converges to some $t^* <\infty$; (II) $t_n\to \infty$. In case (I), $u= \varphi_{t^*}(z)$ and so by Lemma \ref{lem:apbasicinvariant} $u \in [x]_{\ap}$.
But this is a contradiction since $y\notin [x]_{\ap}$. In case (II), for every $\delta, T>0$, there is a $(\delta, T)$ \ap--pseudo-orbit from $z$ to $u$ which says that $z <_{\ap} u$. Since $[x]_\ap$ is a quasiattractor, from Lemma \ref{lem:maximalap} $u \in [x]_\ap$ which is once more a contradiction to the fact that $y\notin [x]_{\ap}$. Thus we have the claim.
Now fix $\delta_0 \le \delta^*$ and $\eps \le \eps^*$ so that \eqref{eq:claim1} holds.

We now argue that
\begin{equation}
    \label{eq:claim2}
    \mbox{ for some } \delta_1 \in (0, \delta_0), \mbox{ whenever } y \in \cln^{\delta_1}([x]_\ap), \mbox{ we have } \varphi_t(y) \in \cln^{\delta_0}([x]_\ap) \mbox{ for all } t \ge 0.
\end{equation}
Once more we proceed via contradiction. Suppose the statement is false. Then there exist $\delta_n \downarrow 0$,
$y_n \in \cln^{\delta_n}([x]_\ap)$, $t_n\ge 0$ such that $\varphi_{t_n}(y_n) \in \left(\cln^{\delta_0}([x]_\ap)\right)^c$. We can find a subsequence along which $y_n \to y$ and $\varphi_{t_n}(y_n) \to u$. We must have $y \in [x]_\ap$ and $u \in \left(\cln^{\delta_0}([x]_\ap)\right)^c$.
Once again we consider two cases as above. In case (I), $u=\varphi_{t^*}(y) \in [x]_{\ap}$ which contradicts the fact that $u \in \left(\cln^{\delta_0}([x]_\ap)\right)^c$. In case (II) $y<_{\ap} u$ and so as before, $u \in [x]_{\ap}$. Once more this is a contradiction. Thus we have shown \eqref{eq:claim2}. Now fix $\delta_1 \in (0, \delta_0)$ such that \eqref{eq:claim2} holds.
Let $U_0 \doteq \cln^{\delta_0}([x]_\ap)$  and $ U_1 \doteq \cln^{\delta_1}([x]_\ap)$.

We will now show that 
\begin{equation}\label{eq:attraprop}
\lim\limits_{t\rightarrow\infty}\sup\limits_{y\in U_1}\dist(\varphi_t(y), [x]_\ap) = 0.
\end{equation}
Together with Lemma \ref{lem:apbasicinvariant} we will then have that $[x]_\ap$ is an attractor, completing the proof of the result.
In order to show \eqref{eq:attraprop} we will show that for each open neighborhood $O$ of $[x]_\ap$,
$O\subset U_1$, there is some $t(O) <\infty$ such that $\varphi_t(U_1) \subset O$ for all $t \geq t(O)$. Let
for any such $O$, $O_1 \subset \subset O$ be an open neighborhood of $[x]_\ap$ such that for all $y \in O_1$,
$\varphi_t(y) \in O$ for all $t\ge 0$. Here for open sets $G_1, G_2$, we write $G_1 \subset \subset G_2$ if $\bar G_1 \subset G_2$. Existence of such an $O_1$ is shown in a similar manner as \eqref{eq:claim2}.
It suffices to show that
$$
t(O) \doteq \inf\{t : \varphi_t(U_0) \subset O\} < \infty,
$$
since then for each $t \geq t(O)$,
$$
\varphi_t(U_1) = \varphi_{t(O)} ( \varphi_{t - t(O)}(U_1)) \subset \varphi_{t(O)}(U_0)  \subset O,
$$
which will complete the proof.

 In order to see that $t(O) < \infty$ for each such $O$, we argue by contradiction.
 Suppose that there is some $O$ (with the associated $O_1$) such that $t(O) = \infty$. Then we can find sequences $\{z_n\} \subset U_0$ and $T_n \uparrow \infty$ such that $\varphi_{T_n}(z_n) \in O^c$. From the definition of $O_1$, this says that $\varphi_t(z_n) \in O^c_1$ for all $0\le t \le T_n$. Suppose that $z_n \to z$ along a subsequence. Then $\varphi_t(z) \in O_1^c$ for all $t > 0$. Also, since $z \in \cln^{\delta_0}([x]_\ap)$, by \eqref{eq:claim1},
 $\varphi_t(z) \in K^{\eps}$ for all $t \ge 0$. Thus we have $\omega(z) \subset K^{\eps} \setminus O_1$.
 The final statement of Lemma \ref{lem:bdryrepel} implies that for each $x \in \Delta^o$, $\omega(x) \neq \emptyset$ and therefore
  $\omega(z)$ is a nonempty subset of $\clr_{\ap}^*$. Thus we have that $(K^{\eps}\setminus O_1)\cap \clr_{\ap}^*$
 is nonempty which contradicts the definition of $K^{\eps}$ and $O_1$.
 Thus we have that $t(O)<\infty$ and the result follows.\\
\end{proof}

The following lemma shows that suitable \ap-pseudo-orbits come arbitrarily close to \ap-recurrence classes.

\begin{lemma}\label{lem:intersectrap} 
	\begin{enumerate}[(a)]
		\item For each $\delta>0$ and compact $A \in \Delta$, there is a  $\delta_0 \in (0,1]$  and $T_A \in (0,\infty)$ such that  any 
 $(\delta_0, T_{A})$ \ap--pseudo-orbit that starts in $A$ intersects $N^{\delta}(\mathcal{R}_\ap)$.
 \item  For each $\delta>0$ and  $A \in \clk$, there is a   $T_A^* \in (0,\infty)$ such that for every $x \in A$, there is a $t_0 \in [0, T_A^*]$ with
 $\varphi_{t_0}(x) \in N^{\delta}(\mathcal{R}_\ap^*)$.
 \end{enumerate}
\end{lemma}

\begin{proof}
Consider first part (a). Fix $\delta>0$, a compact $A \in \Delta$, and let $T=1$. With this choice of $A$ and $T$, let $\delta_0$ and $A_1$ be as given in Lemma \ref{lem:stabilitylemma}.
Let,
	for  $x \in \Delta$,  $T^{\delta}(x) \stackrel{\cdot}{=}\inf\{ t\geq0: \varphi_t(x) \in N^{\delta}(\mathcal{R}_\ap)\}$. 
Since $\omega(x)$ is a nonempty subset of $\mathcal{R}_\ap$, $T^{\delta}(x) <\infty$ for each $x \in \Delta$.	We now claim that 
$T^{\delta}$ is an upper semicontinuous (usc) function on $\Delta$. For this it suffices  to argue that 
for each $\alpha > 0$, the level set $L_{\alpha} \stackrel{\cdot}{=} \{x \in \Delta : T^{\delta}(x) \geq \alpha\}$ is closed.
Let $\{x_n\} \subset L_{\alpha}$ be a sequence converging to some $x \in \Delta$, and note that  for each $t \ge 0$, $\lim\limits_{n\rightarrow\infty}\varphi_t(x_n) = \varphi_t(x)$. For $t < \alpha$, $\varphi_t(x_n) \in \left(N^{\delta}(\mathcal{R}_\ap)\right)^c$, which is closed, so $\varphi_t(x) \in  \left(N^{\delta}(\mathcal{R}_\ap)\right)^c$. Since this holds for all $t < \alpha$, we have that $x \in L_{\alpha}$. 
This shows that the level sets of $T^{\delta}$ are closed and thus establishes the claim. Since an usc function achieves its supremum over any compact set, $T_1 = \sup_{x\in A_1}T^{\delta}(x)<\infty$. Let $T_A \doteq T_1 \vee 1$.
Then, from Lemma \ref{lem:stabilitylemma}, for any $(\delta_0, T_{A})$ \ap--pseudo-orbit, given by a collection of points $\xi = (\xi_0 = x,\dots, \xi_n = y) \in \Delta^{n+1}$ and a collection of times $T_A \leq T_1, \dots,  T_{n-1}$,
with $x\in A$, must satisfy $\xi_i \in A_1$ for every $i \in \{0, \ldots, n\}$. Also, by the definition of $T_A$, we must have that for each
$i \in \{1,\dots,n-1\}$, there is a $t \in [0, T_i]$ such that $\varphi_t(\xi_i) \in N^{\delta}(\mathcal{R}_\ap)$.
The result in part (a) follows.

The proof of part (b) can be completed in a similar manner on observing that from Lemma \ref{lem:bdryrepel}, for every $x \in A$, the forward orbit $\gamma^+(x)$ is contained in a compact subset of $\Delta^o$.
We omit the details.
\end{proof}

The following lemma gives key properties of pseudo-orbits in relation to their visits to neighborhoods of \ap-quasiattractors and non-quasiattractors.
\begin{lemma}\label{lem:leaveattractor}
\begin{enumerate}[(a)]
	\item For every $\theta >0$, there are $\delta=\delta(\theta) < \theta$ and $T=T(\theta) > 0$ with the property if there is a $(\delta,T)$ \ap--pseudo-orbit $\xi \doteq (\xi_0,\dots, \xi_n)$ with 
	\begin{equation}
		\label{eq:poprop}\xi_0 \in N^{\delta}(K_i),\; \xi_n \in N^{\delta}(K_{i'}),\;  \mbox{ and } \xi_j \in (N^{\theta}(K_i))^c \mbox{ for some } j \in \{1,\dots, n-1\}, 
		\end{equation}
		then we must have $i \neq i'$.
	\item There exist $\delta, T > 0$ such that if for some $i,i' \in \{1, \ldots, v\}$ there is a $(\delta,T)$ \ap--pseudo-orbit $\xi \doteq (\xi_0,\dots,\xi_n)$ such that $\xi_0 \in N^{\delta}(K_i)$ and $\xi_n \in N^{\delta}(K_{i'})$, then we must have that $K_{i} \leq_\ap K_{i'}$.
	\end{enumerate}
\end{lemma}

\begin{proof}
For the first statement in the lemma we will argue via contradiction. 
By Lemma \ref{lem:appobd} we can choose $\bar \delta>0, \bar T>0$ and $\tilde K \in \clk$ such that any
 $(\bar \delta, \bar T)$ \ap--pseudo-orbit starting from $N^{\bar\delta}(K_i)$ lies in $\tilde K$ for every $i = 1, \ldots, v$. Henceforth we only consider $(\delta,T)$ \ap--pseudo-orbits with $\delta <\bar \delta$ and $T > \bar T$.
Fix $\theta>0$ and suppose that 
there is a sequence $\theta>\delta_k \downarrow 0$ and $T_k \uparrow \infty$, such that for every $k$ there is a 
$(\delta_k,T_k)$ \ap--pseudo-orbit $\xi^k \doteq (\xi^k_0,\dots, \xi^k_{n(k)})$ that satisfies \eqref{eq:poprop} (with $\xi,\delta, n$ replaced with $\xi_k,\delta_k, n(k)$), with $i=i'$.  Let $j(k) \in \{1,\dots, n(k) - 1\}$ be such that $\xi^k_{j(k)} \in (N^{\theta}(K_i))^c$.
By passing to a subsequence if necessary, we can find $x, y \in K_i$ and $z \in (N^{\theta}(K_i))^c\cap \tilde K$ such that $\xi^k_0 \rightarrow x, \xi^k_{n(k)} \rightarrow y$, and $\xi^k_{j(k)} \rightarrow z$.

In order to see that $x \leq_\ap z$, fix $\delta,T > 0$ and let $k$ be large enough so that $\delta_k < \frac{\delta}{2}$,  $T_k > T$, $\|x - \xi^k_0\| < \frac{\delta}{2}$, and $\| \xi^k_{j(k)} - z\| < \frac{\delta}{2}$. Then
$
\| x - \xi^k_1\| \leq \| x - \xi^k_0\| + \|\xi^k_0 - \xi^k_1\| < \delta$,
and
$$\| \varphi_{T^k_{j(k)-1}}(\xi^k_{j(k)-1}) - z\| \leq \| \varphi_{T^k_{j(k)-1}}(\xi^k_{j(k)-1}) - \xi^k_{j(k)}\| + \| \xi^k_{j(k)} - z\| < \delta,
$$
and so $\tilde{\xi} \doteq (x, \xi^k_1, \dots, \xi^k_{j(k)-1}, z)$ is a $(\delta,T)$ \ap--pseudo-orbit from $x$ to $z$. Thus $x <_\ap z$. Similarly, $z <_\ap y$, which shows that $z \in K_i$. However, since  $z \in (N^{\theta}(K_i))^c$, this is a contradiction. This proves (a).

Now consider part (b).  Fix $i,i' \in \{1, \ldots, v\}$ and suppose that for each $\delta, T > 0$ there is some $(\delta,T)$ \ap--pseudo-orbit $\xi \doteq (\xi_0,\dots,\xi_n)$ such that $\xi_0 \in N^{\delta}(K_i)$ and $\xi_n \in N^{\delta}(K_{i'})$. Let $\delta_k \downarrow 0$ and $T_k \uparrow \infty$
and let  $\xi^k \doteq (\xi^k_0,\dots, \xi^k_{n(k)})$ be a $(\delta_k,T_k)$ \ap--pseuodoorbit such that $\xi^k_0 \in N^{\delta_k}(K_i)$ and $\xi^k_n \in N^{\delta_k}(K_{i'})$. Passing to subsequences if necessary, we can find $x \in K_i$ and $y \in K_{i'}$ such that $\xi^k_0 \rightarrow x$ and $\xi^k_{n(k)}\rightarrow y$. Thus, for any fixed $\delta, T > 0$, when $k$ is sufficiently large, $\tilde{\xi} \doteq (x, \xi^k_1,\dots, \xi^k_{n(k)-1}, y)$ is a  $(\delta,T)$ \ap--pseudo-orbit from $K_i$ to $K_{i'}$,  showing that $K_{i} \leq_\ap K_{i'}$. So if for some $i, i'$, $K_{i} \leq_\ap K_{i'}$ does not hold, there must exist $\bar \delta =\delta(i,i')>0$ and $\bar T = T(i,i')<\infty$ such that 
there is no $(\bar \delta, \bar T)$ \ap--pseudo-orbit $\xi \doteq (\xi_0,\dots,\xi_n)$ with the property that $\xi_0 \in N^{\bar\delta}(K_i)$ and $\xi_n \in N^{\delta}(\bar K_{i'})$. Define, $\delta = \min_{(i,i')} \delta(i,i')$ and $T= \max_{(i,i')} T(i,i')$.  Clearly, the statement in part (b) holds  with this choice of $(\delta,T)$.
\end{proof}

The final result of this section is a consequence of Lemma \ref{lem:intersectrap} and Lemma \ref{lem:leaveattractor}. It summarizes key properties of \ap-pseudo-orbits in relation to \ap-recurrent classes.
This result will be used in Section \ref{sec:pfofmain} in the proof of Theorem \ref{thm:main}.
\begin{lemma}\label{lem:choosevi}
 For each $\delta, T> 0$, and compact set $A\subset \Delta^o$,  there is a collection of open neighborhoods $\{V_i\}_{i=1}^v$ of $\{K_i\}_{i=1}^v$, with $\bar V_i \subset N^{\delta}(K_i)\cap \Delta^o$,  along with $\delta_0 \in (0,\delta)$, $T_0 \in (T, \infty)$, and $n \in \N$, such that  the following hold:
\begin{enumerate}
\item $\overline{N^{\delta_0}(K_i)} \subset V_i$ for each $i \in \{1,\dots, v\}$.
\item For each $i \in \{1,\dots,l\}$, if $\xi \doteq (\xi_0, \dots, \xi_n)$ is a $(\delta_0,T_0)$ \ap--pseudo-orbit with $\xi_0 \in V_i$, then $\xi_j \in V_i$ for all $j \in \{1, \dots, n\}$.
\item If $\xi \doteq (\xi_0,\dots, \xi_n)$ is a $(\delta_0,T_0)$ \ap--pseudo-orbit with corresponding time instants $(T_1,\dots,T_{n-1})$ such that $\xi_0 \in N^{\delta_0}(K_i)$ and $\xi_n \in N^{\delta_0}(K_j)$ for some $i,j \in \{1,\dots,v\}$, and there is $m \in \{1,\dots,n-1\}$ such that $\xi_m \in V_i^c$, then $i \neq j$ and $K_i \leq_\ap K_j$.
\item 
If $\xi \doteq (\xi_0,\dots,\xi_n)$ is a $(\delta_0,T_0)$ \ap--pseudo-orbit with $\xi_0 \in A$, then there is some $k \in \{1,\dots,n-1\}$ and $t \in [0,T_k]$ such that $\varphi_{t}(\xi_k) \in N^{\delta}(\mathcal{R}_\ap\cap \Delta^o)$.
\end{enumerate}
\end{lemma}

\begin{proof}

Fix $\delta, T \in (0,\infty)$ and a compact $A \in \Delta^o$.
Since $K_i$ is an attractor for each $i \in \{1,\dots,l\}$,  
 there is a bounded, open neighborhood ${O}_i$ of $K_i$, with $\bar O_i \subset N^{\delta}(K_i)\cap\Delta^o$ such that 
\begin{equation}\label{eq:dptki}
\lim\limits_{t\rightarrow\infty}\sup\limits_{x\in {O}_i}\dist(\varphi_t(x),K_i) = 0.
\end{equation}
For each $i \in \{l+1,\dots,v\}$, let ${O}_i$ be an arbitrary bounded, open, and isolating neighborhood of $K_i$ such that
$\bar O_i \subset  N^{\delta}(K_i)\cap\Delta^o$. 
Denote the $(\delta, T)$ given by part (b) in Lemma \ref{lem:leaveattractor} by $(\delta^*_1, T^*_1)$ and
$(\delta_0, T_A)$ given by Lemma \ref{lem:intersectrap}(a) as $(\delta^*_2, T^*_2)$.
Let $\theta>0$ be small enough so that $\overline{N^{\theta}(K_i)} \subset {O}_i$ for each $i \in \{1,\dots, v\}$. 
From Lemma \ref{lem:leaveattractor} we can find $\delta_1 < \min\{\theta, \delta^*_1, \delta^*_2\} $ and $T_1 > \max\{T, T^*_1, T^*_2\}$ such that if $\xi \doteq (\xi_0,\dots,\xi_n)$ is a $(\delta_1,T_1)$ \ap--pseudo-orbit with $\xi_0 \in N^{\delta_1}(K_i)$ and $\xi_n \in N^{\delta_1}(K_j)$ such that $\xi_m \in (N^{\theta}(K_i))^c$ for some $m \in \{1,\dots, n-1\}$, then $i \neq j$ and $K_{i} \leq_\ap K_j$.

Now, let $V_i \doteq N^{\theta + \eps}(K_i)$, where $\eps > 0$ is small enough so that $\overline{V_i} \subset {O}_i$ for all $i \in \{1,\dots,v\}$, and let $\delta_2 < \delta_1$ be small enough so that $\overline{N^{\delta_2}(V_i)} \subset {O}_i$. 
Thus for every $i \in \{1,\dots,v\}$
$$K_i \subset N^{\delta_2}(K_i)\subset \subset N^{\theta}(K_i)\subset \subset V_i \subset \subset N^{\delta_2}(V_i)\subset \subset O_i,$$
where, as before, for open sets $G_1, G_2$, we write $G_1 \subset \subset G_2$ if $\bar G_1 \subset G_2$.

From \eqref{eq:dptki}, there is some $T_2 > T_1$ such that if $t \geq T_2$, then for each $i \in \{1,\dots,l\}$,
$$
\sup\limits_{u \in {O}_i}\dist(\varphi_t(u), K_i) < \delta_2.
$$
Then (1) and (2) hold when $\delta_0 \doteq \delta_2$ and $T_0 \doteq T_2$. Additionally,  (3) holds from  the property of
$(\delta_1, T_1)$ \ap--pseudo-orbits noted above since  $V_i^c \subset (N^{\theta}(K_i))^c$ for each $i \in \{1,\dots,v\}$. Finally, since $T_0\ge T_2^*$ and $\delta_0 \le \delta_2^*$, 
 from Lemma \ref{lem:intersectrap},  (4) holds as well.

\end{proof}

\section{Large Deviation Estimates}\label{sec:ldest}
Throughout this section we will assume that Assumption \ref{assu:mgf} is satisfied. We will give some key uniform large deviation bounds that will be 
used in Sections \ref{sec:behnb}, \ref{sec:vrec} and, \ref{sec:pfofmain}.

For $\alpha \in (0,1)$ let $\clv_{\alpha} \doteq \{x \in \Delta^o: \mbox{dist}(x,\partial \Delta) >\alpha\}$.
For each compact $K \in \clk$, let $\clv_{\alpha, K} \doteq\clv_{\alpha} \cap K$ and let $\pi_{\alpha, K}$ denote the projection map from $\R^d$ to $\bar{\clv}_{\alpha, K}$ defined as
$$\pi_{\alpha, K}(x) \doteq \mbox{arg\! min}_y \{\|y-x\|: y \in \bar{\clv}_{\alpha, K}\}.$$
Similarly, denote by $\pi^N_{\alpha, K}$  the projection map from $\R^d$ to $\bar{\clv}_{\alpha, K}\cap \Delta_N$.
Let $\theta^{N,\alpha, K}$ be a transition probability kernel on $\R^d$ defined by
$$\theta^{N,\alpha, K}(\cdot | x) \doteq \theta^N(\cdot | \pi^N_{\alpha, K}(x)).$$
Let $\{X^{N,\alpha,K}_n\}$ be a $\R^d$-valued chain defined as in \eqref{eq:bmc} but with $\theta^N$ replaced with
 $\theta^{N,\alpha,K}$.
We consider continuous time processes $\hat X^{N,\alpha,K}$ associated with
$\{X^{N,\alpha,K}_n\}$ as 
$$
\hat{X}^{N,\alpha,K}(t) = X^{N,\alpha,K}_n + [X^{N,\alpha,K}_{n+1} - X^{N,\alpha,K}_{n}](Nt - n), \; t \in [n/N, (n+1)/N], \; n \in \N_0.
$$
We now present a basic large deviation result for $\hat X^{N,\alpha,K}$. Recall the stochastic kernel $\theta(dy| x)$ from Assumption \ref{assu:mgf}(b).
Define for $x, \zeta \in \R^d$
$$
H_{\alpha, K}(x, \zeta) \doteq \log \int_{\R^d} \exp\{ \langle \zeta , y\rangle\} \theta(dy| \pi_{\alpha, K}(x)),$$
and let
$$L_{\alpha, K}(x,\beta) \doteq \sup_{\zeta \in \R^d} \{ \langle \zeta , \beta \rangle - H_{\alpha, K}(x, \zeta)\}.$$
We note that for every $\beta, \zeta \in \R^d$, $H_{\alpha, K}(x, \zeta) = H_{\alpha', K'}(x, \zeta)$ and $L_{\alpha,K}(x,\beta) = L_{\alpha',K'}(x,\beta)$ whenever $\pi_{\alpha, K}(x)= x= \pi_{\alpha', K'}(x)$. Define for $x \in \Delta^o$ and $\beta, \zeta \in \R^d$,
$$
L(x,\beta) \doteq L_{\alpha, K}(x,\beta), \; H(x,\zeta) \doteq H_{\alpha, K}(x,\zeta) \mbox{ if } x \in \clv_{\alpha, K}.$$

For $\alpha > 0$, $x \in \R^d$, $K \in \clk$, $T \in (0,\infty)$ and $\phi \in C([0,T]:\R^d)$, define
$$
S_{\alpha, K}(x,T, \phi) \doteq \left\{
\begin{array}
[c]{cc}%
\int_0^T L_{\alpha, K}(\phi(t),\dot{\phi}(t)) dt, & \text{if } \phi \mbox{ is absolutely continuous}\\
\infty, & \text{otherwise }
\end{array}.
\right.  
$$
Note that if for $\alpha, \alpha' >0$ and $K,K'\in \clk$, $\phi \in C([0,T]:\bar{\clv}_{\alpha, K})\cap C([0,T]:\bar{\clv}_{\alpha', K'})$
then $S_{\alpha, K}(\phi(0),T, \phi) = S_{\alpha', K'}(\phi(0),T, \phi)$. Thus we define for $\phi \in C([0,T]:\R^d)$ that satisfies $\phi(0) = x$ and $\phi(t) \in \Delta^o$ for all $t \in [0,T]$,
\begin{equation}\label{eq:ratefns}
S(x, T, \phi) = S_{\alpha, K}(x,T, \phi) \mbox{ if } \phi \in C([0,T]:\bar{\clv}_{\alpha, K}) \mbox{ for some } \alpha >0 \mbox{ and }K \in \clk.
\end{equation}

The following uniform large deviation principle  will be used several times in this work. 
\begin{theorem}\label{thm:fwuldp}
	Suppose Assumption \ref{assu:mgf} is satisfied.
	Fix $T\in (0,\infty)$, $\alpha>0$ and   $K, K' \in \clk$.
	For each $a  \in (0,\infty)$, let 
	$$\Phi_{x,\alpha,K',T}(a) \doteq \{\phi \in C([0,T]: \R^d) : S_{\alpha, K'}(x, T, \phi) \leq a\}.$$
	\begin{enumerate}[(a)]
	\item (Compact Level Sets) For every $a \in (0, \infty)$,  the set $\bigcup\limits_{x\in K}\Phi_{x,\alpha,K',T}(a)$ is compact.
	
	\item (Upper Bound) Given $\delta, \gamma \in (0,1)$ and $L \in (0,\infty)$, there is some $N < \infty$ such that
	$$\PP_x(\|\hat{X}^{n,\alpha, K'} - \phi\|_{*,T} < \delta) \geq \exp(-n(S_{\alpha,K'}(x,T,\phi) + \gamma))$$
	for all $n \geq N, x \in K \cap \Delta_N$, and $\phi \in \Phi_{x,\alpha,K',T}(L)$. 
	
	\item (Lower Bound) Given $\delta, \gamma \in (0,1)$ and $L \in (0,\infty)$, there is some $N < \infty$ such that
	$$\PP_x( d(\hat{X}^{n,\alpha,K'}, \Phi_{x,\alpha,K',T}(l)) \geq \delta) \leq \exp( -n(l - \gamma))$$
	for all $n \geq N$, $x \in K \cap \Delta_N$, and $l \in [0,L]$. 
	\end{enumerate}
\end{theorem}

\begin{proof}
We will apply \cite[Theorem 6.7.5]{dupell}.
Let for $x, \zeta \in \R^d$,
$$H_{\alpha, K}^N(x, \zeta) \doteq \log \int_{\R^d} \exp\{ \langle \zeta , y\rangle\} \theta^N(dy| \pi^N_{\alpha, K}(x)).$$
By Assumption \ref{assu:mgf}(b)(iii) for each compact $A \subset \R^d$ and $\zeta \in \R^d$,
\begin{equation}\sup_{N \in \N} \sup_{x\in \R^d} H^N_{\alpha,K}(x,\zeta)<\infty ,\; \sup_{x\in \R^d} H_{\alpha,K}(x,\zeta)<\infty.\label{eq:mgfbd2}\end{equation}
	and
\begin{equation}
	\sup_{x\in A} |H^N_{\alpha,K}(x,\zeta) - H_{\alpha,K}(x,\zeta)| \to 0 \mbox{ as } N\to \infty. \label{eq:mgfbd1}\end{equation}
Furthermore,  from Assumption \ref{assu:mgf}(b)(ii) $x \mapsto  \theta (dy| \pi_{\alpha, K}(x))$ is a continuous map from $\R^d$
to $\clp(\R^d)$. Thus \cite[Condition 6.2.1, Condition 6.7.2]{dupell} are satisfied.
Next, since from  Assumption \ref{assu:mgf}(b)(i)  the convex hull of the support of $\theta (dy| \pi_{\alpha, K}(x))$ is all of $\R^d$, 
\cite[Condition 6.7.4]{dupell} is satisfied as well. Thus, from  \cite[Theorem 6.7.5]{dupell} we have that, for every 
$T \in (0,\infty)$, $\{\hat{X}^{N,\alpha,K}\}_{N\in \N}$ satisfies a Laplace principle, uniformly on compact subsets of $\R^d$, in the sense of 
\cite[Definition 1.2.6]{dupell}, with rate function $S_{\alpha,K}(x, T, \cdot)$. 
It is shown in \cite[Theorem 4.3]{salbuddup} that a uniform Laplace principle of the form given in \cite[Theorem 6.7.5]{dupell}
implies a uniform Large deviation principle in the sense of Freidlin and Wentzell \cite{frewen}, which says that parts (a)-(c) of the theorem hold. The result follows.
\end{proof}

\begin{lemma}
	\label{lem:cty}
	For every $\alpha \in (0,1)$ and a compact $K \in \Delta^o$, $(x,\beta) \mapsto L_{\alpha,K}(x,\beta)$ is a continuous map on $\R^d \times \R^d$.
\end{lemma}
\begin{proof}
	The proof follows from \cite[Lemma 6.5.2]{dupell} on noting that, due to Assumption \ref{assu:mgf}(b) for every $x \in \R^d$,  the convex hull of the support of $\theta(dy| \pi_{\alpha, K}(x))$ is $\R^d$ and that
	$\sup_{x \in \R^d} H_{\alpha,K}(x, \zeta)<\infty$ for every $\zeta \in \R^d$.
\end{proof}

An important consequence of the above uniform large deviation principle is the following uniform upper bound for closed sets $F$ in
$C([0,T]:\R^d)$.

\begin{theorem}\label{thm:dzuldp}
	Fix $T\in (0,\infty)$, $\alpha>0$ and   $K,K' \in \clk$.
	Then, for every  closed set $F$ in $C([0,T]: \R^d)$ 
	$$ \limsup_{N\to \infty} \frac{1}{N} \log \sup_{x\in K\cap \Delta_N} 
	\PP_x(\hat X^{N,\alpha,K'} \in F) \le -\inf_{x \in K} \inf_{\phi \in F} S_{\alpha,K'}(x,T, \phi).$$ 
	
\end{theorem}
\begin{proof}
	Fix $T, \alpha, K, K'$ as in the statement of the theorem.
We begin by showing that 
for each $s \geq 0$ and $\delta > 0$ there is some $\eps \doteq \eps(\delta) \in (0,1)$ such that 
for all $x,y \in K$ with $\|x-y\|\le \eps$
\begin{equation}\label{eq:eq351}
\{ \phi \in C([0,T]: \R^d) : d(\phi, \Phi_{x,\alpha,K',T}(s)) \leq\delta\} \supseteq \left\{\phi \in C([0,T]: \R^d):
 d\left(\phi, \Phi_{y,\alpha,K',T}\left(s - \frac{\delta}{4}\right)\right) \leq \frac{\delta}{2}\right\}.
\end{equation}
Let $\kappa_0 \doteq 1+\sup_{x\in K, \|\beta\|\le 1} L(x,\beta)$. From Lemma \ref{lem:cty} $\kappa_0 <\infty$.
Since $\cup_{y \in K} \Phi_{y,\alpha,K',T}\left(s - \frac{\delta}{4}\right)$ is a compact set, we can find $\eps \in (0, \frac{\delta}{8\kappa_0})$ such that for all $\psi \in \cup_{y \in K} \Phi_{y,\alpha,K',T}\left(s - \frac{\delta}{4}\right)$
and $0 \le t_1\le t_2 \le T$ with $|t_1-t_2|\le \eps$, we have
$\|\psi(t_2)- \psi(t_1)\| \le \frac{\delta}{8}$.

Fix $y \in K$ and $\phi$ in the set on the right side of \eqref{eq:eq351}. Then there is a 
$\psi_1 \in \Phi_{y,\alpha,K',T}\left(s - \frac{\delta}{4}\right)$ such that
$\|\phi - \psi_1\|_{*,T} \le \frac{\delta}{2} + \frac{\delta}{8} = \frac{5\delta}{8}$.
Note in particular that $\psi_1(0)=y$.  Fix a $x \in K$ such that $\|y-x\| \le \eps$.
Let $t_0 \doteq \|x-y\|$ and define the function $\eta_{x,y} : [0,t_0] \rightarrow \R^d$ as 
\begin{equation}\label{eq:etaxy}
	\eta_{x,y}(t) \doteq x + \frac{(y-x)}{\|y-x\|}t.
\end{equation}
Define $\psi_2:[0,T]\to \R^d$ as
$$\psi_2(s) \doteq \eta_{x,y}(s) 1_{[0, t_0]}(s) + \psi_1(s-t_0)1_{(t_0, T]}(t).$$
Note that $\psi_2(0)=x$ and
\begin{align*}
	S_{\alpha,K'}(x,T,\psi_2) &= \int_0^{t_0} L_{\alpha, K'}(\psi_2(t),\dot{\psi}_2(t)) dt
	+ \int_{t_0}^T L_{\alpha, K'}(\psi_2(t),\dot{\psi}_2(t)) dt\\
	& \le \eps \kappa_0 + s - \frac{\delta}{4} = \frac{\delta}{8} +s - \frac{\delta}{4} \le s.
\end{align*}
Thus $\psi_2 \in \Phi_{x,\alpha,K',T}\left(s \right)$.
Furthermore
$$
\|\phi - \psi_2\|_{*,T} \le \|\phi - \psi_1\|_{*,T} + \|\psi_1 - \psi_2\|_{*,T} \le \frac{5\delta}{8} + \|\psi_1 - \psi_2\|_{*,T}.$$
Also, for $t \in (t_0, T]$,
$$\|\psi_1(t) - \psi_2(t)\| = \|\psi_1(t) - \psi_1(t- t_0)\| \le \frac{\delta}{8}$$
and for $t \in [0, t_0]$
$$
\|\psi_1(t) - \psi_2(t)\| \le \|\psi_2(t)-y\| + \|\psi_1(t) - \psi_1(0)\| \le \eps + \frac{\delta}{8}  \le \frac{\delta}{8} + \frac{\delta}{8} = \frac{\delta}{4}.$$
Thus
$$\|\phi - \psi_2\|_{*,T} \le \frac{5\delta}{8} + \frac{\delta}{4} \le \delta.$$
Since $\psi_2 \in \Phi_{x,\alpha,K',T}\left(s \right)$ , we  have 
$d(\phi, \Phi_{x,\alpha,K',T}\left(s \right)) \le \delta$ and thus $\phi$ is in the set on the left side of
\eqref{eq:eq351}. This proves the inclusion in \eqref{eq:eq351}.

 Now fix a closed set $F$ in $C([0,T]: \R^d)$. If $\inf\limits_{x\in K}\inf\limits_{\phi \in F}S_{\alpha,K'}(x,T,\phi) = 0$, the the result clearly holds, so we assume that $\bar S \doteq \inf\limits_{x\in K}\inf\limits_{\phi \in F}S_{\alpha,K'}(x,T,\phi)  > 0$.  Fix $s \in \left(0, \bar S \right)$ and let $\{x_n\} \subset K$ and $\eps \downarrow 0$.  Since $K$ is compact, we may pass to a subsequence and assume that $x_n \rightarrow \tilde{x}$ for some $\tilde{x} \in K$. Since $\inf\limits_{\phi\in F}S_{\alpha,K'}(\tilde{x}, T, \phi) > s$,  $F \cap \Phi_{\tilde{x},\alpha,K',T}(s) = \emptyset$. This, along with the facts that 
 $\Phi_{\tilde{x},\alpha,K',T}(s)$ is compact and $F$ is closed, ensures that there is some $\delta \in (0,s)$ such that 
 $$F \subset \{\phi \in C([0,T]:\R^d): d(\phi,\Phi_{\tilde{x},\alpha,K'}(s)) > \delta \}.$$
  Let $\eps = \eps(\delta) > 0$ be chosen as above \eqref{eq:eq351}.
  Without loss of generality we   assume that $\|\tilde{x} - x_n\| \le \eps$ for all $n$.
  Then, for every $n\in \N$,
$$
F \subset \{ \phi \in C([0,T]: \R^d : d(\phi, \Phi_{\tilde x,\alpha,K',T}(s)) > \delta\} \subset \left\{\phi \in C([0,T]: \R^d : d(\phi, \Phi_{x_n,\alpha,K',T}(s - \frac{\delta}{4})) > \frac{\delta}{2}\right\}.
$$

From the  upper bound in Theorem \ref{thm:fwuldp}(c) we  see that
\begin{equation*}
\begin{split}
\limsup\limits_{N\rightarrow\infty}\frac{1}{N}\log \PP_{x_N}(\hat{X}^{N, \alpha,K'} \in F) &\leq \limsup\limits_{N\rightarrow\infty}\frac{1}{N} \log \PP_{x_N}(d(\hat{X}^{N,\alpha,K'}, \Phi_{x_N,\alpha,K',T}(s - \frac{\delta}{4})) > \frac{\delta}{2})\\
&\leq - (s - \frac{\delta}{4}).\\
\end{split}
\end{equation*}
The result follows from letting $\delta \to 0$ and $s \to \bar S$.
\end{proof}


\section{Asymptotic Behavior of QSD}\label{sec:behnb}

In this section we assume that Assumptions \ref{assu:LLN}, \ref{assu:ap-classesfinite}, \ref{assu:mgf} and \ref{assu:irrbdr} are satisfied.
Using these assumptions we will provide several exponential probability estimates and use them to deduce some asymptotic properties of the QSD  $\{\mu_N\}$ (when they exist).
 Let for $N\in \N$ and $T\in (0,\infty)$
\[
D^N_T \stackrel{\cdot}{=} \sup\limits_{0\leq t \leq T}\|\hat{X}^N(t) - \varphi_t(X_0^N)\| = \|\hat X^N - \varphi_{\cdot}(X^N_0)\|_{*,T}.
\]

The estimates obtained in Lemma \ref{lem:expineq} and Lemma \ref{lem:eigenlemma} are the key steps in the proof of Theorem \ref{thm:eigentheorem} which gives the asymptotics of $\la_N \doteq [\PP_{\mu_N}(X^N_1 \in \Delta^o)]^N$, where $\mu_N$ is a QSD for $\{X^N\}.$ 
Recall the definition of $\clv_{\alpha}$ from Section \ref{sec:ldest}.
\begin{lemma}
	\label{lem:expineq}
	
	For each $ \alpha > 0$, compact set $K \subset \clv_{\alpha}$,
	  $\eps > 0$, and $T > 0$, there is a $c\in (0,\infty)$ and $N_0 \in \N$ such that for every  $N\ge N_0$,
	\[
	\sup\limits_{x \in K \cap \Delta_N} \PP_x\left[ D^N_T \geq \eps \right] \leq \exp(-Nc).
	\]
\end{lemma}

\begin{proof}
	Let $\alpha > 0$ and $K \subset V_{\alpha}$ be compact. For each  $\eps \in (0,\alpha)$, let
	$$
	F_{\eps} = \left\{ \psi \in C([0,T] : \R^d) : \sup\limits_{0\leq t \leq T}\| \psi(t)- \varphi_t(\psi(0))\| \geq \eps \right\}.
	$$  
	Using Lemma \ref{lem:bdryrepel} we can (and will) assume without loss of generality that $\eps$ is small enough so that the  compact set
	$$
	K' \doteq \overline{N^{\eps}\left( \varphi_{[0,\infty)}(K)\right)} \subset \Delta^o.
	$$
	Note that
	$$
	\sup\limits_{x \in K \cap \Delta_N}\PP_x\left[ D^N_T \geq \eps\right]  = \sup\limits_{x \in K \cap \Delta_N}\PP_x\left[ \sup\limits_{0\leq t \leq T} 
	\| \hat{X}^{N, \alpha, K'}(t) - \varphi_t(\hat{X}^{N, \alpha, K'}(0))\| \geq \eps\right].
	$$
	Since $F_{\eps}$ is closed,   Theorem \ref{thm:dzuldp} says that for each $\delta>0$, there is a $N_{\delta}\in \N$ such that for all $N\ge N_{\delta}$	
	\begin{equation*}
	\begin{split}
	\log \sup\limits_{x \in K \cap \Delta_N}\PP_x\left[ \sup\limits_{0\leq t \leq T} \| \hat{X}^{N,\alpha,K'}(t) - \varphi_t(x)\| \geq \eps\right] 
	&= \log \sup\limits_{x\in K \cap \Delta_N}  \PP_x(\hat{X}^{N, \alpha,K'} \in F_{\eps})\\
	&\leq - N \left[\inf\limits_{x \in K }\inf\limits_{\psi \in F_{\eps}} S_{\alpha,K'}(x,T,\psi) - \delta\right].
	\end{split}
	\end{equation*}
	In order to prove the result it suffices to show that $\inf\limits_{x \in K }\inf\limits_{\psi \in F_{\eps}} S_{\alpha,K'}(x,T,\psi) >0$.
	Arguing by contradiction, suppose that this infimum is $0$. Then, there are sequences $\{x_n\} \subset K$ and $\{\psi_n\} \subset C([0,T]: \R^d)$ such that $\psi_n \in F_{\eps}$ for each $n$ and $\lim\limits_{n\rightarrow\infty} S_{\alpha,K'}(x_n,T,\psi_n) = 0$. Since $S_{\alpha,K'}(x, T, \phi) < \infty$ if and only if $\phi(0) = x$, we can assume without loss of generality that  $x_n = \psi_n(0)$ for every $n$. For each $\eps' > 0$, 
	\[
	\psi_n \in \{\phi \in C([0,T]: \R^d): S_{\alpha,K'}(y,T,\phi) \le \eps' \mbox{ for some } y \in K\}
	\]
	whenever $n$ is sufficiently large. Since $K$ is compact, Theorem \ref{thm:fwuldp} ensures that 
	$\{\psi_n\}$ is precompact in $C([0,T]: \R^d)$ and
	 so there is a convergent subsequence of $\{\psi_n\}$. Denoting this  subsequence by $\{\psi_{n_k}\}$ and its limit by $\psi$, 
	 we have that  
	 \[
	 \lim\limits_{k\rightarrow\infty}(\psi_{n_{k}}, x_{n_{k}}) = \lim\limits_{k\rightarrow\infty}(\psi_{n_{k}}, \psi_{n_{k}}(0))  = (\psi, \psi(0)).
	 \]
	Since $\phi \mapsto S_{\alpha,K'}(\phi(0), T, \phi)$ is lower semi-continuous, it follows that
	\[
	S_{\alpha,K'}(\psi(0), T, \psi) \leq \lim\limits_{k\rightarrow\infty} S_{\alpha,K'}(\psi_{n_k}(0), T, \psi_{n_{k}}) = 0,
	\]
	which says that $\psi(t) = \varphi_t(\psi(0))$. However,  this is a contradiction, since $\psi \in F_{\eps}$. The result follows.
\end{proof}

\begin{lemma}\label{lem:eigenlemma}
Let $U$ be a fundamental neighborhood of an attractor $A \subset \Delta^o$ such that $\bar U \subset \Delta^o$.
Then for every $T_0 \in (0,\infty)$, there are  $c_0 \in (0,\infty)$,  $T \in (T_0,\infty)$ and $N_0 \in \N$ such that 
$$
\sup\limits_{x \in U \cap \Delta_N}\PP_x( X_{\lfloor NT \rfloor}^N \in U^c) \leq  \exp\left(-c_0N\right)
$$
for all  $N \ge N_0$.
\end{lemma}

\begin{proof}
Let $\alpha \doteq \dist(A, U^c)$. Since $U$ is a fundamental neighborhood of the attractor $A$,  we can find $T > T_0$ such that $\sup\limits_{t \geq T}\sup\limits_{x \in U}\dist(\varphi_t(x), A) < \alpha/2$.  Let $K \in \clk$ be a compact set containing $U$.
From Lemma \ref{lem:bdryrepel} there exists a $\sigma \in (0, \alpha/4)$ 
and a  $K' \in \clk$ such that 
$\overline{N^{\sigma}(\gamma^+(U))} \subset K'$, where $\gamma^+(U) = \cup_{x\in U}\gamma^+(x)$.
Then for each $x \in U \cap \Delta_N$, we have
\begin{equation}\label{eq:twoterms}
\begin{split}
\PP_x( X_{\lfloor NT \rfloor}^N \in U^c)  &\leq \PP_x(\dist(X_{\lfloor NT \rfloor}^N, A) > \alpha)
\leq \PP_x( \|X_{\lfloor NT\rfloor}^N - \varphi_T(x)\| > \alpha/2)\\
&\leq \PP_x(\|X_{\lfloor NT \rfloor}^N -\hat{X}^N(T)\| + \|\hat{X}^N(T)- \varphi_T(x)\| > \alpha/2)\\
&\leq  \PP_x\left(D_T^N > \sigma \right)+  \PP_x\left( \|X_{\lfloor NT \rfloor}^N- \hat{X}^N(T)\|  > \alpha/4, D_T^N \le \sigma\right).\\
\end{split}
\end{equation}
Using the Markov property, we have
\begin{equation*}
\begin{split}
	\PP_x\left( \|X_{\lfloor NT \rfloor}^N- \hat{X}^N(T)\|  > \alpha/4, D_T^N \le \sigma\right)
	&\le \PP_x\left( \|X_{\lfloor NT \rfloor}^N- X_{\lfloor NT \rfloor + 1}^N\|  > \alpha/4, D_T^N \le \sigma\right)\\
	&\le \sup\limits_{x\in K'\cap \Delta_N}\PP_x\left( \|X^N_1-x\|  > \alpha/4\right).
\end{split}
\end{equation*}

From Assumption \ref{assu:mgf} we have that for every $\lambda>0$
$$C(\lambda) \doteq \sup_{N \in \NN} \sup_{x \in K' \cap \Delta_N} \EE_x(e^{\lambda N \|X^N_1-x\| }) < \infty.$$
Thus for any $\lambda >0$
$$\sup\limits_{x\in K'\cap \Delta_N}\PP_x\left( \|X^N_1-x\|  > \alpha/4\right) \le c(\lambda)e^{-\lambda N \alpha/4}.$$
The result follows on using the above estimate and Lemma \ref{lem:expineq} in \eqref{eq:twoterms}.
\end{proof}

 The following lemma says that for every open $U \subset \Delta^o$, the support of $\mu_N$ (when it exists) has a nonempty intersection with $U$ when $N$ is sufficiently large.

\begin{lemma}
	\label{lem:posmzr}
	Suppose that for each $N \in \NN$, $X^N$ has a QSD $\mu_N$.
Then for each open $U \subset \Delta^o$, there is some $N_0 \in \NN$ such that $\mu_N(U) > 0$ for all $N \geq N_0$.
\end{lemma}

\begin{proof}
Let $N_0$ be large enough so that  $U \cap \Delta^o_N$ is nonempty for all $N\ge N_0$.  Fix $N \geq N_0$, $x \in U \cap \Delta^o_N$  and $w \in \Delta^o_N$ with $\mu_N(w)>0$.
From Assumption \ref{assu:irrbdr}(a), there is a $k \in \NN$ such that $\PP_w(X^N_k=x)>0$. Then
\begin{equation*}
\begin{split}
\mu_N(U) &\geq \mu_N(x) = \frac{\sum\limits_{y \in \Delta_N^o} \mu_N(y)\PP_y(X_k^N = x)}{ \sum\limits_{z\in \Delta^o_N}\sum\limits_{y \in \Delta_N^o} \mu_N(y)\PP_y( X_k^N = z)} \geq  \frac{ \mu_N(w)\PP_w(X_k^N = x)}{ \sum\limits_{z\in \Delta^o_N}\sum\limits_{y \in \Delta_N^o} \mu(y)\PP_y(X_k^N = z)} > 0.
\end{split}
\end{equation*}
\end{proof}

The following lemma quantifies the asymptotic behavior of the sequence $\{\lambda_N\}$ introduced in \eqref{eq:lan}.

\begin{theorem}\label{thm:eigentheorem}
Suppose that for each $N \in \NN$, $X^N$ has a QSD $\mu_N$. 
Then there exist $c, c' \in (0,\infty)$ such that for all $N\in \N$
$$
0 \leq 1 - \lambda_N \leq c' \exp(-cN). 
$$
\end{theorem}

\begin{proof}
From Assumption \ref{assu:ap-classesfinite} and Corollary \ref{cor:quasiattract} there exists an  attractor  $A$  in $\Delta^o$. Let $U \subset \Delta^o$ be a fundamental neighborhood of $A$.
From Lemma \ref{lem:eigenlemma}  there are  $c_0 \in (0,\infty)$ and $T,N_1 \in \N$  such that  for all $N\ge N_1$
\begin{equation}\label{eq:ucomp}
\sup\limits_{x \in U \cap \Delta_N}\PP_x( X_{NT}^N \in U^c) \leq  \exp\left(-c_0N\right).
\end{equation}
From Lemma \ref{lem:posmzr} there is a $N_2 \in \N$ such that for all $N\ge N_2$,
 $\mu_N(U) > 0$. Fixing $N \geq N_1\vee N_2$, we have

\begin{equation*}
\begin{split}
\lambda^{T}_N \mu_N(U) &= \sum\limits_{x \in \Delta^o_N}\PP_{x}(X^N_{NT} \in U)\mu_N(x)\\
&\geq \sum\limits_{x \in U \cap \Delta^o_N}\PP_x(X_{NT}^N \in U)\mu_N(x)\\
&\geq \inf\limits_{x\in U \cap \Delta^o_N}\PP_x(X_{NT}^N \in U) \sum\limits_{x \in U \cap \Delta^o_N}\mu_N(x)\\
&=\inf\limits_{x\in U \cap \Delta^o_N}\PP_x(X_{NT}^N \in U)\mu_N(U).
\end{split}
\end{equation*}
Thus for all $N\ge N_1\vee N_2$
$$
\la_N\ge \lambda^{T}_N \geq \inf\limits_{x\in U \cap \Delta^o_N}\PP_x(X_{NT}^N \in U) = 1 - \sup\limits_{x\in U \cap \Delta^o_N}\PP_x(X_{NT}^N \in U^c)
\ge 1-\exp\left(-c_0N\right),
$$
where the last inequality uses \eqref{eq:ucomp}.
The result follows.
\end{proof}

For $\delta>0$, $T \in \N$, and  $K\in \clk$, let 
\begin{equation}\label{eq:betadlnt}
\beta_{\delta, K}^N(T) \doteq \sup_{x \in \Delta_N\cap K} \PP_x[\|\hat X^N - \varphi_{\cdot}(x)\|_{*,T} \ge \delta].\end{equation}

The following lemma gives a different lower bound on $\lambda_N$. This bound will be needed in the proof of Theorem
\ref{thm:abs} below.  
\begin{lemma}\label{lem:eigenlemma2}
	Suppose that for each $N \in \NN$, $X^N$ has a QSD $\mu_N$.
Let $A$ be an attractor in $\Delta^o$, $\tilde U \subset \Delta^o$ be an open set containing $A$, and    $K \in \clk$ be such that $\tilde U\subset K$. Then  there is some $\delta > 0$ and 
$T, N_0 \in \N$ such that $\lambda_N^{T} \geq 1 - \beta_{\delta,K}^N(T)$ for each $N \ge N_0 $.
\end{lemma}

\begin{proof}
	Since $A$ is an attractor, there is a fundamental neighborhood $U$ of $A$ contained in $\tilde U$.
	Thus we can find a $\delta>0$ and $T \in \N$ such that $N^{\delta}(\varphi_T(\overline{U})) \subset U$.
	From Lemma \ref{lem:posmzr} we can find a $N_0 \in \N$ such that $\mu_N(U)>0$ for all $N \ge N_0$.
Following the proof of Theorem \ref{thm:eigentheorem}, we  see that
\begin{equation*}
\begin{split}
\lambda_N^T \mu_N(U) \geq \left(1 - \sup\limits_{x\in U \cap \Delta^o_N}\PP_x(X^N_{NT} \in U^c)\right)\mu_N(U).
\end{split}
\end{equation*}
From our choice of $U$ and $\delta$ it now follows that
\begin{equation*}
\begin{split}
\lambda_N^T &\geq 1 - \sup\limits_{x\in U \cap \Delta^o_N}\PP_x(X^N_{NT} \in (N^{\delta}(\varphi_T(\overline{U})))^c) \geq 1 - \beta^N_{\delta,K}(T).\\
\end{split}
\end{equation*}
\end{proof}

	
A key consequence of the following theorem is that the support of any weak limit point of $\mu_N$ is contained in $\Delta^o$. This, along with a further characterization of the support of such weak limit points given in Corollary \ref{cor:limitsupportrap}, is a key element in the proof of Theorem \ref{thm:main}.  

\begin{theorem}\label{thm:abs}
	Suppose that for each $N \in \NN$, $X^N$ has a QSD $\mu_N$.
	Then, for every $\delta>0$, $T\in \N$ and  $K \in \clk$, there exists an open neighborhood $V_K$ of $\partial \Delta$ in $\Delta$ such that
	\begin{equation}\label{eq:2433}
		\limsup_{N\to \infty} \mu_N(V_K) \le \limsup_{N\to \infty} \frac{\beta_{\delta, K}^N(T) }{\inf_{x \in V_K \cap \Delta_N} \PP_x[\hat X^N(T) \in \partial \Delta]}
	=0.\end{equation}
	Suppose in addition that $\mu_N$ converges along some subsequence to some probability measure $\mu$ on $\Delta$.
Then, there is an open neighborhood $V_0$ of $\partial \Delta$ in $\Delta$ such that,  
$\mu(V_0) = 0$.
\end{theorem}

\begin{proof}
	Fix $\delta, T,K$ as in the statement of the theorem. 
Let $\delta_0 \doteq \frac{1}{2} \inf\limits_{t \in [0,T]} \inf\limits_{x\in K}\dist(\varphi_t(x), \partial\Delta)$ and
 let  $K' \doteq \overline{ N^{\delta_0}( \varphi_{[0,T]}(K))}$ and consider the closed set
   %
  %
$$
F \doteq \left\{\phi \in C([0,T]:\R^d) : \|\phi  -\varphi_{\cdot}(\phi(0))\|_{*,T} \geq \delta_0\right\}.
$$
Fix $\alpha \in (0,\delta_0)$ and $K_1 \in \clk$ that contains some open neighborhood of $K'$.
Then from Theorem \ref{thm:dzuldp}
\begin{equation*}
\begin{split}
\limsup\limits_{N\rightarrow\infty} \frac{1}{N} \log 
\sup_{x \in K \cap \Delta_N}
 \PP_x (\hat{X}^{N} \in F)  &= \limsup\limits_{N\rightarrow\infty} \frac{1}{N} \log 
 \sup_{x \in K \cap \Delta_N} \PP_x (\hat{X}^{N,\alpha,K_1} \in F)\\
&\leq - \inf_{x \in K} \inf_{\phi \in F} S_{\alpha,K_1}(x,T,\phi) \doteq -c(K).
\end{split}
\end{equation*}
Clearly $c(K)>0$.
From Assumption \ref{assu:irrbdr}(b) we can find an open neighborhood $V_K$ of $\partial \Delta$ 
such that
$$
\liminf\limits_{N\rightarrow\infty}\inf_{x\in V_K\cap \Delta_N}\frac{1}{N} \log \PP_x ( \hat{X}^N(T) \in \partial \Delta) \geq - c(K)/4.
$$
Combining last two displays, we can find  a $N_1 \in \N$ such that for all $N \ge N_1$
\begin{equation*}
\begin{split}
	\frac{\beta_{\delta, K}^N(T) }{\inf_{x \in V_K \cap \Delta_N} \PP_x[\hat X^N(T) \in \partial \Delta]} = 
\frac{\sup\limits_{x \in \Delta_N\cap K} \PP_x\left(\| \hat X^N  - \varphi_{\cdot}(x)\|_{*,T} \ge \delta\right)}{\inf\limits_{x \in V_K \cap \Delta_N} \PP_x[\hat X^N(T) \in \partial \Delta]} \leq \exp(-Nc(K)/2),
\end{split}
\end{equation*}
which converges to $0$ as $N \rightarrow\infty$. 
This proves the last equality in \eqref{eq:2433}.

Next, from Assumption \ref{assu:ap-classesfinite} and Corollary \ref{cor:quasiattract} there exists an  attractor  $A$  in $\Delta^o$.
Let
$\tilde U \in \Delta^o$ be an open set containing $A$, and    $K \in \clk$ be such that $\tilde U\subset K$. 
Then, from Lemma \ref{lem:eigenlemma2}
there is some $\delta > 0$ and 
$T, N_0 \in \N$ such that 
$$
	\lambda_N^{T} \geq 1 - \beta_{\delta,K}^N(T) \mbox{ for each } N \ge N_0.
$$
Since $\mu_N(\Delta^o_N) = 1$, we have, with $V_K$ given as in the first part of the theorem,
\begin{equation*}
\begin{split}
1 - \beta_{\delta,K}^N(T) &\leq \lambda_N^T  \mu_N(\Delta^o_N)\\
&= \sum\limits_{x\in \Delta^o_N}\left(1 - \PP_x( X_{NT}^N \in \partial \Delta)\right)\mu_N(x)\\
&= \sum\limits_{x \in V_K\cap \Delta^o_N}\left(1 - \PP_x( \hat{X}^N(T) \in \partial \Delta)\right)\mu_N(x) + \sum\limits_{x \in  \Delta^o_N \setminus V_K}\left(1 - \PP_x( \hat{X}^N(T) \in \partial \Delta)\right)\mu_N(x)\\
&\leq \left(1 - \inf\limits_{x\in V_K \cap \Delta^o_N}\PP_x(\hat{X}^N(T)\in \partial \Delta) \right)\mu_N(V_K)  +\mu_N(\Delta^o_N \setminus V_K)\\
&= 1 - \inf\limits_{x\in V_K \cap \Delta^o_N}\PP_x(\hat{X}^N(T) \in \partial \Delta)\mu_N(V_K).
\end{split}
\end{equation*}
Rearranging the previous inequality, we obtain 
$$
\mu_N(V_K) \leq \frac{\beta^N_{\delta,K}(T)}{\inf\limits_{x\in V_K \cap \Delta^o_N}\PP_x(\hat{X}^N(T) \in \partial \Delta)}.
$$
This proves the first inequality in \eqref{eq:2433}.

Finally, let $V_0$ be a open neighborhood of $\partial \Delta$ such that $\bar V_0 \subset V_K$. From the first part of the theorem, taking the limit along the convergent subsequence,
$$\mu(V_0)  \le \liminf_{N\to \infty} \mu_N( V_0) \le \liminf_{N\to \infty} \mu_N(V_K) = 0.$$
The result follows.
\end{proof}

The following theorem proves the invariance of $\mu$ under the flow $\{\varphi_t\}$.
\begin{theorem}\label{thm:invariance}
	Suppose that for each $N \in \NN$, $X^N$ has a QSD $\mu_N$ and suppose that $\mu_N$ converges along some subsequence to some probability measure $\mu$. Then $\mu$ is invariant under $\{\varphi_t\}$. In particular, $\mu(\varphi_t^{-1}(B)) = \mu(B)$ for each measurable set $B \subset \Delta$ and $t \geq 0$.
\end{theorem}

\begin{proof}
	From Corollary  \ref{cor:quasiattract}, for each $i \in \{1,\dots,l\}$  $K_i$ is an attractor. Fix $1\le i \le l$, $\delta > 0$ and  $K \in \clk$ such that $N^{\gamma}(K_i ) \subset K$ for some $\gamma > 0$. Let $\beta_{\delta,K}^N $ be as in \eqref{eq:betadlnt}.
  It suffices to show that for any continuous and bounded $f: \Delta \to \RR$ and $t > 0$, $\mu(f) = \mu(f\circ\varphi_t)$.
  Fix $f$ and $t$ as above and let $\eps>0$ be arbitrary.
  Using the fact that $\{\mu_N\}$ (considered along the convergent subsequence) is tight, we can 
  assume that the $K$ chosen above satisfies
  $$\sup_{N\ge 1}\mu_N(K^c) \le \frac{\eps}{2\|f\|_{\infty}}.
  $$
  Let $t_N = \lfloor Nt\rfloor/N$. Note that $t_N \to t$ as $N \to \infty$ and from Theorem \ref{thm:eigentheorem} $\lambda_N^{t_N}\to 1$ as $N\to \infty$.
  For a bounded $g: \Delta^o \to \RR$ and $k \in \NN$, let
  $$\clp^N_k f(x) \doteq \EE_x[f(X^N_k); \tau^N_{\partial} >k], \; x \in \Delta^o_N,$$
  and
  $$\clp^{N,*}_k f(x) \doteq \EE_x[f(X^N_k)], \; x \in \Delta^o_N.$$  
  Then
  $$\mu_N(f) = \la_N^{-t_N} \mu_N(\clp^N_{\lfloor N t\rfloor} f).$$
  In particular, as $N\to \infty$,
  $$|\mu_N(f)-\mu_N(\clp^N_{\lfloor N t\rfloor} f)| \le \|f\|_{\infty} |1-\la_N^{t_N}| \to 0.$$
  Also,
  $$
  |\mu_N(\clp^N_{\lfloor N t\rfloor} f)-\mu_N(f\circ\varphi_t)| \le \frac{\eps}{2\|f\|_{\infty}} 2 \|f\|_{\infty} + \sup_{x\in K}\left|\clp^N_{\lfloor N t\rfloor} f(x) - f\circ\varphi_t(x)\right|.
  $$
  
  
  For each $x \in K \cap \Delta_N$,
  \begin{equation*}
      \begin{split}
        \left| \clp^N_{\lfloor N t \rfloor} f(x) - f \circ \varphi_t(x)\right| &\leq \left|   \clp^{N,*}_{\lfloor N t \rfloor} f(x) -  \clp^{N}_{\lfloor N t \rfloor} f(x) \right| + \left|  \clp^{N,*}_{\lfloor N t \rfloor} f(x) - f \circ \varphi_t(x)\right|\\
          &\leq ||f||_{\infty}\PP_x( \tau^N_{\partial} \leq \lfloor N t \rfloor) + \left|  \clp^{N,*}_{\lfloor N t \rfloor} f(x) - f \circ \varphi_t(x)\right|,
      \end{split}
  \end{equation*}
  and Assumption  \ref{assu:LLN} ensures that as $N \rightarrow \infty$,
  $$
  \sup\limits_{x\in K} \left|  \clp^{N,*}_{\lfloor N t \rfloor} f(x) - f \circ \varphi_t(x)\right| \rightarrow 0.
  $$
  Let $\tilde{\delta} \doteq \inf_{x\in K, 0 \leq s \leq t}\dist(\varphi_{s}(x), \partial \Delta) > 0$, and note that Assumption \ref{assu:LLN} ensures that as $N\rightarrow\infty$,
  \begin{equation*}
  \begin{split}
  \sup\limits_{x \in K}\PP_x(\tau_{\partial}^N \leq \lfloor N t\rfloor) \leq \sup\limits_{x\in K}\PP_x( \|X^N - \varphi_{\cdot}(x)\|_{*,t} > \tilde{\delta}) \rightarrow 0.
  \end{split}
  \end{equation*}
 
  Combining the two previous convergence properties, we see that as $N \rightarrow \infty$,
  $$
  \left| \clp^N_{\lfloor N t \rfloor} f(x) - f \circ \varphi_t(x)\right| \rightarrow 0,
  $$
  and therefore that 
  $$|\mu(f)- \mu(f \circ \varphi_t)| \le \limsup_{N\to \infty} |\mu_N(f)- \mu_N(f \circ \varphi_t)| \le \eps.$$
  Since $\eps>0$ is arbitrary, the result follows.
 \end{proof}
 
 We now recall the definition of the Birkhoff center of $\{\varphi_t\}$.
 \begin{defn}
 	\label{defn:birkctr}
 The Birkhoff center of $\{\varphi_t : t\geq 0\}$ is 
 \[
 BC(\varphi)  \doteq \overline{ \{x \in \Delta : x \in \omega(x)\}}.
 \]
 \end{defn}

 \begin{lemma}\label{lem:birkhoffcenter}
 The Birkhoff center of $\{ \varphi_t : t \geq 0\}$ is contained in the closure of $\clr_{\ap}$.
 Furthermore $BC(\varphi)\cap \Delta^o  \subset \clr_{\ap}^*$.
 \end{lemma}
 \begin{proof}
 Let $\delta, T > 0$ and suppose that $x \in \omega(x)$. There is a sequence of time instants $t_i \uparrow \infty$ such that $\varphi_{t_i}(x) \rightarrow x$, so if we let
 \[
 j = \min\{ i : t_i > T \text{ and } \|\varphi_{t_i}(x) -  x\| < \delta\},
 \]
 then $(x, x, \phi_{t_j}(x), x, x)$ is a $(\delta,T)$ \ap--pseudo-orbit from $x$ to $x$. Since $\delta, T$ are arbitrary, $x \in \clr_{\ap}$. This proves the first part of the lemma. The second part is now immediate on using Assumption \ref{assu:ap-classesfinite}(a).
 \end{proof}

We will use  the Poincar\'{e} recurrence theorem given below. For a proof  see \cite[Theorem 4.1.19]{hofkat}.

\begin{theorem}\label{thm:poincare}
Let $\nu$ be a measure which is invariant under $\{\varphi_t\}$. Then for each measurable $B \subset \Delta$ and $T > 0$,
$$
\nu( \{ x \in B : \{\varphi_t(x)\}_{t\geq T} \subset \Delta \setminus B\}) = 0.
$$
\end{theorem}

The next result is a consequence of Lemma \ref{lem:birkhoffcenter} and Theorem \ref{thm:poincare}.
It shows that the support of $\mu$ is contained in $\clr_{\ap}^*$.
\begin{corollary}\label{cor:limitsupportrap}
Suppose that for each $N \in \NN$, $X^N$ has a QSD $\mu_N$ and suppose that $\mu_N$ converges along some subsequence to some probability measure $\mu$. Then $\operatorname{supp}(\mu) \subset \mathcal{R}_{\ap}^*$.
\end{corollary}
\begin{proof}
	From Theorem \ref{thm:invariance}  $\mu$ is invariant under $\{\varphi_t\}$. 
Enumerate the $d$-dimensional rationals in $\Delta$ as $\mathbb{Q}^d \doteq \{q_1,q_2,\dots\}$ and for $m,n \in \NN$, denote the ball of radius $n^{-1}$ centered at $q_m$ by $B(q_m,n^{-1})$. Then for each $m,n \in \NN$, Theorem \ref{thm:poincare} says that
$$
\mu(\tilde{B}(q_m,n^{-1})) = \mu(B(q_m,n^{-1})),
$$
where
$$
\tilde{B}(q_m,n^{-1}) \doteq \{ x \in B(q_m,n^{-1}) : \mbox{ there exist } t_k \uparrow \infty \mbox{ with } \varphi_{t_k}(x) \in B(q_m,n^{-1}) \text{ for all }k \in \NN \}.
$$
Let $R \doteq \cap_{n=1}^{\infty}\cup_{m=1}^{\infty} \tilde{B}(q_m,n^{-1})$, then
$$1 = \mu(\cap_{n=1}^{\infty}\cup_{m=1}^{\infty} B(q_m,n^{-1})) = \mu(\cap_{n=1}^{\infty}\cup_{m=1}^{\infty} \tilde B(q_m,n^{-1}))  = \mu( R),$$
which together with Theorem \ref{thm:abs} implies that $\text{supp}(\mu) \subseteq \overline{R} \cap \Delta^o$. Furthermore if $x \in R$, then $x \in \omega(x)$, so $R \subseteq \text{BC}(\varphi)$ and consequently
$\bar{R} \cap \Delta^o \subset \text{BC}(\varphi) \cap \Delta^o$. It now follows from Lemma \ref{lem:birkhoffcenter} that
$$
\text{supp}(\mu) \subseteq \overline{R}\cap \Delta^o \subseteq \text{BC}(\varphi)\cap \cap \Delta^o \subseteq \mathcal{R}_{\ap}^*.
$$
\end{proof}

Combining the results of Corollary \ref{cor:limitsupportrap}, Theorem \ref{thm:invariance} and Theorem \ref{thm:eigentheorem} we have most of Theorem \ref{thm:main}. In particular  we have the lower bound on probabilities of non-extinction given in Theorem 
\ref{thm:main} and also that the limit points $\mu$ of the QSD are invariant under the flow,  and  that they are supported on the union of absorption preserving recurrence classes in the interior. The final step is to show that the support in fact lies in the union of the interior attractors. For this  we will introduce another notion of recurrence which is given in terms of the quasipotential associated with the rate functions in the underlying large deviation principles. 

\section{Quasipotential and Chain-recurrence}\label{sec:vrec}
In this section we suppose that Assumptions \ref{assu:LLN}, \ref{assu:ap-classesfinite} and, \ref{assu:mgf} are satisfied.
Recall the rate function $S$ introduced in \eqref{eq:ratefns}. For  $x,y \in \Delta^o$ 
let $\clc(x,y,T) \doteq \{\phi\in C([0,T]:\Delta^o):\phi(0)=x, \phi(T)=y\}$ and define
\begin{equation}\label{eq:vxy}
V(x,y) \doteq \liminf_{T\to \infty} \inf_{\phi \in C(x,y,T)}
S(x,T,\phi).\end{equation}
For $x,y \in \Delta^o$, we say $x$ leads to $y$ in  $ \Delta^o$ if $V(x,y) =0$ and we write $x<_{V} y$.
We say $x \in \Delta^o$ is $V$-chain recurrent if $x <_{V} x$. The collection of $V$-chain recurrent
points is denoted as $\clr_V$. For $x,y \in \Delta^o$ we say $x\sim_{V} y$ if
$x<_{V} y$ and $y<_{V} x$. Equivalence classes under $\sim_{V}$ will be called $V$-basic classes and equivalence class 
associated with a $x \in \clr_V$ will be denoted as $[x]_V$. For $x, y \in \clr_V$ we say $[x]_V \prec [y]_V$ if
$x <_{V} y$. A $V$-basic class $[x]_V$ is said to be maximal if whenever for $y \in \clr_V$,  if $[x]_V \prec [y]_V$, then we have that $y \in [x]_V$. A maximal $V$-basic class is a called a $V$-quasiattractor.
The following is the main result of this section.
\begin{theorem}\label{thm:basicclasses}
	We have  $
	\mathcal{R}_{\ap}^*  = \mathcal{R}_{V}
	$ and
for each $x \in \mathcal{R}_{V}$,
$
[x]_{V} = [x]_{\ap}.
$
In particular there are finitely many $V$-chain recurrent
	points  and for every $x \in \clr_{V}$, $[x]_V$ is a closed set.  Furthermore, $K_i$ for $i=1, \ldots, l$ is a $V$-quasiattractor while $K_i$ for $i=l+1, \ldots, v$ is not  a $V$-quasiattractor.
\end{theorem}

Before proving Theorem \ref{thm:basicclasses} we will establish some basic results regarding $V(\cdot,\cdot)$ and $\mathcal{R}_{V}$.
The following lemma is a consequence of the stability properties of the ODE \eqref{eq:dynsys} studied in Lemma \ref{lem:bdryrepel} and the property that low cost trajectories closely follow the solution of the ODE. 

\begin{lemma}\label{lem:boundedflow}
	Let $\alpha \in (0, \infty)$ and $K \in \clk$. Let $T_0\in (0,\infty)$ and suppose
 $T_n \in [T_0,\infty)$ for all $n \in \N$. Let $\phi_n \in C([0,T_n] : \Delta^o)$ be such that $\phi_n(0) \in \clv_{\alpha, K}$ for each $n \geq 1$. 
 Suppose that $S(\phi_n(0), T_n, \phi_n) \rightarrow 0$ as $n\to \infty$. With $\alpha_0$ and $M_0$ as in Lemma \ref{lem:bdryrepel},
 let $\alpha_1 = \frac{\alpha}{2} \wedge \alpha_0$ and $K_1 = \overline{B_{M_1}(0)}$, where $M_1= 1+ (M_0 \vee \sup_{x\in K}\|x\|)$.
 Then, 
for some $k \geq 1$, 
 $\phi_n(t) \in \clv_{\alpha_1, K_1}$ for all $n\ge k$ and $t\in [0,T_n]$.
\end{lemma}

\begin{proof}
Let for $n\ge 1$,
$$
\tau(\phi_n) \doteq \inf\{t \in [0,T_n]: \dist(\phi_n(t), \partial \Delta) \le \alpha_1 \mbox{ or } \|\phi_n(t)\| \ge M_1\},$$
where the infimum is taken to be $T_n$ if the above set is empty.
Note that the result holds trivially if the above set is empty for all but finitely many $n$.  Now, arguing by contradiction, suppose the set is nonempty for infinitely many $n$. 
Consider the subsequence along which the above sets are nonempty and denote the subsequence once more as $\{n\}$. Also assume without loss of generality that $\gamma_n\doteq S(\phi_n(0), T_n, \phi_n) \le 1$ for every $n$.

We claim that there is a $\delta>0$ and $k_0\in \N$ such that $\tau(\phi_n) \ge \delta$ for all $n \ge k_0$.
Indeed, otherwise, by passing to a further subsequence (once more denoted as $\{n\}$) we can find a sequence $\delta_n \to 0$ such that for every $n$
$$\phi_n(0) \in \clv_{\alpha, K},\; \phi_n(\delta_n) \in [\clv_{\alpha_1, K_1}]^c.$$
Since $S(\phi_n(0), T_n, \phi_n)\le \gamma_n\le 1$, we must have from the compactness of level sets property in Theorem  \ref{thm:fwuldp} that $\phi_n(0)$ and $\phi_n(\delta_n)$ converge along a subsequence to the same limit,
which is a contradiction.

Let $\delta > 0$ be such that $\tau(\phi_n) \geq \delta$ for all sufficiently large $n$. For each such $n$ let $\hat \tau_n \doteq \tau(\phi_n)-\delta$,
 and define
$\phi_n^*: [0,\delta] \to \Delta^o$ as $\phi_n^*(t) = \phi_n(t+\hat \tau_n)$, $t \in [0,\delta]$.
Then $\phi_n^* \in C([0,\delta]: \clv_{\alpha_1, K_1})$.
Also,
\begin{align}\label{eq:eq1206}
	\int_0^{\delta} L_{\alpha_1, K_1}({\phi}_n^*(t), \dot{\phi}_n^*(t))dt &= \int_0^{\delta} L({\phi}_n^*(t), \dot{\phi}_n^*(t))dt\nonumber\\
	&\leq \int_0^{T_n}L(\phi_n(t), \dot{\phi}_n(t))dt = \gamma_n \le 1.
\end{align}
In particular, $\{\phi_n^*\} \subset \cup_{x\in K_1}\Phi_{x,\alpha_1,K_1,\delta}(1)$. From Theorem \ref{thm:fwuldp} the latter set is compact and so, along some subsequence, $\phi_n^*$ converges to some $\phi^*$ in 
$C([0,\delta]: \clv_{\alpha_1, K_1})$. Using the compactness of level sets property again, we have from \eqref{eq:eq1206} and the fact that $\gamma_n\to 0$, that
$$\int_0^{\delta} L_{\alpha_1, K_1}({\phi}^*(t), \dot{\phi}^*(t))dt = 0.$$
In particular, $\phi^*(t)$ solves the ODE \eqref{eq:dynsys}, namely $\phi^*(t) = \varphi_t(\phi^*(0))$ for $t\in [0,\delta]$.
Since $\phi^*(0) \in \clv_{\alpha_1, K_1}$, in view of Lemma \ref{lem:bdryrepel}, we must have that $\|\phi^*(\delta)\| < M_1$ and $\dist(\phi^*(\delta), \partial \Delta) > \alpha_1$. However from the definition of $\tau(\varphi_n)$, we have
that for each $n$, $\phi_n^*(\delta)$ satisfies either, $\dist(\phi^*_n(\delta), \partial \Delta) \le \alpha_1$ or $\|\phi^*_n(\delta)\| \ge M_1$. This is a contradiction since $\phi_n^*$ converges to $\phi^*$ (along some subsequence).
The result follows.
\end{proof}
\begin{corollary}
	\label{corr:phixt}
	Let $K\in \clk$ and $T_0>0$. Then there exists a $\gamma>0$ and a  $A_1 \in \clk$ such that whenever for some $x\in K$ we have
	$T^x \in [T_0, \infty)$ and $\phi^x \in C([0,T^x]:\Delta^o)$ with $\phi^x(0)=x$ and $S(x, T^x, \phi^x) \le \gamma$, then $\phi^x(t) \in A_1$ for all $t \in [0, T^x]$.
\end{corollary}
\begin{proof}
Let $\alpha \in (0,\infty)$	be such that $K= \clv_{\alpha, K}$. Let $\alpha_1, K_1$ be as in Lemma \ref{lem:boundedflow}. We argue by contradiction.
Suppose the statement in the corollary is false. Then there are sequences $\gamma_n \downarrow 0$, $x_n \in K$, time instants $T^{x_n} \in [T_0, \infty)$, 
trajectories $\phi^{x_n}\in C([0, T^{x_n}]: \Delta^o)$,
and sets $B_n = \{x\in \Delta^o: \|x\| \le n, \dist(x, \partial \Delta)\ge 1/n\}$ 
such that $S(\phi^{x_n}(0), T^{x_n}, \phi^{x_n}) \le \gamma_n$ and $\phi^{x_n}(t_n) \in B_n^c$ for some $t_n \in [0, T^{x_n}]$.
However, from Lemma \ref{lem:boundedflow}, there exists a $k \in \N$ such that $\phi^{x_n}(t) \in \clv_{\alpha_1, K_1}$ for all $n\ge k$ and all $t \in [0, T^{x_n}]$, which is clearly a contradiction since we can find a $n_0 >k$ such that $\clv_{\alpha_1, K_1} \subset B_n$ for all $n \ge n_0$.
\end{proof}


The following continuity property of $V$, which is a consequence of continuity of $L_{\alpha,K}$ shown in Lemma \ref{lem:cty}, will be needed in the proof of Theorem \ref{thm:basicclasses}.
\begin{lemma}
	\label{lem:ctyofQP}
	Suppose $x_n, x \in \Delta^o$ are such that $x_n\to x$ as $n\to \infty$. Then for every $y \in \Delta^o$, $V(x_n, y)\to V(x,y)$ and $V(y, x_n) \to V(y, x)$. 
\end{lemma}
\begin{proof}
	Fix $x \in \Delta^o$ and let $G \subset \Delta^o$ be a bounded open ball containing $x$ such that $\bar G \subset \Delta^o$. Without loss of generality assume that $x_n \in G$ for every $n$.
	Choose $\alpha \in (0,1)$ and a  $K \in \clk$ such that $\clv_{\alpha, K} \supset \bar G$.
	Since $\bar G$ is compact, from Lemma \ref{lem:cty}, we have that
	$$\sup_{z \in \bar G, \|\beta\|\le 1 } L_{\alpha,K}(z,\beta) \doteq \kappa_0 <\infty,$$
	where $B_1(0)$ is the unit ball in $\R^d$. Let $\eps \in (0,\infty)$ be arbitrary. Take $x_1, x_2 \in G$ such that $x_1\neq x_2$ and $\|x_1-x_2\| \le \eps/(2\kappa_0)$. Also, fix $y \in \Delta^o$.
	From the definition of $V(x_2,y)$ we can find a sequence $T_k\to \infty$ and $\phi_k \in C([0,T_k]: \Delta^o)$ such that for all $k$, $\phi_k(0)=x_2$, $\phi_k(T_k)=y$ and
	$$S(x_2, T_k, \phi_k) \le V(x_2, y) + \eps/2.$$
	Let $\delta = \|x_1-x_2\|$, $\beta \doteq \frac{(x_2-x_1)}{\|x_2-x_1\|}$, $\tilde T_k \doteq T_k+\delta$ and define for $t \le \tilde T_k$
	\[\tilde \phi_k(t) = \begin{cases}
	x_1 +\beta t & t \le \delta\\
	\phi_k(t-\delta)& t \ge \delta
	\end{cases}.
	\]
	Then
	\begin{align*}
		S(x_1, \tilde T_k, \tilde \phi_k) &= 
		\int_0^{\tilde{T}_k} L(\tilde \phi_k(t), \dot{\tilde{\phi}}_k(t)) dt \\
		&= \int_0^{\delta} L(\tilde \phi_k(t), \dot{\tilde{\phi}}_k(t)) dt + S(x_2, T_k, \phi_k)\\
		&= \int_0^{\delta} L_{\alpha,K}(x_1 +\beta t, \beta) dt + S(x_2, T_k, \phi_k)\\
		&\le \kappa_0 \frac{\eps}{2\kappa_0} + V(x_2,y)+\frac{\eps}{2} = V(x_2,y) + \eps.
	\end{align*}
	Thus $V(x_1, y) \le V(x_2, y) + \eps$ which proves the convergence $V(x_n, y)\to V(x,y)$. The proof of $V(y, x_n) \to V(y, x)$ is similar and is omitted.
\end{proof}

The following result is a consequence of compactness of level sets property in Theorem \ref{thm:fwuldp}
and the uniqueness of the path where the rate function vanishes.
\begin{lemma}\label{lem:4.25}
Fix $T \in (0,\infty)$ and a  $K \in \clk$. For each $\delta > 0$, there is some $\eps \doteq \eps(K,T,\delta) > 0$ such that for any $\phi \in C([0,T]:\clv_{\alpha,K})$ and
$x \in K$, if $S_{\alpha,K}(x,T,\phi) \leq \eps$, then $\sup\limits_{0\leq t \leq T}\|\phi(t) - \varphi_t(x)\| < \delta$.
\end{lemma}

\begin{proof}
Arguing via contradiction, suppose that there is some $\delta > 0$ such that for all $\eps > 0$, there is some $x \in K$ and $\phi_{\eps} \in C([0,T]:\clv_{\alpha,K})$ such that $S_{\alpha,K}(x,T,\phi_{\eps})< \eps$ but $\|\phi_{\eps}(t)- \varphi_x(t)\| \ge \delta$ for some $t \in [0,T]$. Using the compactness of level sets property in part (a) of Theorem \ref{thm:fwuldp}
and recalling that $S_{\alpha,K}(x,T,\phi)=0$ if and only if $\phi(t)=\varphi_t(x)$ for $t \in [0,T]$, we see that
$$
c \doteq \inf\{ S_{\alpha,K}(x,T,\phi) : x \in K, \sup_{t\in [0,T]}\|\phi(t) - \varphi_t(x)\| \ge \delta\} > 0.
$$
Thus $c \leq S_{\alpha,K}(x,T,\phi_{\eps}) < \eps$ for all $\eps > 0$. Letting $\eps \downarrow 0$, we obtain $c = 0$, which is a contradiction.
\end{proof}

As an intermediate step we now prove a somewhat weaker statement than that in Theorem \ref{thm:basicclasses}.
\begin{lemma}\label{lem:prop4.29}
Suppose that $x \in \mathcal{R}_{V}$. Then $x \in \mathcal{R}_{\ap}^*$ and $[x]_{V} \subset [x]_{\ap}$.
\end{lemma}
\begin{proof}
Let $y \in [x]_{V}$. Then there exist time instants $T_n \uparrow \infty$ and $\phi_n \in C([0,T_n] : \Delta^o)$
such that for all $n\ge 1$, $\phi_n(0) = x, \phi_n(T_n) = y$, and $S(x,T_n,\phi_n) < \frac{1}{n}$. 
From Lemma \ref{lem:boundedflow} there exists a $k \in \N$, $\alpha_1>0$ and a  $K_1 \in \clk$ such that, for all
$n\ge k$, $\phi_n \in C([0,T_n] : \clv_{\alpha_1, K_1})$.

Now fix $T,\delta>0$. From Lemma \ref{lem:4.25} there is a $\eps>0$ such that, with $T^*= T$ and $T^*= 2T$, 
\begin{equation}\label{eq:cptcty}
\begin{aligned}
\mbox{ whenever for some }	&\phi \in C([0,T^*], \clv_{\alpha_1, K_1}) \mbox{ and } z \in \clv_{\alpha_1, K_1},
S_{\alpha_1, K_1}(z, T^*, \phi) \le \eps,\\
& \mbox{ then } \|\phi - \varphi_{\cdot}(z)\|_{*,T^*} <\delta.
\end{aligned}
\end{equation}
Choose $n_0$ such that $1/n_0 \le \eps$ and $T_{n_0}\ge T$. Write $T_{n_0} = mT +t_0$ where $m\in \N$ and $t_0 \in [0,T)$.
Then,  from \eqref{eq:cptcty}, with $\phi=\phi_{n_0}$
$$\|\phi(jT)- \varphi_T(\phi((j-1)T))\| <\delta \mbox{ for } j=1, \ldots, m-1, \mbox{ and } \|\phi(mT+t_0)- \varphi_{T+t_0}(\phi((m-1)T))\| <\delta.$$
Thus
with $\xi_0=\xi_1= x$, $\xi_2=\phi(T), \ldots , \xi_m= \phi((m-1)T), \xi_{m+1}= \phi(T_{n_0})$, the sequence
$\xi = (\xi_0, \ldots, \xi_{m+1})$ along with time instants $(T, T, \ldots, T+t_0)$ defines a $(\delta,T)$ \ap--pseudo-orbit from $x$ to $y$.
Since $\delta, T>0$ are arbitrary
$x <_{\ap} y$. Similarly, $y <_{\ap} x$, showing that  $x \in \mathcal{R}_{\ap}$ and $y \in [x]_{\ap}$. This shows $[x]_V \subset [x]_{\ap}$
and completes the proof.
\end{proof}

From Lemma \ref{lem:prop4.29} and Assumption \ref{assu:ap-classesfinite} (see also Lemma \ref{lem:apbasicbounded})
 the closure of $\mathcal{R}_V $ is a compact set in $\Delta^o$. 

We  now complete the proof of  Theorem \ref{thm:basicclasses} by establishing the reverse inclusion from the one
established in Lemma \ref{lem:prop4.29}.

{\em Proof of Theorem \ref{thm:basicclasses}.}
From Lemma \ref{lem:prop4.29}
if $x \in \mathcal{R}_{V}$, then $x \in \mathcal{R}_{\ap}^*$ and $[x]_{V} \subset [x]_{\ap}$.
Now suppose that $x \in \clr_{\ap}^*$. From Assumption \ref{assu:ap-classesfinite} there is a $x^* \in [x]_{\ap}$
such that $\{\varphi_t(x^*); t \ge T\}$ is dense in $[x]_{\ap}$ for every $T>0$.
Fix $y \in [x]$. Let for $n\in \NN$, $t_n, \tilde t_n \in (0,\infty)$ be such that $t_n \uparrow \infty$,
$\tilde t_n \uparrow \infty$ as $n\to \infty$, and for every $n$
$$\|\varphi_{t_n}(x^*) - x\| \le 1/n, \;\; \|\varphi_{t_n+\tilde t_n}(x^*) - y\| \le 1/n.$$
Using Lemma \ref{lem:ctyofQP} it follows that $x<_V y$. This shows that $x\in \clr_V$ and that
$[x]_{\ap}\subset [x]_V$. We thus have that $\clr^*_{\ap}= \clr_V$ and for all $x \in \clr_V=\clr_{\ap}$,
$[x]_{\ap}=[x]_V$. Similar arguments show that $[x]$ is a $V$-quasiattractor if and only if it is an \ap-quasiattractor.
The result follows.

\hfill \qed

In view of Theorem \ref{thm:basicclasses}, henceforth we will use the qualifier `$V$' or `\ap' interchangeably when referring to recurrence classes and quasiattractors in $\Delta^o$.

\section{Proof of Theorem \ref{thm:main}}\label{sec:pfofmain}
In this section we assume that Assumptions \ref{assu:LLN}, \ref{assu:ap-classesfinite}, \ref{assu:mgf} and \ref{assu:irrbdr} are satisfied.
The following lemma shows that there are low cost trajectories that take any given point in a recurrence class to any other point in the same class.
\begin{lemma}\label{lem:discgrid}
	For any $\gamma>0$ and $K\in \clr_V$, there is a $T\in (1,\infty)$ such that for all $x,y\in K$, there exist $T_{x,y}\in (1, T)$
	and $\phi_{x,y} \in C([0, T_{x,y}]: \Delta^o)$ with
	$$
	S(x, T_{x,y}, \phi_{x,y}) \le \gamma, \; \phi_{x,y}(0)=x, \; \phi_{x,y}(T_{x,y})=y.$$
\end{lemma}
\begin{proof}
	Fix $\gamma \in (0,1)$ and $K\in \clr_V$. Let $\gamma_0 \doteq \sup_{z\in K,\|\beta\|\le 1} L(z,\beta)$. 
	Let $k \in \NN$ and $v_1, \ldots, v_k \in K$ be such that for any $x\in K$, there exists $1\le i \le k$ with $\|x-v_i\| \le \gamma/(4\kappa_0)$.
	For $i,j \in \{1, \ldots, k\}$, let $\tilde T_{i,j} \in (1, \infty)$ and $\psi_{i,j} \in C([0,\tilde T_{i,j}]: \Delta^o)$ be such that
	$\psi_{i,j}(0)= v_i$, $\psi_{i,j}(\tilde T_{i,j})= v_j$ and $S(v_i, \tilde T_{i,j}, \psi_{i,j}) \le \gamma/2$.
	Let $x,y\in K$ be arbitrary and select $i,j \in \{1, \ldots, k\}$ such that $\|x-v_i\|\le \gamma/(4\kappa_0)$ and $\|y-v_j\|\le \gamma/(4\kappa_0)$.
	Consider the continuous trajectory $\phi_{x,y}$ in $\Delta^o$ defined over the time interval of length $T_{x,y} = \|x-v_i\| + \tilde T_{i,j}+ \|y-v_j\|$
	as follows.
	 \begin{equation}\label{eq:eq257ne}
	 x \mathrel{\substack{\mbox{\tiny{lin}}\\\longrightarrow\\\|v_i-x\|}} v_i
	 \mathrel{\substack{\psi_{i,j}\\\longrightarrow\\\tilde T_{i,j}}} v_j
	  \mathrel{\substack{\mbox{\tiny{lin}}\\\longrightarrow\\\|v_j-y\|}} y
	 \end{equation}
	In the above display for a term of the form $a \mathrel{\substack{c\\\longrightarrow\\d}} b$, the trajectory connects the points $a$ and $b$
	in time length $d$ in a manner described by $c$.  When $c=\mbox{{lin}}$, the trajectory is just a linear path connecting $a$ and $b$ and 
	 when $c= \psi_{i,j}$, the trajectory is defined by $\psi_{i,j}$ introduced above.
	 Clearly $S(x, T_{x,y}, \phi_{x,y}) \le \gamma$, $\phi_{x,y}(0)=x$ and $\phi_{x,y}(T_{x,y})=y$.
	 Also, $T_{x,y}\le \max_{1\le i,j\le k} \tilde T_{i,j} + 2 \doteq T$.
	 The result follows.
	
\end{proof}

Recall that for  a set $B \subset \Delta$,  $\tau_B^N \doteq \inf\{ t \ge 0: \hat X^N(t) \not \in B\}$.
The following lemma gives an upper bound on the probabilities of long residence times of the Markov chain near non-quasiattractors.

\begin{lemma}\label{lem:timeaway}
Suppose that $K_j \in \mathcal{R}_{V}$ is not a quasiattractor. Then we can find some $\lambda > 0$ such that for all $\gamma  > 0$, there is some $N_0 \doteq N_0(\gamma)$ and $\zeta \doteq \zeta_{\gamma} : \mathbb{N} \rightarrow \mathbb{R}$ satisfying $\lim\limits_{n\rightarrow\infty}\zeta_{\gamma}(n) = 0$ such that
$$
\sup\limits_{x\in N^{\lambda}(K_j)} \PP_x\left( \tau_{N^{\lambda}(K_j)}^N > \exp( N \gamma)\right) \leq \zeta_{\gamma}(N)
$$
for all $N \geq N_0$.
\end{lemma}

\begin{proof}
	Since $K_j$ is not a quasiattractor, there exists a $\la_0\in (0,1)$, $u_1 \in K_j$, $y_1 \in \Delta^o \cap [N^{2\la_0}(K_j)]^c$ such that $u_1 <_{V} y_1$. 
Choose $\la_1 \in (0, \la_0)$ 
	such that, for some  $A_0 \in \clk$, $y_1\in A_0$,
	$\cup_{k=1}^v \overline{ N^{\la_1}(K_k)} \subset A_0$, and for each $i,k \in \{1,\dots,v\}$ such that $i \neq k$,
	$\dist(N^{\la_1}(K_k), N^{\la_1}(K_i)) \ge \la_1$.
	From Lemma \ref{lem:bdryrepel} we can find a $A_1\in \clk$ such that the forward orbit $\gamma^+(x) \subset A_1$ for every $x \in A_0$.
	Let $\sup_{z \in  A_1, \|\beta\|\le 1 } L_{\alpha,K}(z,\beta) \doteq \kappa_0$.
	Let $\gamma>0$ be given and let $\gamma_0=\gamma/6$. Fix $\delta \in (0, \la_1\wedge \frac{\gamma_0}{\kappa_0})$. Then, denoting by 
	$\eta_{x, y}$ the linear trajectory from $x$ to $y$,
	$$\mbox{ for } x^*, y^* \in A_1 \mbox{ with } \|x^*-y^*\| \le \delta, S(x^*, \|y^*-x^*\| , \eta_{x^*, y^*})\le \gamma_0.$$
	With $\delta$ as above, choose, $T_{A_1}^*$ as in Lemma \ref{lem:intersectrap}(b) (with $A$ replaced with $A_1$).
	Then, in view of Theorem \ref{thm:basicclasses}, for every $x \in A_1$, there exists a $t_0 \in [0, T_{A_1}^*]$ such that $\varphi_{t_0}(x) \in N^{\delta}(\clr_V)$.
	
	Define for $x \in N^{\la_1}(K_j)$ the continuous trajectory $\phi^{\gamma}_x(\cdot)$ according to the following two cases: Case I: $\varphi_{t_0}(x) \in N^{\delta}(K_i)$ for some $i \neq j$, Case II: $\varphi_{t_0}(x) \in N^{\delta}(K_j)$.
	
	In Case I, we simply take $\phi^{\gamma}_x(t) = \varphi_t(x)$ for  $t \in [0, t_0]$. In particular, $T^{\gamma}_x \doteq  t_0$ is the length of the time interval over which the trajectory is defined.
	
	For Case II we proceed as follows.  Taking $K=A_0$ and $T_0=1$ in Corollary \ref{corr:phixt}, denote by $(\gamma^*, A^*)$ the $(\gamma, A_1)$ given by the corollary.
	Let $u_0 \in K_j$ be such that $\|u_0 - \varphi_{t_0}(x)\| \le \delta$.  Then $u_0 <_{V} u_1 <_{V} y_1$.
	Let $t_1(x) \in [1,\infty)$ and $\phi_1 \in C([0,t_1(x)]:\Delta^o)$ be such that $\phi_1(0)=u_0$, $\phi_1(t_1(x))= y_1$ and $S(u_0, t_1(x), \phi_1) \le \gamma^*\wedge \gamma/3$.
	Using Lemma \ref{lem:discgrid} we can assume without loss of generality that $\sup_{w \in N^{\lambda}(K_j)} t_1(w)\doteq \bar t_1 <\infty$.
	From Corollary \ref{corr:phixt}, $\phi_1(t)\in A^*$ for all $t \in [0, t_1(x)]$.
	Consider the  continuous trajectory $\phi^{\gamma}_x$ in $\Delta^o$ that connects $x$ and $y_1$ in the manner described by the display:
 $$
 x \mathrel{\substack{\mbox{\tiny{flow}}\\\longrightarrow\\t_0}} \varphi_{t_0}(x)
 \mathrel{\substack{\mbox{\tiny{lin}}\\\longrightarrow\\\|u_0-\varphi_{t_0}(x)\|}} u_0
 \mathrel{\substack{\phi_1\\\longrightarrow\\t_1(x)}} y_1
 $$
The above display is interpreted in a similar manner as \eqref{eq:eq257ne} with  a term of the form $a \mathrel{\substack{c\\\longrightarrow\\d}} b$, when
$c=\mbox{{flow}}$, representing the  segment of  $\varphi_t(a)$ until it reaches $b$. In this case let $T^{\gamma}_x = t_0+ \|u_0-\varphi_{t_0}(x)\| + t_1(x)$ denote the length of the time interval over which 	$\phi^{\gamma}_x$ is defined.

Note that in both cases, $T^{\gamma}\doteq \sup_{x \in N^{\la_1}(K_j)} T^{\gamma}_x \le t_0+ 1 + \bar t_1<\infty$.
Also, in both cases, $\phi_x^{\gamma}(t) \in A_1 \cup A^*\doteq A_2$ for all $t \in [0, T^{\gamma}_x]$.
Furthermore, in Case II,
$$S(x, \phi^{\gamma}_x, T^{\gamma}_x) \le 0 + \gamma_0+ \gamma/3 = \gamma/2$$
and in Case I the cost on the left side of the above display is $0$.

Let $\la \in (0,\la_1)$, $\alpha'>0$ be such that $K' \doteq \overline{N^{\la}(A_2)} \subset \Delta^o$ and $K' = \clv_{\alpha', K'}$.
Extend the trajectory $\phi^{\gamma}_x$ from $[0, T^{\gamma}_x]$ to $[0, T^{\gamma}]$ by defining 
$\phi^{\gamma}_x(t+T^{\gamma}_x) \doteq \varphi_t(\phi^{\gamma}_x(T^{\gamma}_x))$ for $t \in (T^{\gamma}_x, T^{\gamma}]$.
The   bound from Theorem \ref{thm:fwuldp}(b) ensures that for each $\tilde{\delta} \in (0,1)$ there is some $N_0(\tilde{\delta}) \in \N$ such that,
whenever $N \geq N_0(\tilde{\delta})$,
\begin{align*}
\PP_x( \|\phi^{\gamma}_x - \hat{X}^N\|_{*, T^{\gamma}_x} < \lambda) &= \PP_x(\|\phi_x^{\gamma}- \hat{X}^{N,\alpha,K'}\|_{*,T^{\gamma}_x} < \lambda)
\ge \PP_x(\|\phi_x^{\gamma}- \hat{X}^{N,\alpha,K'}\|_{*,T^{\gamma}} < \lambda)\\
&\geq \exp\left( -N(S(x,T^{\gamma},\phi_x^{\gamma}) + \tilde{\delta}/4)\right)
= \exp\left( -N(S(x,T^{\gamma}_x,\phi_x^{\gamma}) + \tilde{\delta}/4)\right)\\
&\geq \exp( -N(\gamma/2 +\tilde{ \delta}/4))
\end{align*}
for all $x \in N^{\lambda}(K_j)$. 

It follows that for each $x \in N^{\lambda}(K_j)$, if $N \geq N_0(\gamma)$, then 
\begin{equation*}
\begin{split}
\PP_x\left( \tau_{N^{\lambda}(K_j)}^N > T^{\gamma}\right) & \leq 1 - \PP_x( \|\phi^{\gamma}_x - \hat{X}^N\|_{*, T^{\gamma}_x} < \lambda)\\
&\leq 1 - \exp(-N(\gamma/2 + \gamma/4))\\
\end{split}
\end{equation*}
Using the Markov property we see that, if $N \geq N_0(\gamma)$ and $x \in N^{\lambda}(K_j)$, then
\begin{equation*}
\begin{split}
\PP_x\left( \tau^N_{N^{\lambda}(K_j)} > \exp(N \gamma)\right) &\leq \PP_x\left( \tau^N_{N^{\lambda}(K_j)} > \left\lfloor \frac{\exp(N \gamma)}{T^{\gamma}} \right\rfloor T^{\gamma}\right)\\
&\leq (1 - \exp(-3N\gamma/4))^{\left\lfloor \frac{\exp(N \gamma)}{T^{\gamma}} \right\rfloor}.\\
\end{split}
\end{equation*}
We can assume without loss of generality that $N(\gamma)$ is large enough so that $\left\lfloor \frac{\exp(N \gamma)}{T^{\gamma}} \right\rfloor > \frac{\exp(N \gamma)}{2T^{\gamma}}$.
Then for  all $N \geq  N_0(\gamma)$, 
\begin{equation*}
\begin{split}
\sup\limits_{x\in N^{\lambda}(K_j)}\PP_x\left( \tau^N_{N^{\lambda}(K_j)} > \exp(N \gamma)\right)  &\leq  (1 - \exp(-3N\gamma/4))^{ \frac{\exp(N \gamma)}{2T^{\gamma}} }\\
&= \exp\left( \log\left( 1 - \exp(-3N\gamma/4))\right)^{ \frac{\exp(N\gamma)}{2T^{\gamma}}}\right)\\
&\leq \exp\left( -\frac{\exp(N\gamma)}{2T^{\gamma}}\exp(-3N\gamma/4)\right)\\
&= \exp\left( - \frac{\exp(N\gamma/4)}{2T^{\gamma}}\right).
\end{split}
\end{equation*}
The result follows from taking $\zeta_{\gamma}(N) \doteq \exp\left( - \frac{\exp(N\gamma/4)}{2T^{\gamma}}\right)$.
\end{proof}

{\em Proof of Theorem \ref{thm:main}.}
Recall that we assume that Assumptions \ref{assu:LLN}, \ref{assu:ap-classesfinite}, \ref{assu:mgf} and \ref{assu:irrbdr} are satisfied.
Also, by assumption, for every $N\in \N$, there exists a quasi-stationary distribution $\mu_N$ for $\{X^N\}$ and that  the sequence $\{\mu_N\}$ is relatively compact. 
From Theorem \ref{thm:eigentheorem} there are $a_1, c_1 \in (0,\infty)$ such that
$$\la_N \ge 1 - a_1 e^{-c_1N}, \mbox{ for all } N \in \N.$$
Let $\mu$ be a limit point of $\mu_N$. From Theorem \ref{thm:invariance} $\mu$ is invariant under the flow $\{\varphi_t\}$.
From Corollary \ref{cor:limitsupportrap} $\operatorname{supp}(\mu) \subset \mathcal{R}_{\ap}^*$.  Thus to finish the proof, it suffices to
show that for every $j \in \{l+1, \ldots, v\}$, there is a neighborhood $V_j$ of $K_j$ such that $\mu(V_j)=0$.
Fix $\eps>0$ and  choose a  $F_0 \in \clk$ such that $\mu_N(F_0^c)< \eps$ for every $N\in \NN$. This can be done in view of Theorem \ref{thm:abs} and our assumption that the sequence $\{\mu_N\}$ is relatively compact.

Using Lemma \ref{lem:bdryrepel}(c) we can assume that $F_0$ is large enough so that for some $T_1 \in (0,\infty)$ and $\hat \delta >0$,
$\varphi_t(x) \in F_1$ for all $t \ge T$ and $x \in F_0$ where $F_1 \subset F_0$ is such that $\dist(F_1, \partial F_0)> \hat \delta$.


Let $\la$ be as in Lemma \ref{lem:timeaway}. Fix $\delta =\la\wedge \hat \delta$.
From Lemma \ref{lem:choosevi}, we can choose $\delta_0 \in (0, \delta)$, an integer $T_0>T_1$, and open sets $V_i$ with $\bar V_i \subset N^{\delta}(K_i)\cap \Delta^o$ such that (1)-(3)
of Lemma \ref{lem:choosevi} hold.

Consider
$$\beta_{\delta_0, T_0, F_0}^N \doteq \sup_{x \in \Delta_N\cap F_0} \PP_x[\| \hat X^N - \varphi_{\cdot}(x) \|_{*,T_0} \ge \delta_0].$$
Then, from Lemma \ref{lem:expineq} there exist $c_2>0$ and $a_2 \in (0,\infty)$ such that
$$\beta_{\delta_0, T_0, F_0}^N  \le a_2 e^{-N c_2}.$$
Define $c^* = \min\{1, c_1, c_2\}$.
Let, with $\gamma = c^*/8$,  $\zeta(N, \gamma)\doteq \zeta^*(N)$ be as in Lemma \ref{lem:timeaway}. Then, for some $a_3 \in (0,\infty)$,
$$
\sup\limits_{x\in N^{\la}(K_j)} \PP_x\left( \tau_{N^{\la}(K_j)}^N > \exp( N c^*/8)\right) \leq a_3\zeta^*(N), \mbox{ for all } N \in \N.
$$
Define $m_N = \exp(Nc^*/2)$ and $m'_N = \exp(Nc^*/4)$.

Define the events
\begin{align*}
	\cle_N &= \{(\hat X^N(0), \hat X^N(T_0), \hat X^N(2T_0), \ldots, \hat X^N(m_NT_0)), (T_0, T_0, \ldots, T_0),\\
	&\quad \quad \mbox{ defines a } (\delta_0, T_0) \mbox{ \ap--pseudo-orbit.}\}
	\end{align*}
and
\begin{align*}\cle'_N &= \{\mbox{for any } i \in \{l+1, \ldots , v\} \mbox{ and any } q \ge m'_N, \mbox{ and } p \ge 0,\\
&\quad \quad \mbox{ if } \hat X^N(pT_0) \in N^{\delta_0}(K_i), \mbox{ then }
\hat X^N((p+q)T_0) \not \in N^{\delta_0}(K_i)\}.\end{align*}
Without loss of generality we can assume that $m_N> (b+2)(m'_N+1)$.
Then, for $x \in \Delta^o$
$$\PP_x(\hat X^N(m_NT_0) \in V_i) \le  \PP_x(\hat X^N(m_NT_0) \in V_i, \cle_N, \cle'_N) + \PP_x(\cle_N, (\cle'_N)^c ) + \PP_x((\cle_N)^c).$$
Define for $\alpha = 1, \ldots, b+1$, $t^N_{\alpha} = \lfloor\alpha m_N/(b+2)\rfloor$.
Then, from Corollary \ref{lem:choosevi} (2) and definition of $\cle'_N$,
with $K = \cup_{j=1}^v K_j$,
$$\PP_x(\hat X^N(m_NT_0) \in V_i, \cle_N, \cle'_N) \le \sum_{\alpha=1}^{b+1} \PP_x(\hat X^N( t^N_{\alpha}T_0)  \in [N^{\delta_0}(K))]^c \cap \Delta^o).$$
Using Lemma \ref{lem:choosevi} (3), for every $x \in \Delta^o$
\begin{align*}
	\PP_x(\cle_N, (\cle'_N)^c ) &\le \sum_{i=l+1}^{v}  \sup_{x \in N^{\delta_0}(K_i)} \PP_x(\tau^N_{V_i} > T_0m'_N)\\
&\le \sum_{i=l+1}^{v}  \sup_{x \in N^{\delta}(K_i)} \PP_x(\tau^N_{N^{\delta}(K_i)} > \exp(Nc^*/4)) \le 
b \zeta^*(N).
\end{align*}
From our choice of $\delta_0, T_0$ we see that if for some $k$,
$\hat X^N((k-1)T_0) \in F_0$, then
$\varphi_{T_0}(\hat X^N((k-1)T_0)) \in F_1$, and if in addition,
$\| \hat X^N(kT_0) - \varphi_{T_0}(\hat X^N((k-1)T_0))\| \le \delta_0$, then
$\hat X^N(kT_0) \in F_0$. Using this observation, we see that, with
$$k^* \doteq \min\{1\le k \le m_N: \|\hat X^N(kT_0) - \varphi_{T_0}(\hat X^N((k-1)T_0))\| > \delta_0\},$$
for every $x \in F_0$,
\begin{align*}
	\PP_x((\cle_N)^c) &= \PP_x(\| \hat X^N(kT_0) - \varphi_{T_0}(\hat X^N((k-1)T_0))\| \ge \delta_0 \mbox{ for some } k=1, \ldots, m_N)\\
	&= \PP_x(k^*\le m_N)\\
	&\le \sum_{k=1}^{m_N} \PP_x(\| \hat X^N(kT_0) - \varphi_{T_0}(\hat X^N((k-1)T_0))\| \ge \delta_0, \hat X^N((k-1)T_0 \in F_0)\\
	&\le m_N \sup_{x\in F_0} \PP_x(\| \hat X^N(T_0) - \varphi_{T_0}(\hat X^N(0))\| \ge \delta_0) \\
	&\le m_N\beta_{\delta_0, T_0, F_0}^N
	\le a_2\exp(Nc^*/2)\exp(-Nc^*)= a_2\exp(-Nc^*/2).
\end{align*}	

Thus, from our choice of $F_0$
\begin{align*}
\la_N^{m_NT_0}\mu_N(V_j) &=  \int \mu_N(dx) \PP_x(\hat X^N(m_NT_0) \in V_j)\\
&\le \int \mu_N(dx) \PP_x(\cle_N, (\cle'_N)^c ) + \sum_{\alpha=1}^{b+1} \int \mu_N(dx)  \PP_x(\hat X^N( t^N_{\alpha}T_0)  \in [N^{\delta_0}(K)]^c \cap \Delta^o)\\
&\quad + \int_{F_0} \mu_N(dx) \PP_x((\cle_N)^c) + \eps\\
&\le b \zeta^*(N) + (b+1)\mu_N([N^{\delta_0}(K)]^c) + a_2\exp(-Nc^*/2) + \eps
\end{align*}
Note that $\mu_N([N^{\delta_0}(K))]^c) \to 0$, in view of Theorem \ref{thm:abs} and Corollary \ref{cor:limitsupportrap}.
Since $\la_N \ge 1 - a_1 e^{-c_1N}$ and $m_Ne^{-c_1N} \le e^{-c_*N/2}\to 0$,
$\la_N^{m_NT_0} \to 1$. 
Thus sending $N\to \infty$ in the above display, we have
$\mu(V_j)\le \eps$. Since $\eps>0$ is arbitrary, the result follows.\hfill \qed

\section{Proof of Theorem \ref{thm:poisbin}}\label{sec:qsd}
In this section we prove Theorem \ref{thm:poisbin}. For this  we first show that when $\theta^N=\theta^{N,*}$, under the conditions of the theorem,  Assumptions \ref{assu:LLN}, \ref{assu:ap-classesfinite}, \ref{assu:mgf} and \ref{assu:irrbdr} are satisfied.
These assumptions are verified in Sections \ref{sec:verassu1}, \ref{sec:verassu2}, \ref{sec:verassu3}, \ref{sec:verassu4}, respectively.
We then argue in Section \ref{sec:verassu5} that, for every $N$, $X^N$ has a QSD $\mu_N$ of the form in the statement of Theorem \ref{thm:poisbin}. In Section \ref{sec:verassu6} we show that the sequence $\{\mu_N\}$ is tight.  Finally, in Section \ref{sec:verassu7}
we combine the results of previous sections to complete the proof of Theorem \ref{thm:poisbin}.

\subsection{Verification of Assumption \ref{assu:LLN}}\label{sec:verassu1}
We need to show that when $\theta^N = \theta^{N,*}$, and $x^N \to x$, then
\eqref{eq:llnps} holds. The proof follows by a standard application of Gr{\"o}nwall's lemma and from moment formulas of Poisson and Binomial random variables and thus we only give a sketch.
First, using the relation \eqref{eq:bmc} and the discrete time Gr{\"o}nwall inequality it is easy to verify that for every $T<\infty$
\begin{equation}\label{eq:unifmombd}
	\sup_{N\in \NN} \EE_{x^N} \max_{0 \le k \le \lfloor N T\rfloor} \|X^N_k\|^2 <\infty.
	\end{equation}
Next, using the relation,
$$
\eta^N_{k+1}(x) = G(x) + [\eta^N_{k+1}(x) - \EE(\eta^N_{k+1}(x))], \; x \in \Delta,$$
and the Lipschitz property of $G$, it can be checked that
\begin{equation}
	\label{eq:discsde}
	 \hat X^N(t) = x^N + \int_0^t G(\hat X^N(s)) ds + M^N(t) + R^N(t), \; t \in [0,T],\; N\in \NN
\end{equation}
where $M^N$ is a martingale and $\sup_{0\le t \le T} \|R^N(t)\|$ converges to $0$ in probability as $N\to \infty$.
Standard moment estimates show that
$\EE (\sup_{0\le t \le T} \|M^N(t)\|^2) \to 0$ as $N\to \infty$. Next, using the moment bound \eqref{eq:unifmombd} and the convergence properties noted above, it can be checked that $\hat X^N$ is tight in $C([0,T]:\Delta)$. Finally, if $\hat X^N$ converges in distribution along a subsequence to $\hat X$, then from \eqref{eq:discsde} it follows that $\hat X$ must satisfy
$$
\hat X(t) = x + \int_0^t G(\hat X(s)) ds, \; t \in [0,T].$$
From the unique solvability of the ODE in \eqref{eq:dynsys}, which is a consequence of the Lipschitz property of $G$, it now follows
that
$\hat X(t) = \varphi_t(x)$ for all $t\in [0,T]$, a.s. This proves the convergence in \eqref{eq:llnps}. \hfill \qed

\subsection{Verification of Assumption \ref{assu:ap-classesfinite}}\label{sec:verassu2}
Parts (a)-(d) hold by assumption. We now verify part (e) of Assumption \ref{assu:mgf}.
Since $G(x)= F(x)-x$, for each $x \in \Delta$ $\lan x, G(x)\ran = \lan x, F(x)\ran - \|x\|^2$.
As $F$ is bounded, taking $M \doteq 2 \|F\|_{\infty}$, we see that
$\|F(x)\| \le \|x\|/2$ for all $\|x\| \ge M$. Thus
$$
\lan x, G(x)\ran \le \frac{1}{2} \|x\|^2 - \|x\|^2 = -\frac{1}{2} \|x\|^2  \mbox{ for all } x \in \Delta \mbox{ with } \|x\| \ge M.$$
Thus Assumption \ref{assu:mgf}(e) holds with $\kappa =1/2$ and $M$ as above. \hfill \qed

\subsection{Verification of Assumption \ref{assu:mgf}}\label{sec:verassu3}
Part (a) of the assumption is immediate from the fact that for $x\in \Delta^o$, $\theta^{N,*}(\cdot|x)$ is the probability law of $U^N-V^N$, where 
 $U^N=(U_i^N)_{i=1}^d$ and $V^N=(V_j^N)_{j=1}^d$ are $d$-dimensional random variables such that $\{U_i^N, V_j^N, \; i,j=1, \ldots, d\}$ are mutually independent and $U_i^N \sim \mbox{Poi}(F_i(x))$, $V_j^N \sim \mbox{Bin}(Nx_j, 1/N)$, for $i,j=1, \ldots, d$.
 
For part (b), define, for $x\in \Delta^o$, $\theta(\cdot|x)$ as the probability law of $U-V$, where 
 $U=(U_i)_{i=1}^d$ and $V=(V_j)_{j=1}^d$ are $d$ dimensional random variables such that $\{U_i, V_j, \; i,j=1, \ldots, d\}$ are mutually independent and $U_i \sim \mbox{Poi}(F_i(x))$, $V_j \sim \mbox{Poi}(x_j)$, for $i,j=1, \ldots, d$.
Then with this choice of $\theta$, Assumption \ref{assu:mgf}(b) parts (i) and (ii) are clearly satisfied.
 Finally, part (iii) is a consequence of the observation that if $z_N\to z \in (0,\infty)$, then for every $\la \in \RR$, as $N\to \infty$
 $$
 \left[\left(1 - \frac{1}{N}\right) + \frac{1}{N} e^{\lambda}\right]^{Nz_N} \to e^{z(e^{\la}-1)}.$$
 \hfill \qed
\subsection{Verification of Assumption \ref{assu:irrbdr}}\label{sec:verassu4}
Part (a) of the assumption is clearly satisfied (in fact with $k=1$). Part (b) is verified in the following lemma.
\begin{lemma}\label{thm:absorption}
	Suppose that $\theta^N = \theta^{N,*}$. Then,
	 for every $\gamma \in (0,\infty)$ and $T\in \N$, there is an open neighborhood $U_{\gamma}$ of $\partial \Delta$ in $\Delta$ such that
	$$\liminf_{N\to \infty} \inf_{x \in U_{\gamma}\cap \Delta_N} \frac{1}{N} \log \PP_x(\hat X^{N}(T) \in \partial \Delta) \ge -\gamma.$$
\end{lemma}

\begin{proof}
	
	For  $x \in \Delta^o$, let $i_x \doteq \arg\min\limits_{1\leq i \leq d}x_i$.
	From Assumption \ref{assu:ap-classesfinite}(d)
	 we can find
 $\delta_0 > 0$   such that $\sup\limits_{y\in N^{\delta_0}(\partial \Delta)} F(y)_{i_y} < \frac{\gamma}{2T}$. Let  $\delta_1 \doteq \frac{\gamma}{2(T - \log(e^T-1))}$, $\delta \doteq \min\{\delta_0 , \delta_1\}$, and  $U_{\gamma} \doteq N^{\delta}(\partial \Delta)$.  Fix $x \in U_{\gamma}\cap \Delta_N^o$, and
note that, under $\PP_x$,
$\hat X^N(T) = x + \frac{1}{N} \sum_{j=1}^{NT} \eta_j^N(X^N_{j-1})$, where $\eta_j^N(X^N_{j-1}) = U^j-V^j$ 
and the conditional distribution of $(U^j-V^j)$ given that $X^N_{j-1}=x$ is that of $(U^N, V^N)$ as in Section \ref{sec:verassu3}.
Thus
\begin{align*}
\PP_x( \hat{X}^N(T) \in \partial \Delta) &\geq \PP_x\left( x_{i_x}  + \frac{1}{N} \sum\limits_{j=1}^{NT}(U^j_{i_x} - V^j_{i_x}) = 0\right)\\
&\geq \PP_x\left( U^1_{i_x} = \cdots = U^{NT}_{i_x} = 0, \sum\limits_{j=1}^{NT} V^j_{i_x}= Nx_{i_x}\right).
\end{align*}
Let $\tilde U^1, \ldots, \tilde U^{NT}$ be i.i.d.~Poisson random variables with mean $\gamma/2T$,
and let $\tilde V^1,  \tilde V^{2}, \ldots, \tilde V^{NT}$ be iid Geometric random variables with probability of success $1/N$ such that
$\{\tilde U^j, \tilde V^k; j,k\}$ are mutually independent. 
Then
\begin{align*}
&\PP_x\left( U^1_{i_x} = \cdots = U^N_{i_x} = 0, \sum\limits_{j=1}^N V^j_{i_x}= Nx_{i_x}\right)\\
&\quad= \PP_x\left( U^1_{i_x} = \cdots = U^N_{i_x} = 0, \text{by time instant }NT \text{ all initial $Nx_{i_x}$ type $i_x$ particles die}
\right)\\
&\quad\ge \PP_x\left( \tilde U^1 = \cdots = \tilde U^{NT} = 0, \tilde V^1\le NT,  \tilde V^{2}\le NT \ldots, \tilde V^{Nx_{i_x}}\le NT\right)\\
&\quad= [\PP(\tilde U^1=0)]^{NT} \left( 1 - \left( 1 - \frac{1}{N}\right)^{NT}\right)^{Nx_{i_x}}\\
&\quad = \exp\left(-NT\frac{\gamma}{2T}\right)\left( 1 - \left( 1 - \frac{1}{N}\right)^{NT}\right)^{Nx_{i_x}}
%
%
\end{align*}
Combining the last two displays
\begin{equation*}
\begin{split}
\frac{1}{N}\log \PP_x( \hat{X}^N(T) \in \partial \Delta) &\geq  \frac{1}{N} \left(  \log \left(\exp\left(-N \frac{\gamma}{2}\right)\right) +   \log \left( \left( 1 - \left( 1 - \frac{1}{N}\right)^{NT}\right)^{Nx_{i_x}}\right)\right)\\
&= -\frac{\gamma}{2} + x_{i_x}\log\left( 1 - \left( 1 - \frac{1}{N}\right)^{NT}\right)\\
&\geq - \frac{\gamma}{2} + \delta  \log\left( 1 - \left( 1 - \frac{1}{N}\right)^{NT}\right),
\end{split}
\end{equation*}
and thus from our choice of $\delta$,
$$
\liminf_{N\to \infty} \inf_{x \in U_{\gamma}\cap \Delta_N} \frac{1}{N} \log \PP_x(\hat X^{N}(T) \in \partial \Delta)  \geq -\frac{\gamma}{2} + \delta(-T + \log(e^T-1)) \ge -  \gamma.
$$
\end{proof}

\subsection{Existence of Quasi-stationary Distributions}\label{sec:verassu5}

In this section we prove the existence of a QSD $\mu_N$ for the Markov chain $\{X_n^N\}$, for each $N \in \N$, and show that the sequence $\{\mu_N\}$ of QSD is relatively compact in $\clp(\Delta)$.
For some uniform bounds needed for the tightness proof in Section \ref{sec:verassu6},	it will be convenient to consider 
%
the $N$-step processes $\{\tilde{X}^N_n\}_{n\in \N_0}$, where
\begin{equation}
\tilde{X}^N_n \doteq X^N_{nN}, n \in \NN_0, N\in \NN. \label{eq:nstepproc}
\end{equation}
Recall the definition of  $\tau^N_{\partial}$ and $P_n^N$ from \eqref{eq:tnpa} and \eqref{eq:pnnf}.

For existence of QSD, we will use the following result from \cite{chavil}.
\begin{theorem}\label{thm:foster}(\cite[Theorem 2.1, Proposition 3.1]{chavil})
Fix $N \in \N$. Suppose that there are    $\theta_1,\theta_2, c_1 \in (0,\infty)$,  functions $\varphi_1, \varphi_2 : \Delta^o_N \rightarrow \R_+$, 
	and a measurable subset $K \subset \Delta^o_N$ such that: 
	\begin{enumerate}[(B1)]
		\item For each $x\in K$, for some $n_2(x)\in \NN$,
		 	$$
			\PP_x({X}^N_n \in K) > 0, \; \mbox{ for all } x \in K \mbox{ and } n \ge n_2(x).
		 	$$
	\item  We have $\theta_1 < \theta_2$ and
	 \begin{enumerate}[(a)]
	 \item $\inf\limits_{x\in \Delta^o_N}\varphi_1(x) \geq 1$, $\sup\limits_{x\in K} \varphi_1(x) < \infty$
	 \item $\inf\limits_{x\in K} \varphi_2(x) > 0$, $\sup\limits_{x\in \Delta^o_N}\varphi_2(x) \leq 1$
	 \item $P_1^N \varphi_1(x) \leq \theta_1\varphi_1(x) + c_1 1_{K}(x)$ for all $x \in \Delta^o_N$
	 \item $P_1^N \varphi_2(x) \geq \theta_2 \varphi_2(x)$ for all $x \in \Delta^o_N$.
	 \end{enumerate}
\end{enumerate}
Suppose also that there exist  $C \in (0,\infty)$ and $n_0, m_0 \in \N$ such that $n_0 \leq m_0 $ and
\begin{equation}\label{eqn:prop3.1}
\begin{split}
\PP_x({X}^N_{n_0} \in \cdot \cap K) \leq C \PP_y({X}^N_{m_0} \in \cdot), \text{ for all }x \in \Delta^o_N \text { and }y \in K.
\end{split}
\end{equation}
Then
there exist $C_1 \in (0,\infty)$, $\alpha \in (0,1)$, and a probability measure $\mu_N$ on $\Delta^o_N$ such that, for all $n\in \N$
$$
\left|\left| \frac{\mu P_n^N}{\mu P_n^N (1_{\Delta_N^o})} - \mu_N\right|\right|_{TV} \leq C \alpha^n \frac{\mu (\varphi_1)}{\mu(\varphi_2)},
$$
for all probability measures $\mu$ on $\Delta^o_N$ which satisfy $\mu(\varphi_1) < \infty$ and $\mu(\varphi_2) > 0$. Moreover, $\mu_N$ is the unique QSD of $\{{X}^N\}$ that satisfies $\mu_N(\varphi_1) < \infty$ and $\mu_N(\varphi_2) > 0$. Additionally, $\mu_N(K) > 0$.
\end{theorem}
\begin{remark}\label{rem:sufftosuff}
	The above theorem combines two different results from \cite{chavil}. Proposition 3.1 of \cite{chavil} shows that under the assumptions of Theorem \ref{thm:foster}
	we have for some $c_2 \in (0,\infty)$, $n_1 \in \N$ and a probability measure $\nu$ supported on $K$
	$$
	\PP_x({X}_{n_1}^N \in \cdot) \geq c_2 \nu(\cdot \cap K),\; \mbox{ for all } x \in K.
	$$
	Also this proposition shows that for some $c_3 \in (0,\infty)$,
	$$
	\sup\limits_{n \in \N_0} \frac{\sup\limits_{y \in K}\PP_y(n < \tau_{\partial}^N)}{\inf\limits_{y\in K}\PP_y(n < \tau_{\partial}^N)} \leq c_3.
 	$$
	Using these facts, it then follows that, under the assumptions of Theorem \ref{thm:foster}, all the conditions of Theorem 2.1 in \cite{chavil} are satisfied, which gives the existence of QSD $\mu_N$ with the properties stated in the above theorem.
\end{remark}

 In Lemma \ref{lem:qsdcond} we use the above result to establish existence of a QSD for the sequence $\{X_n^N\}$ considered in this work, for each $N\in \N$.
 We begin with some preliminary estimates.

Consider for $r\in \N$
\begin{equation}K_r \doteq \{x \in \Delta^o : x \cdot 1 \leq r\},\;\; K_r^N \doteq K_r \cap \Delta^o_N,\label{eq:krdefn}\end{equation}
and let  
\begin{equation*}
	\sigma^N_{\partial} \doteq \inf\{k \in \N_0: \tilde X_k^N \in \partial \Delta_N\},
	\end{equation*}
$$\tau_r^N \doteq \inf\{k \in \N_0: X_k^N \in K_r^N\}, \;\; \sigma_{r}^N = \inf\{k: \tilde{X}^N_{k} \in K_r^N\},$$
and
$$\hat \tau_r^N \doteq \tau_r^N \wedge \tau_{\partial}^N, \;\; \hat{\sigma}_r^N \doteq \sigma^N_r \wedge \sigma_{\partial}^N.$$

\begin{lemma}\label{lem:expmoment}
Fix $\lambda_0 \in (0,\infty)$. There exists a $c(\la_0)\in (0,\infty)$ and $r_0 > 0$ such that for all $r \geq r_0$ and $\la \le \la_0$
$$
\EE_x\left( e^{\lambda \hat{\sigma}^N_r}\right) \le e^{x\cdot 1} c(\la_0) \text{ for all } x\in \Delta^o_N \mbox{ and } N\in \N.
$$
Furthermore, if $r \geq r_0$, then	
$$\EE_x\left(e^{\frac{\lambda}{N} \hat \tau^N_r}\right) \le e^{x\cdot 1} c(\la_0) \mbox{ for all } x \in \Delta_N^o \mbox{ and } N\in \N.$$
\end{lemma}

\begin{proof}
		Let $a= \max_i\|F_i\|_{\infty}$. Given $u \in N^{-1}\N$, consider the random variable $V_u$ that represents the number of particles among $Nu$ initial particles that die in $N$ steps when at each step any particle can die independently of the remaining particles with probability $1/N$.  Note that
		$V_u \sim \mbox{Bin}(Nu, \gamma(N))$ where 
		$$\gamma(N) = 1 - \left(1-\frac{1}{N}\right)^N.$$
		Let $U \sim \text{Poi}(Nad)$ be independent of $V_x$. 
		Then, under $\PP_x$, $(\tilde X^N_1 - x)\cdot 1 \le_d \frac{1}{N}(U-V_{x\cdot 1})$, where for two real random variables $Z_1, Z_2$, we write $Z_1 \le_d Z_2$ if $\PP(Z_2 \ge u) \ge \PP(Z_1 \ge u)$ for all $u \in \R$.
		Also,
		$$\EE_x[e^{\frac{1}{N}(U-V_{x\cdot 1})}] = C_N(1)e^{- V_0^N(1) x\cdot 1},$$
		where $C_N(1) = \EE(\exp\{\frac{1}{N}U\})$ and 
	$$\EE \exp\left\{-\frac{1}{N}V_{x\cdot 1}\right\} = e^{-V_0^N(1) x\cdot 1}.$$

	Note that for $x \in (K_r^N \cup \partial \Delta_N)^c$
	\begin{align*}
		\PP_x(\hat \sigma^N_r > 1) &\le \EE_x( e^{ \tilde X^N_1\cdot 1} 1_{\hat \sigma^N_r > 1}) = e^{ x\cdot 1} \EE_x( e^{ (\tilde X^N_1-x)\cdot 1} 1_{\hat \sigma^N_r > 1})\\
		&= e^{x\cdot 1} C_N(1) e^{-V_0^N(1) x \cdot 1} \le e^{ x\cdot 1} C_N(1) e^{-V_0^N(1) r}.
	\end{align*}
	By a recursive argument, for $n \in \N$
	\begin{equation}\label{eq:recurs}\PP_x(\hat \sigma^N_r > n) \le e^{x\cdot 1} e^{-n( r V_0^N(1) -  \log C_N(1))}.\end{equation}
	Note that
	$$\log C_N(1) = Nad (e^{1/N}-1).$$
	Also,
	$$\EE \exp\left\{-\frac{1}{N}V_{x\cdot 1}\right\} = \left[1- \gamma(N)(1-e^{-1/N})\right]^{N(x\cdot 1)},$$
	and thus
	$$V_0^N(1) = -N \log\left[1-\gamma(N)(1-e^{-1/N})\right].$$
	Combining the above observations
	$$\frac{\log C_N(1)}{V_0^N(1)} = \frac{Nad (e^{1/N}-1)}{-N \log\left[1-\gamma(N)(1-e^{-1/N})\right]}
	\le \frac{ad (e^{1/N}-1)}{\gamma(N)(1-e^{-1/N})} = ad \frac{e^{1/N}}{\gamma(N)}.$$
	Since $\gamma(N)\to (1-e^{-1})$ we can assume without loss of generality that for all $N\in \N$
	$$\frac{\log C_N(1)}{V_0^N(1)} \le \frac{2ade^2}{e-1} \doteq \vartheta, \; V_0^N(1) \ge \frac{1}{2}(1-e^{-1})\doteq \varsigma.$$
	Thus, for $r \ge r_0 \doteq (\frac{\lambda}{\varsigma} + \vartheta)$
	\begin{align*}
			e^{-n( r V_0^N(1) -  \log C_N(1))} &\le e^{-n  V_0^N(1) \left(\frac{\lambda}{\varsigma}+\vartheta - \frac{\log C_N(1)} {V_0^N(1)}\right)}
			\le e^{-n  V_0^N(1) \frac{\lambda}{\varsigma }} \le e^{-n\lambda }.
			\end{align*}
		
	Combining this with \eqref{eq:recurs}, for all $N\in \N$, $x \in (K_r^N \cup \partial)^c$ and $\lambda_0<\lambda$
	$$\EE_x\left(e^{\lambda_0 \hat \sigma^N_r}\right) \le e^{x\cdot 1} \frac{e^{\lambda_0} - e^{(\lambda_0-\lambda)}}{1-e^{(\lambda_0-\lambda)}}.$$
	This proves the first statement in the lemma.
	The second statement follows on noting that 
	$
	\hat{\tau}^N_r \leq N\hat{\sigma}^N_r,
	$
	for each $r \in \R_+$ and $N \in \N$.
\end{proof}

\begin{lemma}\label{lem:cfinite}
 Fix $\lambda_0 \in (0,\infty)$ and let $r_0$ be as in Lemma \ref{lem:expmoment}.
Then for each $\lambda \in (0, \lambda_0)$ and $r\ge r_0$,
$$
\sup\limits_{N\in \N} \sup\limits_{y\in K^N_r}\EE_y\left(\EE_{\tilde{X}^N_1}\left(e^{\lambda \hat{\sigma}^N_r}\right)1_{1 <\sigma_{\partial}^N}\right) < \infty.
$$
Furthermore, for every $N\in \NN$
$$
 \sup\limits_{y\in K^N_r}\EE_y\left(\EE_{{X}^N_1}\left(e^{\frac{\lambda}{N} \hat{\tau}^N_r}\right)1_{1 <\tau_{\partial}^N}\right) < \infty.
$$
\end{lemma}

\begin{proof}
	We only prove the first statement. The second statement is shown in a similar manner.
Fix $r \ge r_0$ and $\la <\la_0$. For notational simplicity, denote $K^N_r$ by $K$. Then, for $y \in K$,

\begin{equation*}
\begin{split}
\EE_y \left( \EE_{\tilde{X}^N_1} \left( e^{\lambda \hat{\sigma}_r^N}\right) 1_{1 < \sigma_{\partial}^N}\right) &= \EE_y \left( \EE_{\tilde{X}^N_1} \left( e^{\lambda \hat{\sigma}_r^N}\right) 1_{1 < \sigma^N_{\partial}}(1_{\tilde{X}^N_1 \in K} + 1_{\tilde{X}^N_1 \in K^c})\right)\\
&\leq 1 + \EE_y \left( \EE_{\tilde{X}^N_1} \left( e^{\lambda \hat{\sigma}_r^N}\right) 1_{1 < \sigma_{\partial}^N}1_{\tilde{X}^N_1 \in K^c})\right).
\end{split}
\end{equation*}
Let $a \doteq \max_i||F_i||_{\infty}$ and $U \sim \text{Poi}(Nad)$. Then with $c(\la_0)$ as in Lemma \ref{lem:expmoment}
we have
\begin{equation*}
\begin{split}
 \EE_y \left( \EE_{\tilde{X}_1^N} \left( e^{\lambda \hat{\sigma}_r^N}\right) 1_{1 < \sigma_{\partial}^N}1_{\tilde{X}^N_1 \in K^c})\right) &\leq \EE_y \left( c(\lambda_0)e^{\tilde{X}_1^N \cdot 1} 1_{1 < \sigma_{\partial}^N}1_{\tilde{X}_1^N \in K^c}\right)\\
&\leq c(\lambda_0) \EE_y\left( e^{y\cdot 1 + \frac{1}{N}U}\right)\\
&= e^{y\cdot 1} c(\lambda_0)  e^{dNa(e^{\frac{1}{N}}-1)}.
\end{split}
\end{equation*}
Since $\sup_{N\in \N} N(e^{1/N}-1) \le e$, the result follows.
\end{proof}
The following lemma will be used to verify condition (B2)(d) of Theorem \ref{thm:foster}.

\begin{lemma}\label{lem:uniftheta2}
There exists $r_1 \in (0,\infty)$ such that  $$\theta_2 \doteq \inf_{r\ge r_1}\inf\limits_{N\in \N} \inf\limits_{x\in K_r^N}\PP_x(\tilde{X}^N_1 \in K_r^N;\; {\sigma}^N_{\partial} > 1) > 0.$$
Furthermore, for each $N\in \NN$, there exists $r_1 \in (0,\infty)$ such that
$$\theta_2(N) \doteq \inf_{r\ge r_1} \inf\limits_{x\in K_r^N}\PP_x({X}^N_1 \in K_r^N;\; {\tau}^N_{\partial} > 1) > 0.$$
\end{lemma}

\begin{proof}
	Once again, we only prove the first statement.
Consider, for $z \in \N/N$, a collection of $Nz$ particles of a single type, where each particle, independently of all other particles, has a $\frac{1}{N}$ chance of dying at each time step. Then the probability that all $Nz$ particles are dead in  $N$ time steps is
$$
p(z,N) \doteq \left(1 - \left(1 - \frac{1}{N}\right)^N\right)^{Nz}.
$$
Note that for any $K \subset \Delta^o$,  $x \in K \cap \Delta_N$ and $N \geq 1$, $\min\limits_{1 \leq i \leq d}x_i  \geq \frac{1}{N}$, and so
$$
p(x\cdot 1,N) \leq p(dN^{-1},N) = \left[1 - \left( 1 - \frac{1}{N}\right)^N\right]^d.
$$
In particular,
\begin{equation*}
\begin{split}
\PP_x( \sigma_{\partial}^N\leq 1) = \PP_x(\tau_{\partial}^N \leq N) & \le  \left[1 - \left(1 - \frac{1}{N}\right)^{N}\right]^d,
\end{split}
\end{equation*}
and
so for any $K \subset \Delta^o$
$$
\sup\limits_{N> 1}\sup\limits_{x\in K}\PP_x(\tau_{\partial}^N \leq N) \leq (1 - e^{-2})^d \doteq \alpha_0.
$$
Thus, for $K \subset \Delta^o$,
\begin{equation*}
\begin{split}
\PP_x(\tilde{X}^N_1 \in K | \sigma^N_{\partial} > 1) &= 1 - \PP_x(\tilde{X}^N_1 \in K^c | \sigma^N_{\partial} > 1) = 1 -  \frac{\PP_x(\tilde{X}_1^N \in K^c;\; \sigma^N_{\partial} > 1)}{\PP_x(\sigma^N_{\partial} > 1)},
\end{split}
\end{equation*}
and
\begin{equation*}
\begin{split}
\sup\limits_{N > 1}\sup\limits_{x\in K}\frac{\PP_x(\tilde{X}^N_1 \in K^c;\; \sigma^N_{\partial} > 1)}{\PP_x(\sigma^N_{\partial} > 1)}
&\leq  \frac{\sup\limits_{N>1}\sup\limits_{x\in K}\PP_x(\tilde{X}^N_1 \in K^c;\; \sigma^N_{\partial} > 1)}{ 1-\alpha_0}.
\end{split}
\end{equation*}
We will now argue that for some $r_1\in (0,\infty)$ 
\begin{equation}\label{stricknr}
\sup_{r\ge r_1}\sup\limits_{N>1}\sup\limits_{x\in K^N_r}\PP_x(\tilde{X}^N_1 \in (K^N_r)^c;\; \sigma^N_{\partial} > 1) < 1-\alpha_0.
\end{equation}

Fix $r>0$ and let $x \in K^N_r$. As before,  let $a= \max_i\|F_i\|_{\infty}$. Fix $k \in \N$ and define
for $a_1 \in (0,\infty)$
$$m = m(N,k,a_1)\doteq \max\{1 \le j \le k : X^N_j \cdot 1 \le a_1\}.$$
Let $Y^N_k \doteq X^N_k\cdot 1$. Then
\begin{align*}
	Y^N_{k} &= Y^N_m + \frac{1}{N} \sum_{j=m+1}^k \eta^{N}_{j} (X^N_k)\cdot 1 \le_d a_1 + \frac{1}{N} \max_{\{1\le l \le k\}}\sum_{j=l}^k (U_j -V_j),
\end{align*}
where $U_j$ are iid $\mbox{Poi} (ad)$, $V_j$ are iid $\mbox{Bin}(Na_1, 1/N)$, and $\{U_j, V_j', j, j' \in \N\}$ are mutually independent. For $a_2>a_1$
\begin{align*}
	\PP_x(Y^N_{k}\ge a_2) \le \PP\left(\max_{\{1\le l \le k\}} \sum_{j=l}^k (U_j -V_j) \ge N(a_2-a_1)\right) \le \sum_{l=1}^k\PP\left(\sum_{j=l}^k (U_j -V_j) \ge N(a_2-a_1)\right)
\end{align*}
Thus for each $\gamma>0$, by Markov's inequality,
\begin{align*}
	\PP_x(Y^N_{k}\ge a_2) &\le e^{-\gamma N(a_2-a_1)}\sum_{l=1}^k [\EE e^{\gamma  U_1}]^{(k-l+1)} [\EE e^{-\gamma  V_1}]^{(k-l+1)}\\
	&= e^{-\gamma N(a_2 - a_1)} \frac{\EE e^{\gamma U_1}\EE e^{-\gamma V_1}\left(1 -  \left(\EE e^{\gamma U_1} \EE e^{-\gamma V_1}\right)^k \right)}{1 - \EE e^{\gamma U_1}\EE e^{-\gamma V_1} }.
\end{align*}

Note that for each $N \geq 1$, 
\begin{equation*}
\begin{split}
\left(\EE e^{\gamma  U_1} \EE e^{-\gamma  V_1}\right) &= e^{ad (e^{\gamma} -1)} \left(1 - \frac{1}{N} +  \frac{1}{Ne^{\gamma}}\right)^{Na_1} \leq e^{ad (e^{\gamma} - 1)} e^{-a_1(1-e^{-\gamma})}.
\end{split}
\end{equation*}
Let $r_1$ be large enough so that $e^{ad (e^{\gamma} - 1)} e^{r_1(e^{-\gamma} - 1)/2} < \frac{1}{2}$ and $r_1> -2\frac{\log (1-\alpha_0)}{\gamma}$. If we fix $r\ge r_1$ and let $a_2=r$ and $a_1= r/2$, then
\begin{equation*}
\begin{split}
\PP_x(Y^N_{k}\ge r)  &\le e^{-\gamma N\frac{r}{2}} \le e^{-\gamma N\frac{r_1}{2}} < (1-\alpha_0)^N \leq (1-\alpha_0).
\end{split}
\end{equation*}
This proves \eqref{stricknr} and hence 
$$\inf_{r\ge r_1}\inf\limits_{N>1}\inf\limits_{x\in K^N_r}
\PP_x(\tilde{X}^N_1 \in K^N_r \mid \sigma^N_{\partial} > 1) \doteq c_0 >0.$$
Finally, for all $N>1$, $r\ge r_1$, and $x \in K^N_r$
$$
\PP_x(\tilde{X}^N_1 \in K^N_r;\; \sigma^N_{\partial} > 1) =
\PP_x(\tilde{X}^N_1 \in K^N_r\mid \sigma^N_{\partial} > 1)\PP_x(\sigma^N_{\partial} > 1) \ge c_0(1-\alpha_0) >0.
$$
The result follows.
\end{proof}

Denote by $\clq^N$ the collection of all $\mu \in \clp(\Delta^o_N)$ such that for every $c\in (0,\infty)$, there exists a $r \in (0, \infty)$ such that
$\EE_{\mu}(e^{c\hat \sigma^N_r}) <\infty$.
The following result gives the existence of QSD for the chain $X^N$ for each $N$ and provides an important characterization of these QSD.

\begin{theorem}\label{lem:qsdcond}
	There is a probability measure $\mu_N$ on $\Delta^o_N$ such that for all $x_N \in \Delta^o_N$,
	$ \frac{\delta_{x_N} P_n^N}{\delta_{x_N} P_n^N(1_{\Delta_N^o})} \to \mu_N$ in the total variation distance.
For each $N \in \N$, the measure $\mu_N$  is a QSD for $\{{X}_n^N\}$. It is the unique QSD for $\{{X}_n^N\}$ that 
belongs to $\clq^N$. 
\end{theorem}
\begin{proof}
Fix $N \in \N$ and let $r_1 \in (0,\infty)$ and $\theta_2 \in (0, 1]$ be as in the second statement in Lemma \ref{lem:uniftheta2}. Fix $r_2\ge r_1$, let
%
%
$K= K_{r_2}^N$ and define $\varphi_2 : \Delta^o_N \rightarrow \R_+$ by $\varphi_2(x) \doteq 1_{K}(x)$.

Fix an arbitrary $\theta_1 \in (0, \theta_2)$. From Lemma \ref{lem:expmoment} there is a $r_3 > r_2$ such that for any fixed $r\ge r_3$
\begin{equation}\label{eq:eqphi1x}
\varphi_1(x) \doteq \EE_x\left( \theta_1^{ - \hat{\tau}^N_{r}}\right) < \infty \text{ for all } x \in \Delta^o_N.
\end{equation}
We now verify the conditions of Theorem \ref{thm:foster} with the above choice of $K, \varphi_1, \varphi_2, \theta_1$ and $\theta_2$.
It is clear that condition (B1)  is satisfied with $n_2(x)=1$.
Also, (B2)(b) is satisfied, since $\varphi_2(x) = 1$ for each $x \in K$.
Since $\theta_1 \in (0,1)$,  $\inf\limits_{x\in \Delta^o_N} \varphi_1(x) \geq 1$. Also, since $K \subset K_{r}^N$, $\sup\limits_{x\in K}\varphi_1(x) = 1$, and so (B2)(a) holds.
 Next, an application of Lemma \ref{lem:cfinite} and Markov property show that (B2)(c) holds with
\begin{equation*}
\begin{split}
c_2 \doteq \sup\limits_{y \in K} \EE_y\left( \varphi_1({X}_1^N)1_{\{\sigma_{\partial}^N>1\}} \right).
\end{split}
\end{equation*}
Finally the validity of (B2)(d) follows from Lemma \ref{lem:uniftheta2}.

  Also, since
$$\inf_{x,y \in K} \PP_y({X}_1^N=x) \doteq \kappa_1>0,$$
the inequality in \eqref{eqn:prop3.1} is satisfied with $C= \kappa_1^{-1}$.  
Thus, from Theorem \ref{thm:foster} it follows that there exists a QSD $\mu_N$ for $\{{X}_n^N\}$ that satisfies 
\begin{equation}\label{eq:355}
	\EE_{\mu_N}(\theta_1^{ - \hat{\tau}^N_{r}}) <\infty, \mbox{ and } \mu_N = \lim_{n\to \infty}  \frac{\delta_{x_N} P_n^N}{\delta_{x_N} P_n^N(1_{\Delta_N^o})}, \; \mbox{ for any } x_N \in K.
\end{equation}
We now show that $\mu_N \in \clq^N$.  
Fix $c \in (0, \infty)$.  Let $\varphi_2$, $\theta_2$ and $K$ be as above. Choose $\theta_1^* \in (0, \theta_2 \wedge e^{-c})$. From the second statement in Lemma \ref{lem:expmoment} 
there exists a $r_4> r_3$ such that 
$$
\tilde\varphi_1(x) \doteq \EE_x\left( (\theta_1^*)^{ - \hat{\tau}^N_{r_4}}\right) < \infty \text{ for all }x \in \Delta^o_N.$$
Then from the previous argument, there is a QSD $\tilde \mu_N$ for $\{{X}_n^N\}$ such that 
$$\EE_{\tilde \mu_N}(e^{ c \hat{\tau}^N_{r_4}}) \le \EE_{\tilde \mu_N}((\theta_1^*)^{ - \hat{\tau}^N_{r_4}}) <\infty.$$
and
$$\tilde \mu_N = \lim_{n\to \infty}  \frac{\delta_{x_N} P_n^N}{\delta_{x_N} P_n^N(1_{\Delta_N^o})}, \mbox{ for any } x_N \in K.$$
From \eqref{eq:355} we now see that  $\mu_N= \tilde \mu_N$ and that $\EE_{ \mu_N}(e^{ c \hat{\tau}^N_{r_4}}) <\infty$. Since $c>0$ is arbitrary, it follows that $\mu_N \in \clq^N$.
Also, since $r_2\ge r_1$ is arbitrary, we see (by choosing a larger $K$ if needed) that the convergence in \eqref{eq:355} holds for all $x_N \in \Delta^o_N$.

Finally we argue uniqueness. Let $\tilde \mu_N \in \clq^N$ be a QSD for $\{{X}_n^N\}$. 
Choose $r_5\ge r_1$ such that $\tilde \mu_N(K^N_{r_5}) >0$.
Consider $\tilde K = K^N_{r_5}$ and $\tilde \varphi_2 = 1_{\tilde K}$.
Fix $ \theta_1 \in (0,  \theta_2)$ and let $r> r_5$ be such that
$$\EE_{\tilde \mu_N}(( \theta_1)^{- \hat{\tau}^N_{r}}) < \infty, \mbox{ and } \EE_{x}(( \theta_1)^{- \hat{\tau}^N_{r}}) < \infty
\mbox{ for all } x \in \Delta_N^o.$$
Then by the previous argument (and Theorem \ref{thm:foster}) 
$$\EE_{\mu_N}(( \theta_1)^{- \hat{\tau}^N_{r}}) < \infty \mbox{ and } \mu_N(\tilde K) >0.$$
But since the above two properties are also satisfied by $\tilde \mu_N$, from Theorem \ref{thm:foster} we must have $\mu_N= \tilde \mu_N$.
\end{proof}

\subsection{Tightness of Quasi-Stationary Distributions}\label{sec:verassu6}

We now prove the tightness of the sequence of QSD $\{\mu_N\}$ given in Theorem \ref{lem:qsdcond}.

\begin{theorem}\label{thm:tight}
Let for $N\in \NN$, $\mu_N$ be as given in Lemma \ref{lem:qsdcond}. Then, the sequence $\{\mu_N\}$ is tight.
\end{theorem}

\begin{proof}
	Recall the definition of $P_n^N$ from \eqref{eq:pnnf} and let $\tilde P_n^N \doteq P^N_{nN}$.
From Lemma \ref{lem:qsdcond}, 
for all $x_N \in \Delta^o_N$
$$\lim_{n\to \infty}\frac{\delta_{x_N} \tilde P_{n}^N}{\delta_{x_N} \tilde P_{n}^N (1_{\Delta^o})} = \mu_N.$$

Thus in order to show that the sequence $\{\mu_N\}$ is tight it suffices to show that the collection
\begin{equation}\left\{\frac{\delta_{x_N} \tilde P_{n}^N}{\delta_{x_N} \tilde P_{n}^N (1_{\Delta^o})}, n,N \in \N\right\}\label{eq:tightnm}\end{equation}
is tight for some sequence $\{x_N\}$, where $x_N \in \Delta^o_N$ for each $N$.
For this it suffices to show that for every $\veps>0$, there is a $L_1\in (0,\infty)$ such that
$$\sup\limits_{N\in \N}\sup_{n\in \N} \PP_{x_N}(\tilde{X}_n^N \cdot 1 \ge L_1 \mid \sigma_{\partial}^N > n) \le \veps.$$

From Lemma \ref{lem:uniftheta2}, for all $r\ge r_1$
$$
\theta_2^r \doteq \inf\limits_{N\geq 1}\inf\limits_{x\in K_r} \PP_x(\tilde{X}_1^N \in K_r ; \sigma^N_{\partial} > N) \ge \theta_2 > 0,
$$
so for every $r\ge r_1$, with $\varphi_2^r(x) = \varphi_2(x) \doteq  1_{K_r}(x)$, for each $N \in \N$,
$$P_1^N\varphi_2^r(x) \ge \theta_2\varphi_2^r(x) \mbox{ for all } x \in \Delta^o_N.$$

Recall that  $a = \max_{1\leq i \leq d}\|F_i\|_{\infty} <\infty$ and $\tilde X^N_k = X^N_{Nk}$ for $k \in \N$.
We now consider a coupling between the sequence of $d$-dimensional random variables $\{X^n_k\}$ and a sequence $\{Z^N_k\}$ of $\N/N$-valued random variables that preserves certain monotonicity properties.
Note that $\{X^n_k\}$ can be constructed as follows.
Consider a collection of iid random fields
$\{(U^N_k(x), V^N_k(x)), x \in \Delta^o_N\}_{k\in \N}$ where $U^N_k(x)$ is a $d$-dimensional random variable with mutually independent coordinates distributed as Poisson random variables with means
$F^N_i(x)$, $i \in \{1, \ldots, d\}$, and $V^N_k(x)$ is a $d$-dimensional random variable, independent of $U^N_k(x)$, of mutually independent Binomial random variables with parameters $(N x_i, 1/N)$, $i \in \{ 1, \ldots, d\}$.
Then
\begin{equation}\label{eq:bmc2}
\begin{aligned}
	X_{k+1}^N &= X_k^N + \frac{1}{N} (U^N_{k+1}(X_k^N) - V^N_k(X_k^N)), \; k \in \N_0,
	X_0^N &= x_N.
\end{aligned}
\end{equation}
gives a construction for the Markov chain $\{X^n_k\}$. 
Then we can construct, along with the above iid random fields, iid fields
$\{(A^N_k(z), B^N_k(z)); z \in \N/N\}_{k\in \N}$ such that
$$A^N_k(x\cdot 1) \sim \mbox{Poi}(ad - F(x)\cdot 1), \; \mbox{ and }  D^N_k(x) \doteq A^N_k(x\cdot 1) + U^N_k(x)\cdot 1 \sim \mbox{Poi}(ad), \mbox{ for all } x \in \Delta_N^o$$
and
$$B^N_k(z) \sim \mbox{Bin}(Nz, 1/N) \mbox{ and whenever } z \ge x\cdot 1, (B^N_k(z)- V^n_k(x)\cdot 1)\le z - x\cdot 1, \mbox{ for } x \in \Delta_N^o \mbox{ and } z \in \N/N.$$
Define, for $z_N \in \N/N$ with $z_N \ge x_N\cdot 1$,
$$Z^N_{k+1} = Z^N_k + \frac{1}{N} [D^N_k(X^N_k) - B^N_k(Z^N_k)], \; Z^N_0 = z_N.$$
The sequence $Z^N_k$ describes the evolution of the (scaled) population size of a single-type population in which at each time step any particle can die with probability $1/N$ independently of other particles, and 
$\mbox{Poi}(ad)$ new particles are born. Let $Y^N_k \doteq X^N_k \cdot 1$. Then, by construction, $Z^N_k \ge Y^N_k$ for all $k,N$.

Fix $r\ge r_1$ and let $x_N$ be in $K_r \cap \Delta_N$ for each $N$. Also, let $z_N = x_N\cdot 1$.
In order to prove the tightness of the collection in \eqref{eq:tightnm}
 it suffices  to show that for every $\veps>0$, there is a $L_1\in (0,\infty)$ such that
$$\sup\limits_{N\in \NN}\sup_{n\in \NN} \PP_{x_N}(\tilde{X}_n^N \cdot 1 \ge L_1 \mid \sigma_{\partial}^N > n) \le \veps.$$

 Let $\tilde Z_n^N \doteq Z^N_{nN}$ for $n \in \NN_0$, $N\in \NN$, and define
 $$\sigma_r^{N,Z} \doteq \inf\{n\in \NN_0: \tilde Z_n^N \le r\}, \;\; \sigma^{Z,N}_{\partial} \doteq \inf\{n: \tilde Z_n^N=0\}.$$ 
Using similar arguments as in the proofs of Lemmas \ref{lem:expmoment}  and \ref{lem:cfinite} we can assume without loss of generality that $r$ is large enough so that 
there is a $\theta_1 \in (0,\theta_2)$  such that for
$$\varphi_1^N(z) \doteq \EE_z\left(\theta_1^{-(\sigma_r^{N,Z} \wedge \sigma^{Z,N}_{\partial})}\right), \; z \in \NN/N$$
and
$$
C \doteq \sup\limits_{N \geq 1} \sup\limits_{y \in \NN/N, y \le r}\EE_y \left( \EE_{\tilde Z^N_1} \left( \theta_1^{-(\sigma_r^{N,Z} \wedge \sigma_{\partial}^{Z,N})}1_{1 < \sigma_{\partial}^{Z,N}}\right)\right),
$$
we have $C<\infty$ and
$$\EE_z\left(\varphi^N_1(\tilde Z^N_1)1_{\sigma^{Z,N}_{\partial}>1}\right) \le \theta_1 \varphi_1^N(z) + C 1_B(z), z \in \NN/N, N \in \NN.$$
For fixed $L < \infty$, there is a $L_1 \in (r_0, \infty)$ such that for all $z \ge L_1$, we have $\varphi_1^N(z) \ge L$ for all $N \in \NN$.
Then, with $\varphi_2= \varphi_2^{r_0}$


\begin{align*}
\PP_{x_N}(\tilde{X}_{n}^N \cdot 1 \ge L_1 \mid \sigma_{\partial}^N > n) &\le \PP_{z^N}(\tilde Z_n^N \ge L_1 \mid \sigma_{\partial}^N > n) \le \PP(\varphi_1^N(\tilde Z_n^N) \ge L \mid \sigma_{\partial}^N > n)\\
&\le L^{-1} \EE(\varphi_1^N(\tilde Z_n^N)\mid \sigma_{\partial}^N > n) = L^{-1} \frac{\EE\left(\varphi_1^N(\tilde Z_n^N)1_{\sigma_{\partial}^N > n}\right)}{\PP(\sigma_{\partial}^N > n)}\\
&\le L^{-1} \frac{\EE\left(\varphi_1^N(\tilde Z_n^N)1_{\sigma_{\partial}^N > n}\right)}{\EE\left(\varphi_2(\tilde{X}_{n}^N)1_{\sigma_{\partial}^N > n}\right)},
\end{align*}
where the last inequality uses the property $\varphi_2 \le 1$.
Also
\begin{align*}
\EE_{x_N}\left(\varphi_2(\tilde{X}_{n}^N)1_{\sigma_{\partial}^N > n)}\right) &\ge \theta_2 \EE_{x_N}\left(\varphi_2(\tilde{X}_{n-1}^N)1_{\sigma_{\partial}^N > n-1)}\right) 
= \theta_2 \EE_{x_N}\left(1_{[0,r]}(\tilde{X}_{n-1}^N\cdot 1)1_{\sigma_{\partial}^N > n-1}\right),
\end{align*}
and, with $\clf_{n} = \sigma \{\tilde X_k, \tilde Z_k, k \le n\}$,
\begin{align*}
\EE_{x_N}\left(\varphi_1^N(\tilde Z_n^N)1_{\sigma_{\partial}^N > n}\right) &= \EE_{x_N}\left(\EE(\varphi_1^N(\tilde Z_n^N)1_{\sigma_{\partial}^N > n}1_{\sigma_{\partial}^N > n-1} \mid \clf_{n-1})\right)\\
&= \EE_{x_N}\left(\EE_{x_N}(\varphi_1^N(\tilde Z_n^N)1_{\sigma_{\partial}^N > n}\mid \clf_{n-1}) 1_{\sigma_{\partial}^N > n-1} \right)\\
&\le \EE_{x_N}\left(\EE_{x_N}(\varphi_1^N(\tilde Z_n^N)1_{\sigma_{\partial}^{Z,N} > n}\mid \clf_{n-1}) 1_{\sigma_{\partial}^N > n-1} \right)\\
&\le \theta_1 \EE_{x_N}\left(\varphi_1^N(\tilde Z_{n-1}^N) 1_{\sigma_{\partial}^N > n-1} \right) + C \EE_{x_N}\left(1_{[0,r]}(\tilde Z_{n-1}^N)1_{\sigma_{\partial}^N > n-1}\right)\\
&\le \theta_1 \EE_{x_N}\left(\varphi_1^N(\tilde Z_{n-1}^N) 1_{\sigma_{\partial}^N > n-1} \right) + C \EE_{x_N}\left(1_{[0,r]}(\tilde X_{(n-1)}^N\cdot 1)1_{\sigma_{\partial}^N > n-1}\right)
\end{align*}
Thus
\begin{align*}
\frac{\EE_{x_N}\left(\varphi_1^N(Z_n^N)1_{\sigma_{\partial}^N > n}\right)}{\EE_{x_N}\left(\varphi_2(\tilde{X}_{n}^N)1_{\sigma_{\partial}^N > n}\right)} &\le \frac{\theta_1}{\theta_2} \frac{\EE_{x_N}\left(\varphi_1^N(Z_{n-1}^N) 1_{\sigma_{\partial}^N > n-1} \right)}{\EE_{x_N}\left(\varphi_2(\tilde{X}_{n-1}^N) 1_{\sigma_{\partial}^N > n-1} \right)} + \frac{C}{\theta_2}.
\end{align*}
Iterating this inequality
\begin{align*}
\frac{\EE_{x_N}(\varphi_1^N(Z_n^N)1_{\sigma_{\partial}^N > n)}}{\EE_{x_N}(\varphi_2(\tilde{X}_{n}^N)1_{\sigma_{\partial}^N > n)})} &\le \left(\frac{\theta_1}{\theta_2}\right)^n \frac{\varphi_1^N(z_N)}{\varphi_2(x_N)} +
\frac{C}{\theta_2} \frac{1}{1- (\theta_1/\theta_2)}.
\end{align*}

Since $x_N \in K_{r}$ for each $N$,
$$\PP_{x_N}(\tilde{X}_{n}^N \cdot 1 \ge L_1 \mid \sigma_{\partial}^N > n ) \le L^{-1}\left[1+\frac{C}{\theta_2-\theta_1} \right].$$

Tightness follows.

\end{proof}

\subsection{Completing the Proof of Theorem \ref{thm:poisbin}}\label{sec:verassu7}
We can now complete the proof of Theorem \ref{thm:poisbin}. We will apply Theorem \ref{thm:main}. 
From Sections \ref{sec:verassu1}, \ref{sec:verassu2}, \ref{sec:verassu3}, \ref{sec:verassu4} it follows that Assumptions \ref{assu:LLN}, \ref{assu:ap-classesfinite}, \ref{assu:mgf} and \ref{assu:irrbdr} are satisfied.
From Section \ref{sec:verassu5} it follows that there is a $\mu_N \in \clp(\Delta^o_N)$ such that for every $N \in \N$, and $x_N \in \Delta^o_N$, 
$$ \frac{\delta_{x_N} P_n^N}{\delta_{x_N} P_n^N(1_{\Delta_N^o})}$$ converges to $\mu_N$ in the total variation distance as $n \rightarrow \infty$.
Furthermore, the measure $\mu_N$ is a QSD for $\{X^N\}$.
From Section \ref{sec:verassu5} the sequence $\{\mu_N\}_{N\in \NN}$ is relatively compact 
 as a sequence of probability measures on $\Delta$.  Theorem \ref{thm:poisbin} is now immediate from Theorem \ref{thm:main}.
 \hfill \qed

\vspace{\baselineskip}\noindent \textbf{Acknowledgment:} 
The research of AB was supported in part by the NSF (DMS-1814894, DMS-1853968).

\bibliographystyle{amsplain}
\bibliography{main}

\vspace{\baselineskip}
\noindent
\scriptsize{\textsc{\noindent A. Budhiraja, N. Fraiman, A. Waterbury\newline
Department of Statistics and Operations Research\newline
University of North Carolina\newline
Chapel Hill, NC 27599, USA\newline
email: budhiraj@email.unc.edu, fraiman@email.unc.edu, atw02@live.unc.edu\vspace{\baselineskip} }

%
}

\end{document}